\documentclass[a4paper,10pt]{article}
\usepackage[utf8]{inputenc}
\usepackage[english]{babel}
\usepackage[T1]{fontenc}
\usepackage{hyperref}
\usepackage{amsfonts}
\usepackage{amsmath}
\usepackage{amssymb}
\usepackage{amsthm}
\usepackage{enumitem}

\usepackage{bm}
\usepackage{dsfont}
\usepackage{esint}

\usepackage[pdftex]{graphicx}
\usepackage{graphics}
\usepackage{float}
\usepackage{color}
\usepackage{tikz}

\usepackage{cancel}

\usepackage[sort,nocompress]{cite}

\usepackage{lscape}
\usepackage{scrextend}

\usepackage{color}

\newcommand\blfootnote[1]{
\begingroup
\renewcommand\thefootnote{}\footnote{#1}
\addtocounter{footnote}{-1}
\endgroup
}
\renewcommand{\star}{*}

\newcommand{\R}{\mathbb{R}}

\newcommand{\Z}{\mathbb{Z}} 

\newcommand{\NN}{\mathbb{N}}

\newcommand{\LL}{L}
\newcommand{\CC}{C}
\newcommand{\HH}{H}

\newcommand{\lt}{\left}
\newcommand{\rt}{\right}
\def\[{[\![}
\def\]{]\!]}

\newcommand{\dd}{{\rm{d}}}

\newcommand{\Supp}{{\rm{Supp}}}

\newcommand{\Id}{{\rm{I}}}

\renewcommand{\div}{{\rm div}}

\newcommand{\FF}{\mathcal{F}}

\newcommand{\loc}{ {loc} }
\newcommand{\comp}{c}
\newcommand{\per}{ {per }}

\newcommand{\Boule}{{ B}}
\newcommand{\Q}{{ Q}}

\def\Xint#1{\mathchoice
  {\XXint\displaystyle\textstyle{#1}}%
  {\XXint\textstyle\scriptstyle{#1}}%
  {\XXint\scriptstyle\scriptscriptstyle{#1}}%
  {\XXint\scriptscriptstyle\scriptscriptstyle{#1}}%
  \!\int}
\def\XXint#1#2#3{{\setbox0=\hbox{$#1{#2#3}{\int}$}
    \vcenter{\hbox{$#2#3$}}\kern-.5\wd0}}

\def\fint{\Xint-}

\renewcommand{\epsilon}{\varepsilon}
\renewcommand{\tilde}{\widetilde}

\newcommand{\langl}{\lt\langle}
\newcommand{\rangl}{\rt\rangle}

\newtheorem{theorem}{Theorem}

\newtheorem{corollary}{Corollary}
\newtheorem{lemma}{Lemma}
\newtheorem{lemma_appendix}{Lemma}[section]

\newtheorem{proposition}{Proposition}

\newtheorem*{proposition*}{Proposition}
\newtheorem{definition}{Definition}
\newtheorem*{theorem*}{Theorem}
\newtheorem*{lemma*}{Lemma}
\newtheorem*{corollary*}{Corollary}
\newtheoremstyle{TheoremNum}
{\topsep}{\topsep}
{\itshape}
{}
{\bfseries}
{.}
{ }
{\thmname{#1}\thmnote{ \bfseries #3}}
\theoremstyle{TheoremNum}

\theoremstyle{remark}
\newtheorem{remark}{Remark}
\newtheorem{Example}{Example}

\title{Quantitative homogenization for the case of an interface between two heterogeneous media}
\author{Marc Josien\footnote{Max-Planck-Institut f\"ur Mathematik in den Naturwissenschaften, Inselstrasse 22, 04103 Leipzig, Germany} \quad \& \quad Claudia Raithel\footnote{Technische Universit\"at Wien, Karlsplatz 13, 1040 Wien, Austria} 
\blfootnote{The authors gratefully acknowledge financial support from the Max Planck Institute for Mathematics in the Sciences. 
}
\blfootnote{The second author gratefully acknowledges partial support from the Austrian Science Fund (FWF), grants P30000, W1245, and F65.
}
}
\begin{document}

\maketitle
\begin{abstract}
In this article we are interested in quantitative homogenization results for linear elliptic equations in the non-stationary situation of a straight interface between two heterogenous media. This extends the previous work \cite{Josien_InterfPer_2018} to a substantially more general setting, in which the surrounding heterogeneous media may be periodic or random stationary and ergodic.
Our main result is a quantification of the sublinearity of a homogenization corrector adapted to the interface, which we construct using an improved version of the method developed in \cite{Fischer_Raithel_2017}. 
This quantification is optimal up to a logarithmic loss and allows to derive almost-optimal convergence rates.
\end{abstract}

\paragraph{Keywords:} homogenization, interfaces, correctors, Lipschitz estimates, convergence rate\\

\newpage
\tableofcontents

\newpage

\section{Introduction}

In this article we construct and estimate the growth rate of homogenization correctors associated to linear elliptic operators in divergence form in the context of a flat interface between two heterogenous media (see, \textit{e.g.}, Figure \ref{Fig_Examples}). It is a continuation of the previous work of the first author \cite{Josien_InterfPer_2018}, inspired by \cite{BLLcpde}, which studies the case of an interface between two periodic media. We refer the reader to \cite{Josien_InterfPer_2018}, which is more elementary than the present study. There, definitions for the homogenization correctors and $2$-scale expansion adapted to the interface are designed, motivated and proved to produce an accurate approximation of the solution of the multiscale problem.
Equipped with these algebraic definitions, we explore here a substantially broader framework, in which we do not assume any structure on the two surrounding heterogeneous media, but only that each of them admits a constant homogenized matrix and correctors with a controlled growth rate; such a framework could be applied for periodic, almost-periodic, or stochastic homogenization. Under these assumptions, we build adapted correctors satisfying suboptimal sublinearity estimates by taking advantage of the techniques developed in \cite{Fischer_Raithel_2017} by Fischer and the second author. In our main theorem, we use Green's function estimates to obtain an almost-optimal control of the growth rate of the correctors.  

\begin{figure}[h]
\begin{center}
\begin{minipage}{0.48\linewidth}
\includegraphics[width=\textwidth]{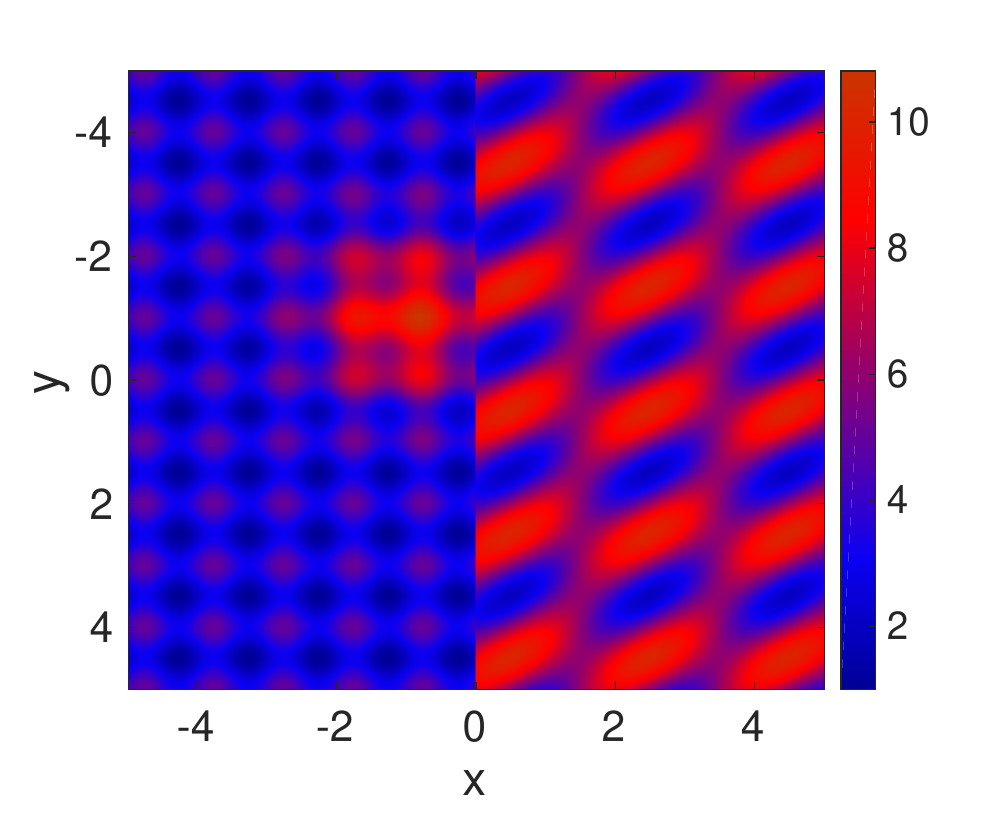}
\end{minipage}
\begin{minipage}{0.48\linewidth}
\includegraphics[width=\textwidth]{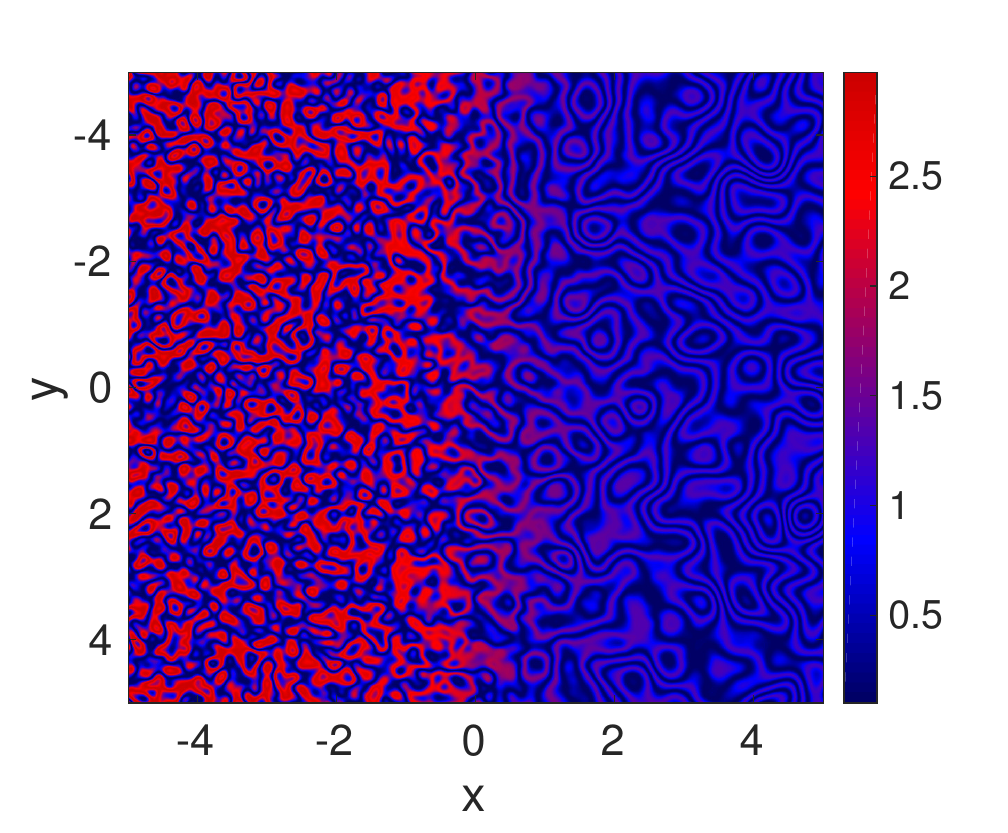}
\end{minipage}
\end{center}  
\caption{On the left, a sharp interface between two periodic media with a defect; on the right, a smooth interface between two random media generated from independent Gaussian fields. The colors indicate the value of $a$ (which is here assumed to be scalar).}
\label{Fig_Examples}
\end{figure}

\subsection{Motivation and related works}

We now motivate our study and discuss related results from the literature.

\paragraph{General theory of homogenization}
Consider a linear elliptic equation in divergence form:
\begin{equation}\label{DivAgrad0}
- \div \lt(a(x) \nabla u(x)\rt) = f(x).
\end{equation} 
Such equations play a central role in many branches of material physics; \textit{\textit{e.g.}} in elasticity, electrostatics, and thermostatics.
We refer to \cite{Allaire} for a didactic introduction to homogenization and its applications.
The coefficient field $a$ typically represents local characteristics of a sample: elasticity, electrical conductance, or thermal conductance (depending on the context). Here, as in the classical theory of homogenization, the coefficient $a$ is assumed to be varying at a characteristic small scale, which here is of order $1$ (by a change of variables). When this small scale vanishes (or equivalently, on infinitely large scales), equation \eqref{DivAgrad0} may be approximated by the following homogenized equation:
\begin{equation}\label{DivAgradbar0}
- \div \lt(\bar{a}(x) \nabla \bar{u}(x)\rt) = f(x),
\end{equation} 
where the so-called \textit{homogenized matrix} $\bar{a}$ is usually simpler that the original matrix $a$. 

In most works, the coefficient $a$ is assumed to have stationarity properties. Roughly speaking, the behavior of the medium is shift-invariant; \textit{e.g.} $a$ might be periodic \cite{Allaire}, almost-periodic \cite{Armstrong_Shen_2016}, or random stationary and ergodic \cite{JKO, Gloria_Neukamm_Otto_2015,Armstrong_book_2018}.
In those cases, the homogenized matrix $\bar{a}$ is constant. This is, in particular, shows that homogenization is an efficient tool for approximating \eqref{DivAgrad0}, which would be very costly  to solve numerically. Nevertheless, even though stationarity --in all its aforementioned expressions-- is a convenient mathematical tool, it may not always be a realistic hypothesis.

\paragraph{Beyond stationary coefficient fields}
Quite recently, in \cite{BLLcpde}, there was a deliberate attempt to study theoretically more general structures.
More precisely, two cases were proposed: The case of a defect in a periodic structure and the case of an interface between two periodic media.

The first case breaks stationarity, but only on the microscopic level, for the defect has no macroscopic impact (at least at the main order).
Thus, once the corrector is built and estimated \cite{BLLcpde,NewBLL1}, classical approaches in periodic homogenization (namely Avellaneda and Lin's \cite{AvellanedaLin}, later improved in \cite{KLSGreenNeumann}) are sufficient to obtain accurate convergence rates \cite{CRAS_Homog,Article_Homog}.

The second case not only breaks stationarity at the microscopic level, but also at the macroscopic level.
Indeed, the interface plays a role at any scale: The homogenized matrix $\bar{a}$ is generically piecewise constant with a discontinuity through the interface. Note that the book \cite[Chap.\ 9 p.\ 312]{Bakhvalov_Panasenko}, which predates \cite{BLLcpde,Josien_InterfPer_2018} and inspired \cite{Vinoles}, proposes another point of view on interfaces, with slightly different --however consistent-- definitions for the correctors and asymptotic expansion than we give below.

\paragraph{The case of an interface between periodic media}
The case of an interface requires adapted definitions of correctors  \cite{Josien_InterfPer_2018}.
These correctors $\phi_j$, for $j \in \[1,d\]$, are strictly sublinear (see \eqref{StricSub}) solutions to the following equation:
\begin{align}
\label{DefCorr*}
-\div \lt( a(x)  \nabla \lt( P_j(x) + \phi_j(x) \rt)\rt)=0 \quad\text{in}\quad \R^d,
\end{align}
where the piecewise affine functions $P_j$ span the space of non-constant and strictly subquadratic $\overline{a}$-harmonic functions\footnote{By the classical Liouville principle for piecewise constant coefficient fields, this space has dimension $d$.}. Namely, the functions $P_j$ solve      
\begin{equation}\label{Equat_P}
-\div\lt( \overline{a}(x) \nabla P_j(x)\rt)=0 \quad\text{in}\quad \R^d.
\end{equation}
In \cite{Josien_InterfPer_2018}, in a specific case of periodic media,  these correctors were actually built and an almost-optimal convergence rate for the gradient of the adapted $2$-scale expansion\footnote{Interestingly, such an expansion in only required when considering the gradient in the vicinity of the interface, which may be relevant in elasticity in the context of fractures. See Figure \ref{FigA}.},
\begin{align*}
\widetilde{u}:=\overline{u} + \phi \cdot \lt( \nabla P\rt)^{-1} \nabla \overline{u},
\end{align*}
was obtained. The techniques of \cite{Josien_InterfPer_2018}, however, were crucially based on some periodic structures of the underlying heterogeneous media.

\paragraph{The case of general interface}
In the current contribution we consider a more general case of flat interface between two media. 
We do not assume any joint structure on them, but only that each of them admits a constant homogenized matrix and correctors with a controlled growth rate.
Our main result is that the global medium-- which consists of the two heterogeneous  parts glued along the interface-- enjoys the \textit{same} quantitative homogenization properties as the two components, up to a logarithmic loss.
In particular, the growth rate of the global corrector is essentially bounded by the maximum of the growth rates of the correctors associated with each of the heterogeneous media (see Theorem \ref{PropRefinedDeter} below).

To obtain this result, we first rely on the approach of \cite{Fischer_Raithel_2017,Raithel_2017}, which construct correctors for the half-space with homogeneous Dirichlet or Neumann boundary conditions. These articles provide a robust way to build correctors for simple geometries, but with a suboptimal growth rate. Other than the existence of correctors on the whole space satisfying a weak quantified sublinearity condition,  there are no other structural assumptions made on the coefficient fields. Here, capitalizing on estimates for the heterogeneous Green's function provided by large-scale Lipschitz regularity, we prove an almost optimal growth rate. We remark that the strategy for proving the large-scale Lipschitz regularity is to transfer large-scale regularity properties from the homogenized to the heterogeneous problem --here we adapt the strategy of \cite{Gloria_Neukamm_Otto_2015}. However, since now the homogenized problem involves a piecewise continuous coefficient, we make use of the results of \cite{Vogelius_2000,LiNirenberg_2003}.

Last, as is classical in homogenization (see, \textit{e.g.} \cite{Gloria_Neukamm_Otto_2015} or the introductory course \cite{JosienOtto_2019}), our estimate for the growth rate of the correctors produces, in turn, a convergence rate on the level of the adapted $2$-scale expansion.

\subsection{Precise mathematical setting}\label{ensemble}

In this section we fix the model for a flat interface between two heterogeneous media that we will consider throughout this paper.

\paragraph{General notations}
Let $d$ be the dimension and $\lt(e_i\rt)_{i \in \[1,d\]}$ be the canonical basis of $\R^d$. In this paper we always assume that $d \geq 2$.
If $x\in \R^d$, we define
\begin{equation*}
x^\perp:=x\cdot e_1 \in \R \quad\text{and}\quad x^\parallel:=\lt(x \cdot e_2,\cdots, x \cdot e_d\rt) \in \R^{d-1},
\end{equation*}
so that $x=(x^\perp,x^\parallel)$.
If $R>0$, we denote by $\Q(x,R) \subset \R^d$ the cube of side length $R$ centered at $x$; also $\Boule(x,R) \subset \R^d$ is the ball of radius $R$ centered at $x$. If $x=0$, it might be omitted.

We highlight that throughout this paper we make use of the Einstein summation convention.

We say that a function $f$ is \textit{sublinear} if it satisfies the following condition:
\begin{align}\label{StricSub}
&\limsup_{r \uparrow \infty} \frac{1}{r} \lt( \fint_{\Boule(0,r)} \lt| f  -  \fint_{\Boule(0,r)}f  \rt|^2 \rt)^{\frac{1}{2}}< \infty.
\end{align}
It is said to be \textit{strictly sublinear} if the above limit is equal to $0$.

\paragraph{Definition of the interface}
We define a coefficient field $a$ by
\begin{align}
\label{Defa}
a(x)=
\lt\{
\begin{aligned}
&a_-(x) \quad\text{if}\quad x^\perp < -1,
\\
&a_\circ(x) \quad\text{if}\quad -1<x^\perp<1,
\\
&a_+(x) \quad\text{if}\quad x^\perp > 1.
\end{aligned}
\rt.
\end{align}
The interface is defined by $\mathcal{I}:=\{0\} \times \R^{d-1}$.
In our model, the thin layer $[-1,1] \times \R^{d-1}$ allows for a transition between the surrounding media represented by $a_\pm$.
Our running assumption on every coefficient field $a$ is that  they are uniformly elliptic and bounded; namely, there exists a fixed constant $\lambda>0$ such that, for every $x, \xi \in \R^d$, there holds:
\begin{align}
\label{Ellipticite}
\lambda\lt|\xi\rt|^2 \leq \xi \cdot a(x) \xi \quad\text{and}\quad  \lambda \lt|\xi\rt|^2 \leq \xi \cdot a(x)^{-1} \xi.
\end{align}
In order to describe random media, we assume that we have an ensemble $\langl \cdot \rangl$ on the space $\Omega$ (with the topology of H-convergence), which we define as follows:
\begin{align}
\label{Omega}
\Omega:= \lt\{ (a_+, a_-, a_\circ) \,  | \, a_{\pm}, a_\circ : \R^d \rightarrow \R^{d\times d}  \textrm{ satisfy } \eqref{Ellipticite} \rt\}.
\end{align}
In a deterministic case, the measure of the ensemble $\langl\cdot\rangl$ concentrates on one specific coefficient field.

Our first hypothesis is that the coefficient field $a$, $\langle \cdot \rangle$-almost surely, admits the following piecewise constant (deterministic) homogenized matrix $\bar{a}$
\begin{align}
\label{Defabar}
\overline{a}(x)=
\lt\{
\begin{aligned}
\overline{a}_+ \quad\text{if}\quad x^\perp >0,\\
\overline{a}_- \quad\text{if}\quad x^\perp <0.
\end{aligned}
\rt.
\end{align}
(By local properties of H-convergence, $\overline{a}_-$ and $\overline{a}_+$ depend only on $a_-$ and $a_+$ respectively.)
We also assume that, $\langl \cdot \rangl$-almost surely, there exist generalized homogenization correctors
\begin{align*}
\Phi_\pm:=\lt(\Phi_-,\Phi_+\rt)  \quad\text{for}\quad \Phi_-:=\lt(\phi_-,\phi^\star_-,\sigma_-,\sigma^\star_-\rt) \;\text{ and }\; \Phi_+:=\lt(\phi_+,\phi^\star_+,\sigma_+,\sigma^\star_+\rt).
\end{align*}
Here, $(\phi_+, \sigma_+)$ are strictly sublinear functions satisfying\footnote{Notice that our indexing convention for the flux corrector $\sigma$ defined below is different from \cite{Gloria_Neukamm_Otto_2015}.}
\begin{align}\label{Def_Corrpm}
-\div\lt(a_+ \lt( \nabla \lt(\phi_+\rt)_i + e_i\rt) \rt)=0 \quad\text{and}\quad \lt(\sigma_+\rt)_{ijk}:=\partial_i \lt(N_+\rt)_{jk}-\partial_j \lt(N_+\rt)_{ik},
\end{align}
in $\R^d$, where $\lt(N_+\rt)_{jk}$ is a strictly subquadratic solution of the following equation
\begin{align}
\label{Def_Npm}
\Delta \lt(N_+\rt)_{jk}=\lt(\overline{a}_+\rt)_{jk} -\lt(a_+\rt)_{jl} \lt(\delta_{lk} + \partial_l \lt(\phi_+\rt)_k \rt) \quad\text{in}\quad \R^d.
\end{align}
(The other correctors $\lt(\phi_-,\sigma_-\rt)$, $\lt(\phi_-^\star,\sigma_-^\star\rt)$ and  $\lt(\phi_+^\star,\sigma_+^\star\rt)$ correspondingly satisfy similar equations, where the coefficients fields $\lt(a_+,\bar{a}_+\rt)$ should be respectively replaced by $\lt(a_-,\bar{a}_-\rt)$, and the transposed coefficient fields $\lt(a_-^\star,\overline{a}_-^\star\rt)$ and $\lt(a_+^\star,\overline{a}_+^\star\rt)$.)

Our second and main hypothesis about the two heterogeneous media is that the correctors related to $a_{\pm}$ and $a^\star_{\pm}$ are strongly sublinear in the following annealed way:
\begin{equation}
\label{CorrSubDef} 
\begin{aligned}
&\displaystyle \sup_{x,y \in \R^d, |x-y| \leq r} \langl \left( \int_{\Q(0,1)}| \Phi_{\pm}(x+z)- \Phi_{\pm}(y+z)|^2 \dd z \right)^{\frac{p}{2}} \rangl^{\frac{1}{p}}
\leq c_p r^{1-\nu},
\end{aligned}
\end{equation}
for every $r>1$ and $p<\infty$, and for a given exponent $\nu \in (0,1]$ and a constant $c_p \geq 1$ (without loss of generality --by the H\"older inequality-- we may assume that the constants $c_p$ are increasing in $p$).

\begin{remark} 
While we assume that $a_{\pm}$ are coefficient fields on $\R^d$ with corresponding generalized homogenization correctors, it would suffice to have these coefficient fields and generalized correctors defined on $\R_{\pm} \times \R^{d-1}$ with an accordingly modified  assumption \eqref{CorrSubDef}. Also, the coefficient field $a_\circ$ might only be defined on the layer $[-1, 1] \times \R^{d-1}$. We define the space $\Omega$ by \eqref{Omega} for simplicity.
\end{remark}

\subsection{Definition of adapted correctors and $2$-scale expansion}

Following \cite{Josien_InterfPer_2018}, we introduce a basis for the space of strictly subquadratic $\overline{a}$-harmonic functions (see \eqref{Equat_P}): For $j \in [\![  1, d ]\!]$ we define
\begin{align}
\label{DefPj}
P_j(x)=P(x)\cdot e_j:=
\lt\{
\begin{aligned}
&x \cdot e_j && \quad\text{if}\quad x^{\perp}<0, \\
&x \cdot e_j + \frac{\lt(\overline{a}_-\rt)_{1j}-\lt(\overline{a}_+\rt)_{1j}}{\lt(\overline{a}_+\rt)_{11}} x\cdot e_1&& \quad\text{if}\quad x^{\perp}>0,
\end{aligned}
\rt.
\end{align}
where the bottom line corresponds to the transmission condition through the interface. This prompts us to seek homogenization correctors and flux correctors satisfying the following definition:
\begin{definition}[Generalized Correctors]\label{DefGenCorr}
We define the \textit{generalized correctors} $(\phi,\sigma)$ associated to a coefficient field $a$ of the form \eqref{Defa} as follows.
The correctors $\phi_j$, for $j \in \[1,d\]$, are strictly sublinear solutions to the following equation:
\begin{align}
\label{DefCorr}
-\div \lt( a  \nabla \lt( P_j + \phi_j \rt)\rt)=0 \quad\text{in}\quad \R^d.
\end{align}
Simultaneously, the  flux correctors $\sigma_{ijk}$, for $i, j, k \in \[1,d\]$ are defined as
\begin{align}
\label{Def2_Pot}
\sigma_{ijk}:=\partial_i N_{jk}-\partial_j N_{ik},
\end{align}    
where $N_{jk}$ is a strictly subquadratic solution of the following equation:
\begin{align}
\label{Def_N1}
\Delta N_{jk}=\overline{a}_{jl}  \partial_l P_k -a_{jl} \lt(\partial_l P_k + \partial_l \phi_k \rt) \quad\text{in}\quad \R^d.
\end{align}
\end{definition}

At this point we make an important distinction: Notice that \eqref{Def2_Pot} and \eqref{Def_N1} imply that the flux corrector $\sigma$ satisfies the familiar identity \cite[(7)]{Gloria_Neukamm_Otto_2015}
\begin{align}\label{Def_Pot}
\partial_i \sigma_{ijk} =\overline{a}_{jl}  \partial_l P_k -a_{jl} \lt(\partial_l P_k + \partial_l \phi_k \rt) \quad\text{in}\quad \R^d
\end{align}
along with the skew-symmetry constraint
\begin{align}
\label{PotSym}
\sigma_{ijk}=-\sigma_{jik}.
\end{align}
It turns out that the two latter identities are sufficient for many purposes (\textit{e.g.} to obtain large-scale Lipschitz estimates, that is Theorem \ref{ThAL} below).
Functions $\sigma^{\rm u} $ that are strictly sublinear and satisfy \eqref{Def_Pot} and \eqref{PotSym} we call \textit{ungauged flux correctors}; we use the superscript ``${\rm u}$'' to indicate it. 
The main difference between the gauged and ungauged flux correctors is that, in contrast to the gauged flux correctors of Definition \ref{DefGenCorr}, the ungauged flux correctors are not unique, which becomes an issue in the proof of Theorem \ref{PropRefinedDeter}.

In our setting with the interface we need a modification of the standard $2$-scale expansion, that is:
\begin{align}
\label{2scale}
\widetilde{u}:=\overline{u} + \phi \cdot \lt( \nabla P\rt)^{-1} \nabla \overline{u}.
\end{align}
With this definition of $\widetilde{u}$ we find that
\begin{align}
\label{Algebre}
-\div\lt(a \cdot \nabla \lt(u-\widetilde{u}\rt) \rt) =\partial_i \lt( \lt(a_{ij} \phi_k-\sigma_{ijk}\rt)\partial_j \overline{\partial}_k \overline{u} \rt),
\end{align}
where
\begin{align}\label{DefUUU}
\overline{\nabla}\overline{u}:=\lt( \nabla P \rt)^{-1}  \nabla \overline{u} \quad \quad \textrm{and} \quad \quad \overline{\partial}_k \overline{u}:= e_k \cdot  \overline{\nabla}\overline{u}.
\end{align}
The motivation for \eqref{2scale} and the detailed calculation leading to \eqref{Algebre} lie in \cite[Section 3.3]{Josien_InterfPer_2018}. We underline that the function $\overline{\nabla} \overline{u}$ defined in \eqref{DefUUU} is continuous through the interface: Thus, its gradient $\nabla \overline{\nabla} \overline{u}$ lies in $\LL^\infty_{\loc}\lt(\R^d\rt)$, so that the terms on the right-hand side of \eqref{Algebre} are well-defined  (see Lemma \ref{LemRegUbar} below).

\subsection{Theorem \ref{PropRefinedDeter}: Main result}
\label{main_result}

The main contribution of this article is the following:
\begin{theorem}\label{PropRefinedDeter}
Let $d \geq 2$ and $\langl \cdot \rangl$ be an ensemble on $\Omega$ defined in \eqref{Omega} that satisfies the conditions given in Section \ref{ensemble}. Then, $\langle \cdot \rangle$-almost surely there exists a unique (up to addition of a random constant\footnote{We use here the quite paradoxical words ``random constant'' to design a random field that is constant in space.}) generalized corrector $(\phi, \sigma)$ associated to $a$ such that for every $\nu_0 < \nu$ and $2 \leq p<\infty$ the following relations hold for any $r>0$:
\begin{align}\label{CorrSubNu}
\begin{aligned}
&  \sup_{x,y \in \R^d, |x-y| =r} \langl \left(  \int_{\Boule(0,1)}  \lt|\phi(x+z)-\phi(y+z)\rt|^2 \dd z \right)^{\frac{p}{2}} \rangl^{\frac{1}{p}} 
\\
& \hspace{1cm} \lesssim_{d, \lambda, \nu, \nu_0, p}    c^{d/\nu_0 +1}_{2pd/\nu}      \lt\{
\begin{aligned}
&\lt(1+ r \rt)^{1-\nu}  &&\quad\text{if}\quad \nu<1,
\\
& \ln(2+r)  &&\quad\text{if}\quad \nu=1,
\end{aligned}
\rt.
\end{aligned}
\end{align}
and
\begin{align}
\nonumber
&  \sup_{x,y \in \R^d, |x-y| =r} \langl  \lt(\int_{\Boule(0,1)} \lt|\sigma(x + z)-\sigma(y + z) \rt|^2 \dd z \rt)^{\frac{p}{2}}  \rangl^{\frac{1}{p}}      \\
& \hspace{1cm} \lesssim_{d, \lambda, \nu, \nu_0,  p}c^{d/\nu_0 +1}_{2pd/\nu}
\lt\{
\begin{aligned}
&\lt(1+r\rt)^{1-\nu} \ln (2+r )  &&\quad\text{if}\quad \nu<1,
\\
& \ln^3(2+r)  &&\quad\text{if}\quad \nu=1.
\end{aligned}
\rt.
\label{CorrSubNuSig}
\end{align}
\end{theorem}

\noindent Above and in the sequel, the symbol ``$\lesssim_\delta$'' reads ``$\leq C$, for a constant $C$ depending only on the tuple $\delta$ of previously defined parameters'' (for simplicity, throuhgout the course of the proofs, the subscript might be omitted).

In words, as previously advertised, we learn from Theorem \ref{PropRefinedDeter} that the correctors adapted to the interface enjoy the same quantified sublinearity properties as the correctors on the left and on the right of the interface (possibly up to a logarithmic correction). We emphasize that we do not assume any structure of the coefficients on the left and on the right. However, the sublinearity property \eqref{CorrSubDef} usually comes from an underlying structure such as, in the deterministic case, periodicity, quasi-periodicity, or periodicity perturbed by a defect and, in the stochastic case, stationarity and quantitative ergodicity assumptions; \textit{e.g.} a log-Sobolev inequality or a spectral gap \cite{Gloria_Neukamm_Otto_2015}. As a consequence, Theorem \ref{PropRefinedDeter} may be applied in various frameworks, as illustrated in Section \ref{SecExample} below.

Theorem \ref{PropRefinedDeter} is a bit surprising in light of \cite{BLLcpde,Josien_InterfPer_2018} since we do not assume any structural relationship between the coefficients $a_\pm$.
Indeed, in \cite[Prop.\ 5.4]{Josien_InterfPer_2018} the coefficients $a_\pm$ have a common periodic cell in the directions of the interface $\mathcal{I}$ and in \cite[Th.\ 5.7]{BLLcpde} a diophantine condition relating the periods of the coefficients $a_\pm$ is assumed.
The more general statement in Theorem \ref{PropRefinedDeter}, however, does come at a cost. In particular, defining the glued composite correctors 
\begin{align}\label{DefVecphi}
\check{\phi}(x):=
\lt\{
\begin{aligned}
\phi_+(x) \quad\text{if}\quad x^{\perp} >0,\\
\phi_-(x) \quad\text{if}\quad x^{\perp} <0
\end{aligned}
\rt.
\quad\text{and}\quad 
\check{\sigma}(x):=
\lt\{
\begin{aligned}
\sigma_+(x) \quad\text{if}\quad x^{\perp} >0,\\
\sigma_-(x) \quad\text{if}\quad x^{\perp} <0,
\end{aligned}
\rt.
\end{align}
we remark that the estimates in \cite[Prop.\ 5.4]{Josien_InterfPer_2018} provide an exponential decay of $\nabla \phi(x) - \nabla \check{\phi}(x)$ (and accordingly of $\nabla \sigma(x) - \nabla \check{\sigma}(x)$) in the distance to the interface $|x^\perp|$. In contrast, our methods used to prove Theorem \ref{PropRefinedDeter} only yield a decay as the inverse of this distance. 

Following our proof of Theorem \ref{PropRefinedDeter}, we show that enforcing a structural assumption between the two surrounding media may lead to stronger estimates than \eqref{CorrSubNu} or \eqref{CorrSubNuSig}, and not only in the periodic case \cite{Josien_InterfPer_2018}. In particular, we prove in Theorem \ref{ThIndepGauss} that, in a special stochastic setting where both heterogeneous media are \textit{independently} generated from two Gaussian fields with integrable correlation functions, we obtain the optimal growth rates of the correctors. 
Namely, all the stochastic moments of the generalized corrector are uniformly bounded in $\R^d$. 

As a counterpart to Theorem \ref{ThIndepGauss}, we then justify that, under the assumptions outlined in Section \ref{ensemble}, the rate \eqref{CorrSubNu} is optimal. (The only non-obvious case is $\nu=1$.) In particular, in Proposition \ref{PropContrex} we give an example of coefficients $a_\pm$ that admit bounded correctors and a uniformly elliptic and bounded $a_\circ$ such that the global corrector for the medium with interface displays a logarithmic growth.\footnote{Nevertheless, we do not claim that the exponent in the logarithm of \eqref{CorrSubNuSig} is optimal.} 

We lastly underline that, apart from boundedness and uniform ellipticity, no further assumptions are imposed on $a_\circ$ inside the layer of width 2 along the interface. This is indeed a zone that we need to ``sacrifice'' in the proof of Theorem \ref{PropRefinedDeter} because of our use of cut-off functions --we cannot take advantage of any good behavior of $a$ in this zone, but we also do not suffer from any bad behavior. As can be seen in \eqref{CorrSubNu}, the presence of this zone does not worsen the growth rate when $\nu<1$, but its influence is felt when $\nu =1$. (This is also apparent in Proposition \ref{PropContrex}.) In the terminology of \cite{BLLcpde}, the layer $a_\circ$ could be seen as a \textit{defect}  which is only in $\LL^\infty(\R^d)$, but not in any $\LL^r(\R^d)$ for $r<\infty$. (It is restricted to a layer of infinite Lebesgue measure, although it appears as ``microscopic'' when zooming out.)

\subsection{Applicability and examples}\label{SecExample}

In this section, we discuss the assumption \eqref{CorrSubDef} and propose three different simple, but representative, examples of interfaces between heterogeneous media that satisfy it.

In the case that $\nu =1$ the assumption \eqref{CorrSubDef} is quite well-motivated. In particular, if $a_{\pm}$ are  both periodic, then the corresponding generalized correctors are bounded and \eqref{CorrSubDef} holds with $\nu=1$.  Of course, the assumption \eqref{CorrSubDef} is less common for $\nu < 1$. It has, however, been shown that such a growth rate naturally arises when studying periodic media perturbed by a defect that is quite spread (see Example \ref{Example_PerDef} below). In particular, this situation was studied in \cite{Article_Homog,CRAS_Homog}. Also, in stochastic homogenization, taking a general random field satisfying a log-Sobolev inequality (see \cite[Th.\ 3]{Gloria_Neukamm_Otto_2015}) produces \eqref{CorrSubDef} with an exponent $\nu$ depending on the parameter of the log-Sobolev inequality and the dimension $d$.

\medskip

The first example is totally deterministic:
\begin{Example}\label{Example_PerDef}
The matrices $a_\pm$ represent periodic media perturbed by defects:
\begin{align*}
a_- = a_{\per,-} + \tilde{a}_- \quad\text{and}\quad a_+ = a_{\per,+} + \tilde{a}_+,
\end{align*}
where the coefficient fields $a_{\per,\pm}$ are both periodic (with possibly different periods) and H\"older continuous.
Moreover, the defects are localized in the sense of $\tilde{a}_\pm \in \LL^{\infty}\lt(\R^d\rt) \cap \LL^r\lt(\R^d\rt)$, for $r \in [1,\infty)$, and uniformly H\"older continuous.
The coefficient fields $a_\pm$, $a_{\per,\pm}$ satisfy \eqref{Ellipticite}.
There is no layer, in the sense of:
\begin{equation}\label{NoLayer}
a_\circ(x):=
\lt\{
\begin{aligned}
&a_-(x) \quad\text{if}\quad x^{\perp} <0,
\\
&a_+(x) \quad\text{if}\quad x^{\perp} >0.
\end{aligned}
\rt.
\end{equation} 
\end{Example}
In such a case, by \cite[Th.\ 4.1]{BLLcpde}, \eqref{CorrSubDef} is satisfied for $\nu:=\min\lt(1,\frac{d}{r}\rt)$ if $r \neq d$ (the special case $r=d$ can be treated in a suboptimal way by artificially increasing $r$), and for a trivial ensemble $\langl\cdot\rangl$.
Such an example is illustrated on the left-hand side of Figure \ref{Fig_Examples}, and might be a realistic model for an interface between two crystals.

The second and third examples are stochastic:

\begin{Example}\label{Ex_2GaussIndep}
Let $d \geq 2$, $d_{\rm g} \geq 1$, $0<\lambda<1$ and $\kappa>0$.
Let $c_-$, $c_+$ and $c_\circ : \R^d \rightarrow \R^{d_{\rm g} \times d_{\rm g}}$ be covariance functions such that their Fourier transforms satisfy
\begin{align}
\label{Hypoc_bis}
\lt|\FF c_-(k) \rt| + \lt|\FF c_+(k) \rt| + \lt| \FF c_\circ(k)\rt| \leq \kappa \lt(1 + |k|\rt)^{-d-2\alpha},
\end{align}
for any $k \in \R^d$, and for a given exponent $\alpha>0$, and let the deterministic matrix-valued functions $A_-$, $A_+$, $A_\circ : \R^{d_{\rm g}} \rightarrow \R^{d\times d}$ be such that each element in the range of $A_-$, $A_+$, and $A_\circ$ satisfies \eqref{Ellipticite}.
Assume that these functions are uniformly Lipschitz continuous, \textit{\textit{i.e.}},
\begin{align}\label{HypoA_bis}
\lt\|\nabla A_-\rt\|_{\LL^{\infty}}
+\lt\|\nabla A_+\rt\|_{\LL^{\infty}}
+\lt\|\nabla A_\circ\rt\|_{\LL^{\infty}}
\leq \kappa.
\end{align}

The coefficient fields $a_+$, $a_-$ and $a_\circ$ are generated from independent vectorial stationary Gaussian fields $g_-, g_+$ and $g_\circ : \R^d \rightarrow \R^{d_{\rm g}}$ with correlation functions $c_-$, $c_+$ and $c_\circ$ in the following sense:
\begin{align}\label{Def_a_Ex}
a_-(x):=A_-(g_-(x)),  \quad a_+(x):=A_+(g_+(x)) \quad\text{and}\quad a_\circ(x):=A_\circ(g_\circ(x)).
\end{align}
We denote by $\langl\cdot\rangl$ the ensemble induced by the joint laws of $g_-$, $g_+$ and $g_\circ$.
\end{Example}
By \cite[Prop.\ 3.2]{JosienOtto_2019} (see also \cite{Gloria_Neukamm_Otto_2015}), estimate \eqref{CorrSubDef} is satisfied for $\nu=1$ in Example \ref{Ex_2GaussIndep}. Of course, it might be more realistic to make use of the layer coefficient $a_\circ$ to have a smooth transition between the two surrounding media as in Figure \ref{Fig_Examples}.

\medskip 

There is no need to assume independence between all the media. Instead, there might be a total correlation (here by reflection) between the medium on the left and the medium on the right:
\begin{Example}
This example is similar to Example \ref{Ex_2GaussIndep}. The only difference is that we set
\begin{equation*}
a_-(x):=A_-(g_-(x)), \quad a_+(x):=a_-(-x) \quad\text{and}\;\; a_\circ(x) \text{ is defined by \eqref{NoLayer}}.
\end{equation*}
instead of \eqref{Def_a_Ex}.
\end{Example}

\begin{remark}
From a practical point of view, it may happen that correctors related to some heterogeneous materials can be computed numerically.
Thus, condition \eqref{CorrSubDef} would be easier to check than a structure assumption.
\end{remark}

\subsection{Outline}

The article is organized as follows:
In Section \ref{SecResult}, we sketch the main steps leading to Theorem \ref{PropRefinedDeter}, discussing some technical aspects.
Then, we state a few additional results: We deduce from Theorem \ref{PropRefinedDeter} an almost-optimal convergence rate for the two-scale expansion; we also provide an example of interface, where the rate \eqref{CorrSubNu} is attained; and, in a special stochastic case, we get a slightly better growth rate for the generalized correctors.
The Sections  \ref{section_4} - \ref{Sec_Proof_ThIndepGauss} are devoted to the proofs.
Namely, Sections \ref{section_4}, \ref{sec:Prop_1} and \ref{section_6} contain the proof of Theorem \ref{PropRefinedDeter}, each of them corresponding to an intermediate result, whereas Sections \ref{Sec_proof_ThCvgDet}, \ref{Sec_Proof_PropContrex} and \ref{Sec_Proof_ThIndepGauss} contain the proofs of the additional results.
Last, we state and prove in Appendix \ref{A} a useful result on harmonic functions.

\section{Strategy of proof and additional results}\label{SecResult}

\subsection{Strategy for the proof of Theorem \ref{PropRefinedDeter}}

We go through the following sequence of steps: First, in Theorem \ref{ThAL}, we assume access to a strictly sublinear generalized ungauged corrector (see Definition \ref{DefGenCorr}) and obtain an averaged Lipschitz estimate for $a$-harmonic functions above some minimal radius $r^\star>0$. Then, in Proposition \ref{PropExistCorr}, we show that, assuming the existence of generalized correctors $\Phi_{\pm}$ corresponding to $a_{\pm}$ satisfying \eqref{CorrSubDef}, we can construct the generalized ungauged corrector needed as input in Theorem \ref{ThAL}. Therefore, we obtain a large-scale Lipschitz estimate for $a$-harmonic functions. In turn, the latter yields annealed estimates for the first and second mixed derivatives of the Green's function associated to $- \div \left( a \nabla\right)$ (as shown in \cite{BellaGiunti_2018}). These Green's function estimates are a main ingredient to get the almost-optimal growth rates in Theorem \ref{PropRefinedDeter}. 
Their use is complemented by Lemma \ref{enforce_gauge_1}, in which we go from the ungauged flux corrector that comes out of Proposition \ref{PropExistCorr} to a unique (up to addition of a random constant) flux corrector satisfying the same sublinearity properties, and Lemma \ref{moments}, in which we control the moments of the minimal radius $r^*$. 

\subsection{Theorem \ref{ThAL}: Large-scale Lipschitz estimate}

Our Theorem \ref{ThAL} generalizes the previous result \cite[Th.\ 4.1]{Josien_InterfPer_2018} by adapting the proof of \cite[Lem.\ 2]{Gloria_Neukamm_Otto_2015}.
It takes as input strictly sublinear ungauged generalized correctors and yields a large-scale Lipschitz estimate for $a$-harmonic functions. 
The method in \cite{Gloria_Neukamm_Otto_2015} is inspired by the earlier work of Avellaneda and Lin in the setting of periodic coefficients \cite{AvellanedaLin}.
The main idea is to transfer regularity properties from the constant-coefficient homogenized operator to the heterogenous operator at large scales.
In our case, to overcome the discontinuity of the homogenized matrix at the interface, we need to use the modified $2$-scale expansion \eqref{2scale}.

We use the convention that the \textit{excess of an $a$-harmonic function on the ball of radius $r > 0$  centered around $x_0 \in \R$} is given by:
\begin{align}
\label{DefExcess}
\mathcal{E}(x_0,r)[u]=\inf_{\xi \in \R^d} \fint_{\Boule(x_0,r)} \lt|\nabla u - \lt( \nabla P + \nabla \phi \rt)\cdot \xi \rt|^2 .
\end{align}
For this definition of the excess we obtain the following large-scale regularity result:

\begin{theorem}\label{ThAL}
Assume that the coefficient field $a$ has the form \eqref{Defa} and satisfies \eqref{Ellipticite}, the homogenized matrix $\overline{a}$ has the form \eqref{Defabar}, and the $\overline{a}$-harmonic coordinates are defined by \eqref{DefPj}. We let $\lt(\phi,\sigma^{\rm u} \rt)$ denote an associated generalized ungauged corrector. Then, for any Hölder exponent $\alpha \in (0,1)$, there exists a constant $\delta=\delta(d,\lambda,\alpha)$ such that the following properties hold:

Let $x_0 \in \R^d$ and $r_{\max}>r^\star>0$.
Assume that $(\phi,\sigma^{\rm u} )$ satisfy the sublinearity condition
\begin{align}
\label{CorrSub}
\sup_{r \in [r^\star,r_{\max}]} \frac{1}{r} \lt( \fint_{\Boule(x_0,r)} \lt| (\phi,\sigma^{\rm u} ) - \fint_{\Boule(x_0,r)} (\phi,\sigma^{\rm u} ) \rt|^2 \rt)^{\frac{1}{2}} \leq \delta.
\end{align}
Then, for $R\in[r^\star,r_{\max}]$ and a function $u$ that is $a$-harmonic in $\Boule(x_0,R)$, we have that
\begin{align}\label{Borne_Excess}
\mathcal{E}(x_0,r)[u] \leq \delta^{-1} \lt(\frac{r}{R}\rt)^{2\alpha} \mathcal{E}(x_0,R)[u]
\end{align}
for any $r \in [r^\star,R]$, where the excess $\mathcal{E}$ is defined by \eqref{DefExcess}.
Moreover, the correctors have the following non-degeneracy property:
\begin{align}\label{NonDegen}
\delta \lt|\xi\rt|^2 \leq \fint_{\Boule(x_0,r)} \lt| \nabla P\cdot \xi + \nabla \phi \cdot \xi \rt|^2  \leq \delta^{-1} \lt|\xi\rt|^2
\end{align}
for any $\xi \in \R^d$ and $r \in [r^\star,r_{\max}/2]$.
Finally, the following large-scale Lipschitz estimate holds for any $r^\star \leq r \leq R \leq r_{\max}$:
\begin{align}\label{LargeLip}
\fint_{\Boule(x_0,r)} \lt|\nabla u\rt|^2  \leq \delta^{-1} \fint_{\Boule(x_0,R)} \lt|\nabla u\rt|^2.
\end{align}
\end{theorem}

Notice that, due to the presence of the interface, $a(\cdot + x_0)$ does not have the same structure as $a$ if $x_0 \neq 0$ and we may, therefore, not assume that $x_0 = 0$ as in \cite{Gloria_Neukamm_Otto_2015}. 

\begin{remark}[Liouville theorem]
Theorem \ref{ThAL} implies that the space of $a$-harmonic functions $u \in \HH^1_{\loc}\lt(\R^d\rt)$ that are strictly subquadratic in the sense that there exists $\alpha \in (0,1)$ such that 
\begin{align*}
\lim_{R \uparrow \infty} R^{-(1+ \alpha)} \lt( \fint_{\Boule(0,R)} \lt|u -\fint_{\Boule(0,R)} u \rt|^2 \rt)^{\frac{1}{2}} =0,
\end{align*}
is of dimension $d+1$. More precisely, such functions $u$ can be written as
\begin{align*}
u(x)= c + b \cdot \lt( P(x) + \phi(x)\rt) \quad \text{ for constants } c \in \R, b \in \R^d.
\end{align*}
\end{remark}

\subsection{Proposition \ref{PropExistCorr}: Construction of generalized \textit{ungauged} correctors}
\label{state_correctors}

For the construction of the generalized ungauged corrector $(\phi,\sigma^{\rm u} )$ that we take as input for Theorem \ref{ThAL}, we adapt the method used in \cite{Fischer_Raithel_2017, Raithel_2017} to construct Dirichlet and Neumann correctors. The general iterative scheme of \cite{Fischer_Raithel_2017, Raithel_2017} was first introduced to build higher order correctors in \cite{Fischer_Otto_2015}.  

In \cite{Fischer_Raithel_2017, Raithel_2017, Fischer_Otto_2015} it is sufficient to assume existence of a whole-space (first order) corrector satisfying a quantified sublinearity condition. 
For simplicity, here, we restrict ourselves to a slightly less general sublinearity condition\footnote{An inspection of the proof of Proposition \ref{PropExistCorr} should convince the reader that the result, in fact, holds under the direct analogue of the quantified sublinearity condition \cite[(11)]{Fischer_Raithel_2017}. }.
Indeed, we assume that there exist whole-space correctors $\lt(\phi_-,\sigma_-\right)$ and $\left(\phi_+,\sigma_+\rt)$ associated with $a_{\pm}$ such that 
\begin{align}
\label{Condition_delta}
\begin{split}
&\frac{1}{r} \lt(\fint_{\Boule(x_0,r)} \lt| \lt(\phi_-,\sigma_-,\phi_+,\sigma_+\rt) \rt|^2  \rt)^{\frac{1}{2}} \leq r^{-\nu} \qquad \text{for any} \quad  r  \geq 1,
\end{split}
\end{align}
for a given exponent $\nu \in (0,1]$ and $x_0 \in \R^d$. Note that if \eqref{CorrSubDef} is satisfied, then, $\langle \cdot \rangle$-almost surely, for any $x_0 \in \mathbb{R}^d$ there exist $\lt(\phi_-,\sigma_-\right)$ and $\left(\phi_+,\sigma_+\rt)$ such that \eqref{Condition_delta} holds (up to a uniform multiplicative constant).

The basic strategy of the construction we use here is to iteratively, on increasingly large scales, correct the glued composite correctors $(\check{\phi},\check{\sigma})$ defined in \eqref{DefVecphi}. The intuition is that far from the interface $\lt(\nabla \phi,\nabla \sigma^{\rm u} \rt)$ should behave like $\lt(\nabla \phi_\pm \cdot \nabla P,\nabla\sigma_\pm \cdot\nabla P\rt)$ on the right/left. This naturally leads to the ansatz:
\begin{align}\label{Defsw}
\phi_k  		=(1-\chi) \check{\phi}_j \partial_j P_k 		+ \tilde{\phi}_k, &&
\sigma^{\rm u} _{ijk} 	=(1-\chi) \check{\sigma}_{ijl} \partial_l P_k 	+ \tilde{\sigma}_{ijk},
\end{align}
where the function $\chi$ is smooth, equals $1$ on a narrow layer along the interface (containing the interface layer $[-1,1] \times \R^{d-1}$), and vanishes far from the interface (it will specified precisely in Section \ref{SecEquatConst}). The functions $\tilde{\phi}$ and $\tilde{\sigma}$ correspond to layer corrections along the interface.  

More formally, we decompose $\R^d$ into dyadic annuli and solve the corrector equations \eqref{DefCorr} and \eqref{Def_N1} in the associated increasing balls by using the ansatz \eqref{Defsw}. We thus obtain a sequence $\left\{ \lt(\phi^M, \sigma^{{\rm u}, M}\rt) \right\}_{M \in \mathbb{N}}$ of ``local generalized ungauged correctors''. An induction argument yields the convergence of this sequence. Indeed, by appealing to the large-scale Lipschitz estimate of Theorem \ref{ThAL} in the $M$\textsuperscript{th} step, we ascertain a sublinearity estimate on the local generalized ungauged correctors $\lt(\phi^M,\sigma^{{\rm u },M}\rt)$. This latter property is then used in the $M+1$\textsuperscript{th} step in order to invoke Theorem \ref{ThAL} again. 
Last, by these sublinearity estimates, we find that $\lt(\phi^M,\sigma^{{\rm u},M}\rt)$ converges to solutions $(\phi,\sigma^{\rm u} )$ of \eqref{DefCorr} and \eqref{Def_Pot} on the whole-space $\R^d$.

Using the above strategy, we obtain the following result:

\begin{proposition}\label{PropExistCorr}
Let $a$ be defined in \eqref{Defa} and satisfy \eqref{Ellipticite}, and $\overline{a}_\pm$ be the constant homogenized matrices associated with $a_\pm$.
Assume that there exist generalized correctors $\lt(\phi_-,\sigma_-\right)$ and $\left(\phi_+,\sigma_+\rt)$ associated with $a_{\pm}$ such that \eqref{Condition_delta} holds for $\nu \in (0,1]$ and $x_0 \in \R^d$. 

Then, there exists a generalized ungauged corrector $\lt(\phi,\sigma^{\rm u} \rt)$ associated with the coefficient field $a$ that satisfies
\begin{align}
\label{SublinSSop_2}
\frac{1}{r} \lt(\fint_{\Boule(x_0,r)} \lt| \lt(\phi,\sigma^{\rm u} \rt) - \fint_{\Boule(x_0,1)} \lt(\phi,\sigma^{\rm u} \rt) \rt|^2  \rt)^{\frac{1}{2}} \leq \kappa r^{-\tilde{\nu}} \quad \text{for any} \quad r  \geq 1,
\end{align}
for the exponent $\tilde{\nu}:=\nu/3$, and a constant $\kappa$ depending on $d, \lambda$ and $ \nu$.
\end{proposition}

\begin{remark}
Proposition \eqref{PropExistCorr} also applies if the input correctors are ungauged.
\end{remark}

Notice that the sublinearity condition \eqref{Condition_delta} involves an anchoring point $x_0 \in \mathbb{R}^d$. As we will see in Section \ref{SecEquatConst}, this affects the definition of the cut-off functions defining the various local generalized ungauged correctors. In particular, in the method that we have described above, the successively large annuli were implicitly centered at $0$. Of course, the output $(\phi, \sigma^{\rm u} )$ of the proposition apriori depends on the anchoring point $x_0$. However, we see in Section \ref{SecUnique0} that the gradient of the corrector $\nabla \phi$ is unique and thus independent of $x_0$, whereas $\nabla \sigma^{\rm u} $ generally depends on $x_0$ (because $\sigma^{\rm u} $ is ungauged and, therefore, not unique up to the addition of a random constant).

As already observed in \cite{Fischer_Raithel_2017} for the boundary correctors, the estimate \eqref{SublinSSop_2} is suboptimal in terms of the exponent $\tilde{\nu}$. Actually, even if the generalized correctors $\lt(\phi_\pm,\sigma_\pm\rt)$ were uniformly bounded, optimizing this method would only upgrade \eqref{SublinSSop_2} to the exponent $\tilde{\nu}=1/2$ (whereas one may hope for $\tilde{\nu}=1$). The non-optimality of this estimate is an inherent feature of the method, which relies on energy estimates to capture the "smallness" of the layer around the interface. This strategy is predestined to be suboptimal: Indeed, the normalized $\LL^2$-energy corresponding to the layer around the interface and inside a ball of radius $r$ scales like $r^{-1/2}$, whereas the normalized $\LL^1$-norm of the same domain scales like $r^{-1}$.  In particular, we formally have that 
\begin{align*}
\lt(\fint_{\Boule(0,r)} \lt(\mathds{1}_{[-1,1] \times \R^{d-1}}\rt)^2  \rt)^{\frac{1}{2}} 
\simeq r^{-1/2} \gg r^{-1} \simeq  \fint_{\Boule(0,r)} \mathds{1}_{[-1,1] \times \R^{d-1}} \quad\text{for}\quad r \gg 1.
\end{align*}
Thus, the energy norm is not the best way to account for the smallness of the layer. This observation advocates for using more refined tools, namely estimates for the mixed gradient of the Green's function. The latter will transfer the $\LL^1$ optimal bound corresponding to the layer to an $\LL^\infty$ bound for the growth rate of the generalized correctors (up to logarithmic losses). This remark is at the core of the proof of Theorem \ref{PropRefinedDeter} and a key observation of this paper.

\subsection{Theorem \ref{ThIndepGauss}: Improved rates via independence}

To demonstrate the price that we pay in Theorem \ref{PropRefinedDeter} due to a lack of joint structure assumptions on the media, we take a closer look at the situation of Example 2 in Section \ref{SecExample}. Here, a relationship between the two media and the layer is enforced by assuming independence of their laws. As an analogue of \cite[Sec.\ 3.2]{JosienOtto_2019}, we obtain:

\begin{theorem}
\label{ThIndepGauss} Assuming the situation described in Example \ref{Ex_2GaussIndep} of Section \ref{SecExample}, there exists a unique (up to addition of a random constant) generalized corrector $\lt(\phi,\sigma\rt)$ associated with the coefficient field $a$ that satisfies the estimate:
\begin{align}
\label{SublinRes}
\langl \lt| \lt( \phi,\sigma \rt)(x) - \lt(\phi,\sigma\rt)(y) \rt|^p \rangl^{\frac{1}{p}}
\lesssim_{d, \lambda, \kappa, \alpha, p}
\lt\{
\begin{aligned}
&  \ln^{\frac{1}{2}}\lt(2+|x-y|\rt) && \quad\text{if}\quad d=2,
\\
& 1 && \quad\text{if}\quad d \geq 3.
\end{aligned}
\rt.
\end{align}
\end{theorem}

Our argument for Theorem \ref{ThIndepGauss} is essentially a corollary of the techniques used in \cite[Sec.\ 3.2]{JosienOtto_2019} (where there is no interface), which rely on the availability of a powerful tool: a spectral gap estimate.
This ingredient is actually available in our current setting.
Indeed, thanks to the independence assumption\footnote{
Independence is a sufficient --but not necessary-- condition for deriving a spectral gap for the global medium.} and \eqref{Hypoc_bis}, the ensemble $\langl \cdot\rangl$ is such that, 
for any functional $F=F(a)$ and $p \in [1,\infty)$ the following spectral gap holds:
\begin{align}
\label{Prop21a}
\langl \lt|F-\langl F \rangl\rt|^{2p}\rangl^{\frac{1}{p}} 
\lesssim_{d, \lambda, \kappa,  \alpha ,p } \langl \lt( \int_{\R^d}\lt| \frac{\partial F}{\partial a(z)}\rt|^2 \dd z \rt)^{p} \rangl^{\frac{1}{p}},
\end{align}
where we make use of the functional derivative defined by
\begin{align*}
\lim_{\epsilon \rightarrow 0} \frac{F(a+\epsilon \delta a) - F(a)}{\epsilon} = \int_{\R^d} \frac{\partial F(a)}{\partial a(z)}  \lt(\delta a(z)\rt) \dd z.
\end{align*}
Moreover, the Gaussian fields $g=g_-$, $g=g_+$, and $g=g_\circ$ are smooth in the following sense: For any $0<\alpha'<\alpha$ and $p \in [1,\infty)$, there holds
\begin{align}\label{wg03}
\langl\lt(\sup_{x,x'\in \Boule(0,1)}\frac{|g(x)-g(x')|}{|x-x'|^{\alpha'}}\rt)^p\rangl^\frac{1}{p}
\lesssim_{d, \lambda,  \kappa ,  \alpha,\alpha',p } 1.
\end{align} 

\subsection{Proposition \ref{PropContrex}: Example for optimality of Theorem \ref{PropRefinedDeter}}

Let $\eta$ be a function on $\R^d$ defined by
\begin{align}\label{Defeta_strip}
\eta(x):=\eta_1\lt(x^\perp\rt) \eta_2\lt(x \cdot e_2\rt),
\end{align}
where $\eta_1$ and $\eta_2:\R \rightarrow [0,1]$ are two smooth functions such that
\begin{align*}
\lt\{
\begin{aligned}
&[1, +\infty) \subset \lt\{ t \, : \, \eta_1(t) =1 \rt\} \subset \Supp(\eta_1) \subset [0, + \infty),
\\
&[-1/2,1/2] \subset \lt\{ t \,:\, \eta_2(t) =1 \rt\} \subset \Supp(\eta_2) \subset [-1,1].
\end{aligned}
\rt.
\end{align*}
By definition \eqref{Defeta_strip}, the support of $\eta$ lies in the strip $D$ defined by
\begin{align}
\label{DefD}
D:= [0,+\infty) \times [-1,1] \times \R^{d-2}.
\end{align}
We then define the symmetric coefficient field $a$ as
\begin{align}\label{CoeffDefA}
a(x):= \Id + \eta(x) e_1 \otimes e_1.
\end{align}

\begin{figure}[h]
\begin{center}
\includegraphics[scale=0.9]{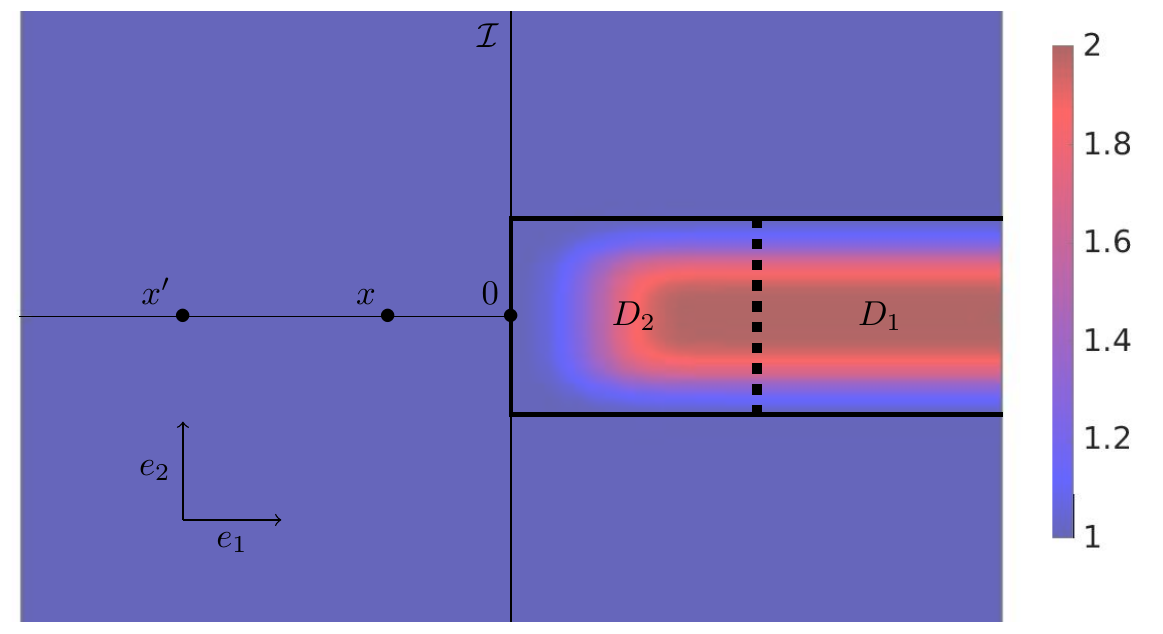}
\end{center}
\caption{Value of $a_{11}(x)$. The support of $\eta$ corresponds to the zone $D=D_1 \cup D_2$.}\label{FigureDefa_strip}
\end{figure}

Note that one may write $a$ in the form \eqref{Defa} for
$a_-(x):=\Id$, $a_+(x):=\Id+\eta(x) e_1 \otimes e_1$, and $a_\circ=\Id + \eta(x) e_1 \otimes e_1$.
Hence, $a$ admits the matrix $\overline{a}=\Id$ as its homogenized matrix.
Moreover, we easily derive that
\begin{align*}
&\phi_-=\phi_+=0, \quad \sigma_-=0,
\quad\text{and}\quad \lt(\sigma_+\rt)_{ijk}=\lt(\delta_{i1}\delta_{j2} - \delta_{i2}\delta_{j1} \rt)\delta_{k1} \int_0^{x \cdot e_2} \eta_2.
\end{align*}
Thus, the generalized correctors $\lt(\phi_{\pm}, \sigma_{\pm}\rt)$ are uniformly bounded in $\R^d$.  Now, the relevant observation is that the global corrector $\phi$ has an unbounded growth rate:

\begin{proposition}
\label{PropContrex}
Assume that $d\geq 3$. 
Let the coefficient field $a$ be defined by \eqref{CoeffDefA} and $\phi$ be the associated unique (up to addition of a constant) strictly sublinear corrector.
Then, there exists a constant $C(d) >0$ such that for any $r\geq 1$ there exist points $x, x' \in \R^d$ satisfying
\begin{align}
\lt|\phi(x)-\phi(x')\rt| \geq C \ln\lt(2+r\rt) \quad\text{and}\quad \lt|x-x'\rt| =r.
\label{BorneCorr_strip}
\end{align}
In particular, the corrector $\phi$ is not bounded.
\end{proposition}

\subsection{Corollary \ref{ThCvgDet}: A convergence rate}
As an application of Theorem \ref{PropRefinedDeter}, we obtain optimal convergence rates (up to powers of a logarithm). In particular, we prove the following corollary (which, for simplicity, is deterministic):
\begin{corollary}\label{ThCvgDet}
Let $d \geq 3$.
Assume that the deterministic coefficient field $a$ is defined by \eqref{Defa}, satisfies \eqref{Ellipticite}, is uniformly $\alpha$-Hölder continuous for a fixed exponent $\alpha \in (0,1)$, and satisfies:
\begin{align*}
\lt\|a \rt\|_{\CC^{0,\alpha}(\R^d)} \leq \kappa.
\end{align*}
Suppose that the underlying coefficient fields $a_\pm$ admit constant homogenized matrices $\overline{a}_\pm$, and that there exist generalized correctors associated with $a_\pm$ that satisfy:
\begin{equation*}
\sup_{x,y \in \R^d, |x-y| =r} \lt|\Phi_{\pm}(x)- \Phi_{\pm}(y)\rt| \leq \kappa r^{1-\nu},
\end{equation*}
for a fixed exponent $\nu \in (0,1]$.

Let~$f \in \LL^p(\R^d)$ with support inside~$\Boule(x_0,1)$, for~$p>d$.
Assume that the functions~$u^\epsilon$ and~$\overline{u}$ are the zero-mean solutions to
\begin{align}\label{DivAGrad}
\lt\{
\begin{aligned}
&-\div\left(a \left( x/\epsilon \right)\nabla u^\epsilon(x) \right)=-\div\left(\overline{a}(x)\nabla \overline{u}(x) \right)  =f(x) \quad\text{in}\quad \R^d,
\\
&\nabla u^\epsilon, \nabla \overline{u} \in \LL^2\lt(\R^d,\R^d\rt).
\end{aligned}
\rt.
\end{align}
Then, there there holds
\begin{align}
\label{Eq:ThLinfty3-1}
\lt\|u^\epsilon-\overline{u}\rt\|_{\LL^\infty\lt(\R^d\rt)} \lesssim_{d,\lambda,\kappa,\alpha,\nu,p}
\lt\{
\begin{aligned}
& \epsilon^\nu \ln\lt(2 + \epsilon^{-1}\rt) \lt\|f\rt\|_{\LL^p(\R^d)} && \quad\text{if}\quad \nu <1,
\\
& \epsilon \ln^3\lt(2 + \epsilon^{-1}\rt) \lt\|f\rt\|_{\LL^p(\R^d)} && \quad\text{if}\quad \nu=1.
\end{aligned}
\rt.
\end{align}
Moreover, if $f \in \LL^\infty\lt(\R^d\rt)$, then:
\begin{align}
\label{Eq:ThWLinfty}
\begin{split}
& \lt\|\nabla u^\epsilon - \nabla \overline{u} - \nabla\phi\lt(\frac{\cdot}{\epsilon} \rt) \cdot \overline{\nabla} \overline{u} \rt\|_{\LL^\infty(\R^d)} \\
& \quad \quad \quad \lesssim_{d,\lambda,\kappa,\alpha,\nu}
\lt\{
\begin{aligned}
&    \epsilon^\nu \ln^2(2+\epsilon^{-1}) \lt\|f \rt\|_{\LL^\infty(\R^d)} && \quad\text{if}\quad \nu <1,\\
& \epsilon \ln^5(2+\epsilon^{-1}) \lt\|f \rt\|_{\LL^\infty(\R^d)}  && \quad\text{if}\quad \nu=1.
\end{aligned}
\rt.
\end{split}
\end{align}
\end{corollary}

\begin{remark}
The assumptions of Corollary \ref{ThCvgDet} encompass Example \ref{Example_PerDef}.
\end{remark}

\begin{remark}
We assume in Corollary \ref{ThCvgDet} that the dimension $d \geq 3$ and that $f$ has compact support in order to ascertain that $f \in \HH^{-1}(\R^d)$. (Relaxations of these assumptions are possible, but we do not consider these subtleties for simplicity.)
\end{remark}

In our proof of Corollary \ref{ThCvgDet}, we make use of the generalized 2-scale expansion \eqref{2scale}. In Figure \ref{FigA} the accuracy of the generalized 2-scale expansion in the case of Example \ref{Example_PerDef} (for the coefficient drawn on the left of Figure \ref{Fig_Examples} and for a fixed smooth right-hand side $f$) is illustrated in comparison with a naive glued 2-scale expansion. For this, we define the local errors $E^{\epsilon,1}$ and $E^{\epsilon,2}$ respectively corresponding to the left-hand side of \eqref{Eq:ThWLinfty} and to the glued 2-scale expansion:
\begin{align*}
&E^{\epsilon,1}:=\nabla u^\epsilon - \nabla \overline{u} - \nabla\phi\lt(\frac{\cdot}{\epsilon} \rt) \cdot \overline{\nabla} \overline{u},
\\
&E^{\epsilon,2}:=
\lt\{
\begin{aligned}
&\lt|\nabla u^\epsilon - \lt(\Id + \nabla \phi_+ \lt(\frac{\cdot}{\epsilon}\rt) \rt) \cdot \nabla \overline{u}\rt|  && \quad\text{in}\quad \R_+ \times \R^{d-1},
\\
&\lt|\nabla u^\epsilon - \lt(\Id + \nabla \phi_- \lt(\frac{\cdot}{\epsilon}\rt) \rt) \cdot \nabla \overline{u}\rt|  && \quad\text{in}\quad \R_- \times \R^{d-1}.
\end{aligned}
\rt.
\end{align*}    
\begin{figure}
\begin{center}
\includegraphics[width=\linewidth]{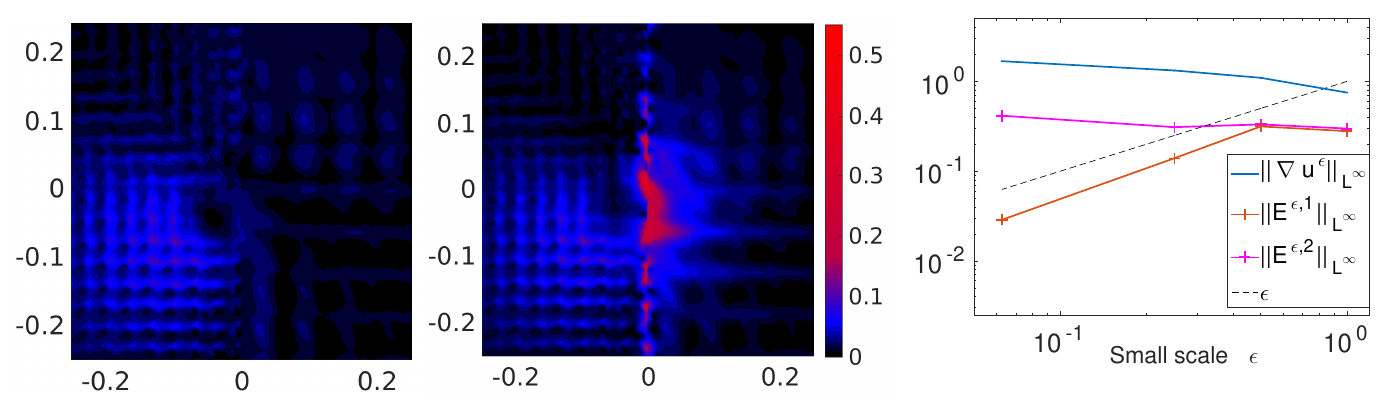}
\caption{\label{FigA} Consider a coefficient field $a$ as on the left of Figure \ref{Fig_Examples} and a smooth forcing term $f$ with compact support. On the left of this figure we have the local error $E^{\epsilon,1}$, in the middle is the local error $E^{\epsilon,2}$, and on the right we see the $\LL^\infty$-norms of the aforementioned quantities when $\epsilon \downarrow 0$.}
\end{center}  
\end{figure}
As expected, we observe that our approximation performs well uniformly on the space, whereas the glued two-scale expansion is accurate far from the interface, but useless in the vicinity of it.
While letting $\epsilon \downarrow 0$, we observe that the convergence rate is approximately linear in $\epsilon$, as predicted by Corollary \ref{ThCvgDet}.

\subsection{Further geometrical extensions}

\paragraph{Corners and boundary}
In this contribution, we only consider equations \eqref{DivAGrad} posed on the whole-space $\R^d$.
This has the beneficial consequence of avoiding the problem of boundaries.
However, this is \textit{not} only a convenient simplification.
Indeed, if we would replace $\R^d$ by a smooth bounded domain and complement the equation with homogeneous Dirichlet boundary conditions, we would face the geometrical problem of the crossing between the interface and the boundary.
There, the regularity results for discontinuous coefficients that we are using from \cite{Vogelius_2000,LiNirenberg_2003,Lorenzi_1972} do not apply.
This geometrical problem is not unrelated with the case of a non-flat interface with a corner --which might be relevant when modeling interfaces between crystals.
We intend to explore these issues in upcoming works.

From another perspective, let us emphasize that the techniques used here could improve the results of \cite{Fischer_Raithel_2017,Raithel_2017}. In particular, one could obtain almost-optimal growth rates for correctors on the half-space with homogeneous Dirichlet and Neumann boundary conditions.

\paragraph{The case of a wide interface layer}
In this article, we have assumed that the layer around the interface has a width $2L=2$ (see \eqref{Defa}).
In general, this width $L \geq 1$ might be much larger than the characteristic scale (here $1$) of the oscillations of the coefficients $a_\pm$. 
Thus, it might be useful to track the dependence in $L$ in the estimates \eqref{CorrSubNu} and \eqref{CorrSubNuSig}.

To handle this situation, we perform the rescaling $\tilde{a}(y):=a(Ly)$.
Notice that the width of the interface layer of $\tilde{a}$ is now $2$, so that we may apply our results to $\tilde{a}$.
Also, remark that the correctors $\tilde{\phi}$, $\tilde{\sigma}$, $\tilde{\phi}_\pm$, \textit{etc.} associated to $\tilde{a}$ are obtained from the correctors $\phi$, $\sigma$, $\phi_\pm$ associated to $a$ by the following rescaling: $\tilde{\phi}(y)=L^{-1} \phi(Ly)$.
Performing the arguments for Theorem \ref{PropRefinedDeter} (see Section \ref{section_6}) in this special context produces (after rescaling back) the following estimates:
\begin{align*}
&  \sup_{x,y \in \R^d, |x-y| =r} \langl \left(  \int_{Q_1}  \lt|\phi(x+z)-\phi(y+z)\rt|^2 \dd z \right)^{\frac{p}{2}} \rangl^{\frac{1}{p}} \\
& \hspace{1cm} \lesssim L \ln \lt( 2 +\frac{|r|}{L}\rt) + |r|^{1-\nu} \ln^{\lfloor \nu \rfloor}\lt( 2+ \frac{|r|}{L}\rt),
\end{align*}
and
\begin{align*}
\nonumber
& \sup_{x,y \in \R^d, |x-y| =r} \langl  \lt(\int_{\Boule(0,1)} \lt|\sigma(x + z)-\sigma(y + z) \rt|^2 \dd z \rt)^{\frac{p}{2}}  \rangl^{\frac{1}{p}}      \\
& \hspace{1cm} \lesssim L \ln^3 \lt( 2 + \frac{|r|}{L}\rt) + |r|^{1-\nu} \ln^{1+2\lfloor \nu \rfloor}\lt( 2+ \frac{|r|}{L}\rt).
\end{align*} 

\paragraph{Systems}
In all of this article, we consider a scalar equation. However, we do not use any elliptic tool specific to scalar equations (\textit{e.g.}, the maximum principle).
Therefore, our theorems may extend to the case of systems, that is replacing \eqref{DivAgrad0} by
\begin{equation*}
- \div \left(A : \nabla u \right)=f \quad \text{ in the sense of } \quad
- \sum_{i}^d \partial_i \left( \sum_{\beta=1}^m \sum_{j=1}^d A_{ij}^{\alpha\beta} \partial_j u^\beta \right)=f^{\alpha},
\end{equation*}
where the coefficient $A=\left( A_{ij}^{\alpha\beta}\right)$, for~$i, j \in [\![1,d]\!]$ and~$\alpha, \beta \in [\![1,m]\!]$, $m \in \mathbb{N}$, is elliptic in the following sense\footnote{Notice that \textit{elastic} systems do not satisfy fully \eqref{Homog_Green_EllipSys}, but a weaker version of it, where the vector-valued vectors $\xi$ in \eqref{Homog_Green_EllipSys} shall be symmetric in the sense of $\xi_i^\alpha=\xi_\alpha^i$ (in such a case, $m=d$) --see \cite[Sec.\ 2.4 p.\ 26]{Shen}.}:
\begin{align}
\label{Homog_Green_EllipSys}
\lambda |\xi|^2 \leq \sum_{i, j=1}^d \sum_{\alpha, \beta=1}^m A_{ij}^{\alpha\beta}(x) \xi^{\alpha}_i \xi^{\beta}_j \leq \left|\xi\right|^2 && \forall x \in \mathbb{R}^d, \xi=\left(\xi_i^\alpha\right) \in \mathbb{R}^{d\times m}.
\end{align}

\section{Argument for Theorem \ref{ThAL}}
\label{section_4}

In all this section, we place ourselves under the assumptions of Theorem \ref{ThAL}.

We do not directly follow the strategy of \cite{Gloria_Neukamm_Otto_2015}, but a more recent version that can be found in \cite[Th.\ 3 in Chap.\ 3]{Gloria_Notes}.
The main idea is to take advantage of the sublinearity of the (ungauged) generalized corrector \eqref{CorrSub} in oder to control the excess.
This is rephrased in the following lemma, which is an adaptation of \cite[Lem.\ 3]{Gloria_Neukamm_Otto_2015}:

\begin{lemma}\label{LemCore}
Assume that the function $u$ is $a$-harmonic in $\Boule(x_0,R)$ and that $\delta$ defined by
\begin{align}\label{Defdelta}
\delta:= \frac{1}{R} \lt( \fint_{\Boule(x_0,R)} \lt| (\phi,\sigma^{\rm u} ) - \fint_{\Boule(x_0,R)} (\phi,\sigma^{\rm u} )  \rt|^2  \rt)^{\frac{1}{2}}
\end{align}
satisfies $\delta \leq 1$. Then, there exists a constant $\epsilon = \epsilon(d,\lambda)>0$ such that, for any $0< r \leq R$, the following estimate holds:
\begin{align}\label{EstimLemCore}
\mathcal{E}(x_0,r)[u] \lesssim_{d,\lambda}  \lt( \lt(\frac{r}{R}\rt)^2 + \delta^{2\epsilon} \lt(\frac{R}{r}\rt)^{d+2} \rt) \fint_{\Boule(x_0,R)} \lt|\nabla u\rt|^2 .
\end{align}
\end{lemma}
The difference between our proof of Lemma \ref{LemCore} and \cite[Lem.\ 3]{Gloria_Neukamm_Otto_2015} is that we use the generalized $2$-scale expansion \eqref{2scale} instead of the classical one. This results in the standard algebraical identity involving the difference between the solution of the oscillating problem and the $2$-scale expansion  (see \cite[p.\ 26-27]{JKO} or \cite[(79)]{Gloria_Neukamm_Otto_2015}) being replaced by \eqref{Algebre}; and the regularity estimate for $\nabla^2 \overline{u}$  used in the original proof (that is unavailable in the case of an interface) is replaced by a corresponding estimate for $\nabla \overline{\nabla} \overline{u}$, which is given in Lemma \ref{LemRegUbar}. 

Before proceeding with the proof of Lemma \ref{LemCore}, we emphasize two technical lemmas: Lemma \ref{LemRegUbar}, which concerns the regularity of $\nabla\overline{\nabla} \overline{u}$, and Lemma \ref{LemWeight}, which provides a weighted energy estimate.
We start with Lemma \ref{LemRegUbar}, which is classical and, therefore, proved in Appendix \ref{Sec:ProofLemmas} (see also \cite[Prop. 2.1]{Vogelius_2000}):

\begin{lemma}\label{LemRegUbar}
Let $x_0 \in \R^d$ and $\rho \in (0,1)$.
Assume that $\overline{u} \in \HH^1\lt(\Boule(x_0,1)\rt)$ is $\overline{a}$-harmonic in $\Boule(x_0,1)$, where $\overline{a}$ is defined in \eqref{Defabar}.
Then, we have that
\begin{align}\label{BorneNabla2}
\displaystyle\sup_{x \in \Boule(x_0, 1 - \rho) \setminus \mathcal{I}} \rho^2 |\nabla \overline{\nabla} \overline{u}(x)|^2 + \displaystyle\sup_{x \in \Boule(x_0, 1 - \rho)} |\overline{\nabla} \overline{u}(x)|^2  \lesssim \rho^{-d} \fint_{\Boule(x_0, 1) }|\overline{\nabla} \overline{u}|^2 .
\end{align}
\end{lemma}
\noindent
Moving on to the weighted energy estimate, here is:
\begin{lemma}[Weighted energy estimate]\label{LemWeight}
Assume that $a$ satisfies \eqref{Ellipticite} and let $u \in H^1(\Boule(0,1))$ be a weak solution of 
\begin{align}
\label{DivagradG}
\lt\{
\begin{aligned}
-\div\lt(a \nabla u\rt)& = \div(g) && \quad\text{in}\quad \Boule(0,1),
\\
u&=0 && \quad\text{on}\quad \partial\Boule(0,1).
\end{aligned}
\rt.
\end{align}
Then, there exists an exponent $\beta=\beta(d,\lambda) \in (0,1)$ such that the following estimate holds:
\begin{align}
\label{WHardy}
\int_{\Boule(0,1)} \lt(1-|x|\rt)^\beta \lt|\nabla u(x)\rt|^2 \dd x \lesssim_{d,\lambda} \int_{\Boule(0,1)} \lt(1-|x|\rt)^{\beta} \lt|g(x)\rt|^2 \dd x.
\end{align}
\end{lemma}

The following proof can be found in the lecture notes in \cite[Chap.\ 3]{Gloria_Notes}. For the sake of self-containedness, we reproduce here the argument:
\begin{proof}[Proof of Lemma \ref{LemWeight}]
We will fix $\beta \in (0,1)$ at the end of the proof. First, notice that using the same argument as in Step 1 of Lemma \ref{LemRegUbar}, but with the test function $\eta^2 u$ with $\eta(x):=(1-|x|)^{\beta/2}$, we find that 
\begin{align}
& \int_{\Boule(0,1)} \lt(1-|x|\rt)^{\beta}\lt| \nabla u\rt|^2  \nonumber \\
& \quad \lesssim_{d,\lambda} \int_{\Boule(0,1)} \beta^2 \lt(1-|x|\rt)^{\beta-2}\lt| u \rt|^2 
+ \int_{\Boule(0,1)} \lt(1-|x|\rt)^{\beta} \lt| g\rt|^2.
\label{Absorb1}
\end{align}
We then invoke the following Hardy inequality \cite[Th.\ 1.6]{Necas_1965}: There exists a constant $C$ (depending only on $d$) such that, for any $\beta \in [0,1/2]$, there holds
\begin{align}\label{Ineq_Hardy}
\int_{\Boule(0,1)} \lt(1-|x|\rt)^{\beta-2}\lt| u \rt|^2 \leq C \int_{\Boule(0,1)} \lt(1-|x|\rt)^{\beta}\lt| \nabla u\rt|^2.
\end{align}
Therefore, we may set $\beta>0$ sufficiently small so that the first term on  right-hand side of \eqref{Absorb1} can be absorbed by the left-hand side.
This proves \eqref{WHardy}.
\end{proof}

We now proceed with the:
\begin{proof}[Proof of Lemma \ref{LemCore}]
For brevity, we drop the superscript ${\rm u}$ on the ungauged flux correctors $\sigma^{\rm u}$. Without loss of generality, we may assume that $4r \leq R$ and $R=1$. Furthermore, we may assume that $\phi$ and $\sigma$ are of zero mean on $\Boule(x_0,1)$.

We define the function $\overline{u} \in \HH^1(\Boule(x_0,1))$ as the Lax-Milgram solution of
\begin{align}\label{Defubar}
\lt\{
\begin{aligned}
-\div\lt(\overline{a} \nabla \overline{u}\rt)& =0  &&\quad\text{in}\quad \Boule(x_0,1),\\
\overline{u} & =u  &&\quad\text{on}\quad \partial\Boule(x_0,1),
\end{aligned}
\rt.
\end{align}
and denote the corresponding homogenization error as 
\begin{align}\label{DefW}
w:=u-\overline{u} - \eta \phi \cdot \overline{\nabla} \overline{u},
\end{align}
where $\eta \in \CC^{\infty}_\comp\lt( \Boule(x_0,1) \rt)$ will be fixed later. Notice that $w\equiv0$ on $\partial \Boule(x_0,1)$.

\paragraph{Strategy of proof}
We show \eqref{EstimLemCore} in five steps:
First, we use Lemma \ref{LemRegUbar} to  estimate the left-hand side of \eqref{EstimLemCore} by means of $\nabla \overline{u}$ and $\nabla w$. Then, the three following steps are dedicated to estimating $\nabla w$.
In Step 2, we show that $\nabla w$ satisfies an elliptic equation in divergence form.
In Step 3, we estimate a weighted energy of the right-hand side of this equation and, in Step 4, obtain an estimate for $\nabla w$ that involves only $\nabla \overline{u}$ via Lemma \ref{LemWeight}.
In Step 5, we conclude the proof by estimating $\nabla \overline{u}$ in terms of $\nabla u$.

\paragraph{Step 1: Estimate for the excess in terms of $\overline{u}$ and $w$}
Let $\beta$ be given by Lemma \ref{LemWeight}.
We claim that
\begin{equation}
\label{Step3}
\begin{aligned}
\mathcal{E}(x_0,r)[u] \lesssim& r^{-d-2} \fint_{\Boule(x_0,1)} \lt( 1 -|x-x_0|\rt)^{\beta} \lt|\nabla w\rt|^2
\\
&+\lt(r^2 +r^{-d}\delta^2 \rt) \fint_{\Boule(x_0,1)} \lt|\nabla \overline{u}\rt|^2.
\end{aligned}
\end{equation}

Indeed, we set the vector
\begin{align*}
\xi:=\fint_{\Boule(x_0,r)} \overline{\nabla} \overline{u}
\end{align*}
and the function
\begin{align}\label{Defwtilde}
\widetilde{w}:=u - \lt( P + \phi - P(x_0)\rt) \cdot \xi - \overline{u}(x_0).
\end{align}
By definition \eqref{DefExcess} of the excess, there holds:
\begin{align}\label{Excesswtilde}
\mathcal{E}(x_0,r)[u] \leq \fint_{\Boule(x_0,r)} \lt|\nabla \widetilde{w}\rt|^2.
\end{align}
Since the function $\widetilde{w}$ is $a$-harmonic in $\Boule(x_0,1)$,
the Caccioppoli estimate yields
\begin{align}
\fint_{\Boule(x_0,r)} \lt|\nabla \widetilde{w}\rt|^2 
\lesssim& r^{-2} \fint_{\Boule(x_0,2r)} \lt|\widetilde{w}\rt|^2.
\label{1}
\end{align}

Applying the triangle inequality to the definition \eqref{Defwtilde} of $\tilde{w}$, we may decompose:
\begin{align*}
\lt|  \widetilde{w} \rt| \leq \lt|u-\overline{u}-\phi\cdot\overline{\nabla}\overline{u}\rt|+\lt|\phi\rt|\lt|\xi - \overline{\nabla} \overline{u}\rt|+\lt|\overline{u}- \lt( P - P(x_0)\rt) \cdot \xi - \overline{u}(x_0)\rt|.
\end{align*}
The second and third terms on the right-hand side may then be handled by appealing to Lemma \ref{LemRegUbar} (see also \cite[Lem.\ 5.3]{Josien_InterfPer_2018} for more details):
\begin{align*}
\lt\|\xi - \overline{\nabla} \overline{u}\rt\|_{{\LL^\infty(\Boule(x_0,2r))}} 
\lesssim r \lt\| \nabla \overline{\nabla} \overline{u} \rt\|_{\LL^\infty(\Boule(x_0,1/2))}
\lesssim& r  \lt(\fint_{\Boule(x_0,1)} \lt| \nabla \overline{u} \rt|^2\rt)^{\frac{1}{2}}
\end{align*}
and
\begin{align*}
\lt\|\overline{u}- \lt( P - P(x_0)\rt) \cdot \xi - \overline{u}(x_0)\rt\|_{\LL^\infty(\Boule(x_0,2r))}
\lesssim& r^2 \lt\| \nabla \overline{\nabla} \overline{u} \rt\|_{\LL^\infty(\Boule(x_0,1/2))}
\\
\lesssim& r^2  \lt(\fint_{\Boule(x_0,1)} \lt| \nabla \overline{u} \rt|^2\rt)^{\frac{1}{2}}.
\end{align*}

Combining these observations with \eqref{1}, we find that 
\begin{align*}
&\fint_{\Boule(x_0,r)} \lt|\nabla \widetilde{w}\rt|^2\\
&\quad\lesssim r^{-2} \fint_{\Boule(x_0,2r)} \lt|u-\overline{u}-\phi \cdot \overline{\nabla} \overline{u}\rt|^2 +\lt(r^2 +\fint_{\Boule(x_0,2r)} \lt|\phi\rt|^2 \rt) \fint_{\Boule(x_0,1)} \lt|\nabla \overline{u}\rt|^2
\\
&\quad\lesssim
r^{-d-2} \fint_{\Boule(x_0,1)} \lt( 1 -|x-x_0|\rt)^{\beta-2} \lt|w\rt|^2
+\lt(r^2 +r^{-d} \delta^2 \rt) \fint_{\Boule(x_0,1)} \lt|\nabla \overline{u}\rt|^2.
\end{align*}
Notice that we have additionally used the definition \eqref{Defdelta} of $\delta$, the definition \eqref{DefW} of $w,$ and that $2r \leq 1/2$.
The Hardy inequality \eqref{Ineq_Hardy} produces
\begin{align*}
& \fint_{\Boule(x_0,r)} \lt|\nabla \widetilde{w}\rt|^2\\
& \hspace{.5cm} \lesssim
r^{-d-2} \fint_{\Boule(x_0,1)} \lt( 1 -|x-x_0|\rt)^{\beta} \lt|\nabla w\rt|^2
+\lt(r^2 +r^{-d}\delta^2 \rt) \fint_{\Boule(x_0,1)} \lt|\nabla \overline{u}\rt|^2,
\end{align*}
which, in light of \eqref{Excesswtilde}, yields \eqref{Step3}.

\paragraph{Step 2: Algebraical computation}
We show the following cut-off version of \eqref{Algebre}:
\begin{align}
-\div\lt(a\nabla w\rt)
=&\div(g) \quad\text{in}\quad \Boule(x_0, 1),
\label{Algebraical_Cutoff}
\end{align}
where
\begin{align*}
g_i:=\lt(a_{ij} \phi_k-\sigma_{ijk} \rt) \partial_j \lt( \eta \overline{\partial}_k \overline{u}\rt) + (1-\eta)\lt(a_{ij}-\overline{a}_{ij}\rt)\partial_j \overline{u}.
\end{align*}

Indeed, since $u$ and $\overline{u}$ are respectively $a$-harmonic and $\overline{a}$-harmonic in $\Boule(x_0,1)$, there holds:
\begin{align*}
-\div\lt(a\nabla w\rt)
=& - \div \lt( \overline{a} \nabla \overline{u} - a \nabla \overline{u} - a \nabla \lt( \eta \phi \cdot \overline{\nabla} \overline{u}\rt) \rt).
\end{align*}
But, on the one hand, we have:
\begin{align*}
\div \lt( \overline{a} \nabla \overline{u} - a \nabla \overline{u}\rt)
=\div\lt( \eta \lt(\overline{a}-a\rt) \cdot \nabla P  \overline{\nabla} \overline{u}\rt)+ \div\lt((1-\eta)\lt(\overline{a}-a\rt)\nabla \overline{u}\rt),
\end{align*}
and, on the other hand,
\begin{align*}
\div \lt( a \nabla \lt( \eta \phi \cdot \overline{\nabla} \overline{u}\rt) \rt)
=\partial_i \lt(a_{ij} \phi_k \partial_j\lt( \eta \overline{\partial}_k \overline{u} \rt)+ a_{ij}  \eta \overline{\partial}_k \overline{u} \partial_j \phi_k \rt).
\end{align*}
As a consequence, we find that
\begin{align*}
-\div\lt(a\nabla w\rt)
=&
-\partial_i\lt( \eta \lt(\overline{a}_{ij} \partial_j P_k-a_{ij} \partial_j (P_k + \phi_k)\rt) \overline{\partial}_k \overline{u}\rt)
\\
&- \partial_i\lt((1-\eta)\lt(\overline{a}_{ij}-a_{ij}\rt)\partial_j \overline{u}\rt)
+\partial_i \lt(a_{ij} \phi_k \partial_j\lt( \eta \overline{\partial}_k \overline{u} \rt)\rt).
\end{align*}
By the relation \eqref{Def_Pot} and the skew-symmetry condition \eqref{PotSym}, the first term on the right-hand side above can be expressed as:
\begin{align*}
\partial_i\lt( \eta \lt(\overline{a}_{ij} \partial_j P_k-a_{ij} \partial_j (P_k + \phi_k)\rt) \overline{\partial}_k \overline{u}\rt)
=&\partial_i\lt( \eta \partial_j \sigma_{jik} \overline{\partial}_k \overline{u}\rt)
\\
=&\partial_j \sigma_{jik} \partial_i \lt( \eta \overline{\partial}_k \overline{u}\rt)
\\
=&\partial_j \lt( \sigma_{jik} \partial_i \lt( \eta \overline{\partial}_k\overline{u}\rt)\rt).
\end{align*}
Whence, we obtain \eqref{Algebraical_Cutoff}.

\paragraph{Step 3: Estimate for $g$}
We claim that the following inequality holds:
\begin{align}\label{Claim_g}
\int_{\Boule(x_0,1)} \lt(1-|x-x_0|\rt)^\beta \lt|g\rt|^2 \lesssim \lt( \rho^\beta + \rho^{-d-2} \delta^2 \rt) \int_{\Boule(x_0,1)} \lt|\nabla \overline{u} \rt|^2,
\end{align}
where we assume now that $\eta=1$ in $\Boule(x_0,1-2\rho)$, $\Supp(\eta) \subset \Boule(x_0,1-\rho)$, and $\lt|\nabla \eta\rt| \lesssim \rho^{-1}$.

By definition, we have that
\begin{align*}
&\int_{\Boule(x_0,1)} (1-|x-x_0|)^\beta \lt|g_i\rt|^2\\
&\qquad\lesssim
\int_{\Boule(x_0,1-\rho)} (1-|x-x_0|)^\beta \lt|\lt(\sigma_{ijk}- a_{ij} \phi_k \rt) \partial_j \lt( \eta \overline{\partial}_k \overline{u}\rt) \rt|^2
\\
&\qquad\quad+\int_{\Boule(x_0,1-\rho)} (1-|x-x_0|)^\beta \lt(1-\eta\rt)^2 \lt|\lt(\overline{a}_{ij}-a_{ij}\rt)\partial_j \overline{u}\rt|^2.
\end{align*}
The first term on the right-hand side is handled by appealing to the Hölder inequality, whereas we recall that $(1-\eta)$ is supported in $\Boule(x_0,1-\rho) \backslash \Boule(x_0,1-2\rho)$ for the second term.
Thus:
\begin{align*}
& \int_{\Boule(x_0,1)} (1-|x-x_0|)^\beta \lt|g_i\rt|^2\\
& \lesssim \left(\lt\| \nabla \overline{\nabla} \overline{u} \rt\|^2_{\LL^\infty(\Boule(x_0,1-\rho))} + \rho^{-2}  \lt\| \overline{\nabla} \overline{u} \rt\|^2_{\LL^\infty(\Boule(x_0,1-\rho))} \right)
\int_{\Boule(x_0,1)} \lt|(\phi,\sigma)\rt|^2
\\
&\quad \quad +\rho^\beta \int_{\Boule(x_0,1)} \lt|\nabla \overline{u}\rt|^2.
\end{align*}
Thanks to Lemma \ref{LemRegUbar} and definition \eqref{Defdelta} of $\delta$, this yields \eqref{Claim_g}.

\paragraph{Step 4: Estimate for $\nabla w$}
The aim of this step is to estimate the first term on the right-hand side of \eqref{Step3}.
Applying Lemma \ref{LemWeight} to $w$, which satisfies \eqref{Algebraical_Cutoff}, we deduce that
\begin{align*}
\int_{\Boule(x_0,1)} \lt( 1 -|x-x_0|\rt)^\beta \lt|\nabla w\rt|^2 \lesssim \int_{\Boule(x_0,1)}\lt( 1 -|x-x_0|\rt)^\beta \lt|g\rt|^2.
\end{align*}
Then, invoking \eqref{Claim_g}, this yields
\begin{align*}
\int_{\Boule(x_0,1)} \lt( 1 -|x-x_0|\rt)^\beta \lt|\nabla w\rt|^2 
\lesssim \lt( \rho^\beta + \rho^{-d-2} \delta^2 \rt) \int_{\Boule(x_0,1)} \lt|\nabla \overline{u} \rt|^2.
\end{align*}
We optimize the above estimate by setting $\rho:=\delta^{\frac{2}{d+2+\beta}}$. Thus,
\begin{align}\label{EstimW}
\int_{\Boule(x_0,1)} \lt( 1 -|x-x_0|\rt)^\beta \lt|\nabla w\rt|^2 
\lesssim \delta^{2\epsilon}\int_{\Boule(x_0,1)} \lt|\nabla \overline{u} \rt|^2
\end{align}
for $\epsilon:=\beta/(d+2+\beta)$.

\paragraph{Step 5: Conclusion of the proof of \eqref{EstimLemCore}}
By \eqref{Step3} and \eqref{EstimW}, we have:
\begin{align}\label{Step5}
\mathcal{E}(x_0,r)[u] \lesssim& \lt(r^{-d-2} \delta^{2\epsilon} +r^2 +r^{-d}\delta^2 \rt) \fint_{\Boule(x_0,1)} \lt|\nabla \overline{u}\rt|^2.
\end{align}
Since $\delta$, $r$, and $\epsilon$ are smaller than $1$, we deduce that $\delta^2 \leq r^{-2}\delta^{2\epsilon}$.
Moreover, by \eqref{Defubar}, there holds:
\begin{align*}
\fint_{\Boule(x_0,1)} \lt|\nabla \overline{u}\rt|^2 \lesssim \fint_{\Boule(x_0,1)} \lt|\nabla u\rt|^2.
\end{align*}
As a consequence, \eqref{EstimLemCore} is obtained from \eqref{Step5} by restoring the scale $R$.
\end{proof}

With Lemma \ref{LemCore} in-hand, we are now in a position to give the argument for Theorem \ref{ThAL}. Since it closely follows  \cite[Proof of Lem.\ 3]{Gloria_Neukamm_Otto_2015}, we do not detail every step. 

\begin{proof}[Proof of Theorem \ref{ThAL}]
The proof proceeds in two steps: In Step 1, we turn the estimate \eqref{EstimLemCore} into an iterative excess decay (in the spirit of \cite{AvellanedaLin}).
Then, in Step 2, we show the non-degeneracy property \eqref{NonDegen}.
We do not show \eqref{LargeLip} since it is a corollary of \eqref{EstimLemCore} and \eqref{NonDegen}, the proof of which can be found in \cite[Proof of Lem.\ 3]{Gloria_Neukamm_Otto_2015}. The latter relies on a dyadic argument to control the minimizer $\xi_{r'}$ of $\mathcal{E}(x_0,r')$, for $r' \in [r,R]$, from $\xi_R$ down to $\xi_r$.

\paragraph{Step 1: Iterative excess decay}
Let $R \in [r^\star,r_{\max}]$ and $\theta  \in (0,1)$.
Since $u - \lt(P + \phi \rt) \cdot \xi$ is $a$-harmonic for any $\xi \in \R^d$, we deduce from Lemma \ref{LemCore} that
\begin{align}\label{Excess1}
\mathcal{E}(x_0,\theta R)[u] \leq C \lt( \theta^2 + \delta^{2\epsilon} \theta^{-d-2} \rt) \mathcal{E}(x_0,R)[u].
\end{align}
We set $\theta$ and then $\delta$ (in that order) sufficiently small so that
\begin{align*}
C\theta^2 \leq \frac{\theta^{2\alpha}}{2} \quad\text{and}\quad C \delta^{2\epsilon}\theta^{-d-2} \leq \frac{\theta^{2\alpha}}{2}.
\end{align*}
Then \eqref{Excess1} reads:
\begin{align}\label{Excess2}
\mathcal{E}(x_0,\theta R)[u] \leq \theta^{2\alpha} \mathcal{E}(x_0,R)[u].
\end{align}
Let $r \in [r^\star,R]$. Then, there exists $n$ such that $R\theta^{n+1} \leq r \leq R\theta^{n}$. By iterating \eqref{Excess2}, we deduce that
\begin{align*}
\mathcal{E}(x_0,r)[u]\leq \theta^{-d} \mathcal{E}(x_0,\theta^{n}R)[u] \leq  \theta^{2n\alpha-d} \mathcal{E}(x_0,R)[u] \leq C \lt(\frac{r}{R}\rt)^{2\alpha}\mathcal{E}(x_0,R)[u],
\end{align*}
where $C = \theta^{-d-2\alpha}$. This proves \eqref{Borne_Excess} (up to setting $\delta$ smaller).

\paragraph{Step 2: Non-degeneracy of the excess operator}
The upper bound is provided by the Caccioppoli inequality:
\begin{align*}
\fint_{\Boule(x_0,r)} \lt|\nabla (P+ \phi)  \cdot\xi \rt|^2
\lesssim
r^{-2}\fint_{\Boule(x_0,2r)} \lt|\lt(P-P(x_0)+\phi\rt) \cdot \xi  \rt|^2
\lesssim (1+\delta)^2 \lt|\xi\rt|^2
\end{align*}
for any $\xi \in \R^d$. The lower bound is a consequence of the Poincaré inequality:
\begin{align*}
\fint_{\Boule(x_0,r)} \lt|\nabla (P+ \phi)  \cdot\xi \rt|^2
\gtrsim&
r^{-2}\fint_{\Boule(x_0,r)} \lt|(P-P(x_0)) \cdot \xi + \phi  \cdot\xi \rt|^2
\\
\gtrsim&
r^{-2}\fint_{\Boule(x_0,r)} \lt|(P-P(x_0))\cdot \xi \rt|^2-r^{-2}\fint_{\Boule(x_0,r)} \lt|\phi  \cdot\xi \rt|^2
\\
\gtrsim&
\lt(1-C\delta^2\rt) \lt|\xi\rt|^2.
\end{align*}
This proves \eqref{NonDegen} (up to setting $\delta$ smaller).
\end{proof}

\section{Argument for Proposition \ref{PropExistCorr}}
\label{sec:Prop_1}

We begin by introducing the geometric objects needed in our argument and constructing local ungauged generalized correctors.  In the second subsection, we then show that we can inductively use the large-scale Lipschitz estimate \eqref{LargeLip} from Theorem \ref{ThAL} to, given a local ungauged generalized corrector at some scale, obtain another local ungauged generalized corrector on a larger scale that still satisfies the same sublinearity property. In the final subsection we show that we can pass to the limit in this procedure to obtain a global ungauged generalized corrector that is strictly sublinear. Again, for brevity, here we usually drop the superscripts $``\textrm{u}"$ on ungauged flux correctors.

\subsection{Construction of local ungauged generalized correctors}\label{SecEquatConst}

We use a geometrical construction involving dyadic balls and a narrow layer along the interface (see Figure \ref{Fig_Decompose}) according to which we decompose the interface layer corrections $\tilde{\phi}$ and $\tilde{\sigma}$ that we have already introduced in Section \ref{state_correctors}.

\paragraph{Geometrical setting:}
Fix $x_0 \in \mathbb{R}^d$. For $r_0 \geq 1$ to be fixed later, we set a dyadic sequence $r_m:=2^m r_0$ for $m \geq 0$ and $r_{m}=0$ if $m<0$. In order to define the interface layer, we introduce the function
\begin{align}\label{DefLpm}
L(r):= r^{-2\nu/3 +1}+ 2,
\end{align}
which, as will see in our proof, actually turns out to be optimal for our estimates. We now let $S_m$ be functions of $x$ depending only on $x^\perp$ and satisfying the following constraints:
\begin{align}
\label{DefSm}
\lt\{
\begin{aligned}
&\lt\{x \in \R^d: S_m(x)=1\rt\} \subseteq \lt\{x \in \R^d : \lt|x^\perp\rt| \leq \frac{L(r_m)}{2}\rt\},
\\
&\Supp(S_m) \subseteq \lt\{x \in \R^d  :  \lt|x^\perp\rt| \leq L(r_m)\rt\},
\\
&
\lt\|\nabla S_m \rt\|_{\LL^\infty(\R^d)} \lesssim \lt(L(r_m)\rt)^{-1}, \textrm{ and } \| S_m \|_{\LL^\infty(\R^d)} \leq 1  .
\end{aligned}
\rt.
\end{align}
We remark that the $``+2"$ in \eqref{DefLpm} is included to ensure that $S_m \equiv 1$ on $[-1,1] \times \mathbb{R}^{d-1}$. We also define the cut-off functions $\eta_m$ associated with the nested balls
\begin{align}
\label{DefetaM}
\begin{split}
\Boule\lt(x_0,r_{m}-L\lt(r_{m}\rt)\rt) & \subseteq \lt\{x \in \R^d : \eta_m(x)=1\rt\} \subseteq \Supp(\eta_m) \subseteq \Boule\lt(x_0,r_{m}\rt)
\end{split}
\end{align}
such that $|\nabla \eta_m| \lesssim |L( r_{m})|^{-1}$. To compliment this definition, we use the convention that a ball of negative radius is the empty set. Finally, we introduce the notation
\begin{align*}
\chi_m:=S_m (\eta_m-\eta_{m-1} ).
\end{align*}
and 
\begin{align*}
\chi:=\sum_{m=0}^{+\infty} \chi_m
\end{align*}
and notice that $\chi$ equals $1$ in a layer along the interface and vanishes far from it. This layer is thin along the interface, but is shaped like a trumpet and becomes wider far from the origin. 
\begin{figure}[h]
\begin{center}
\includegraphics{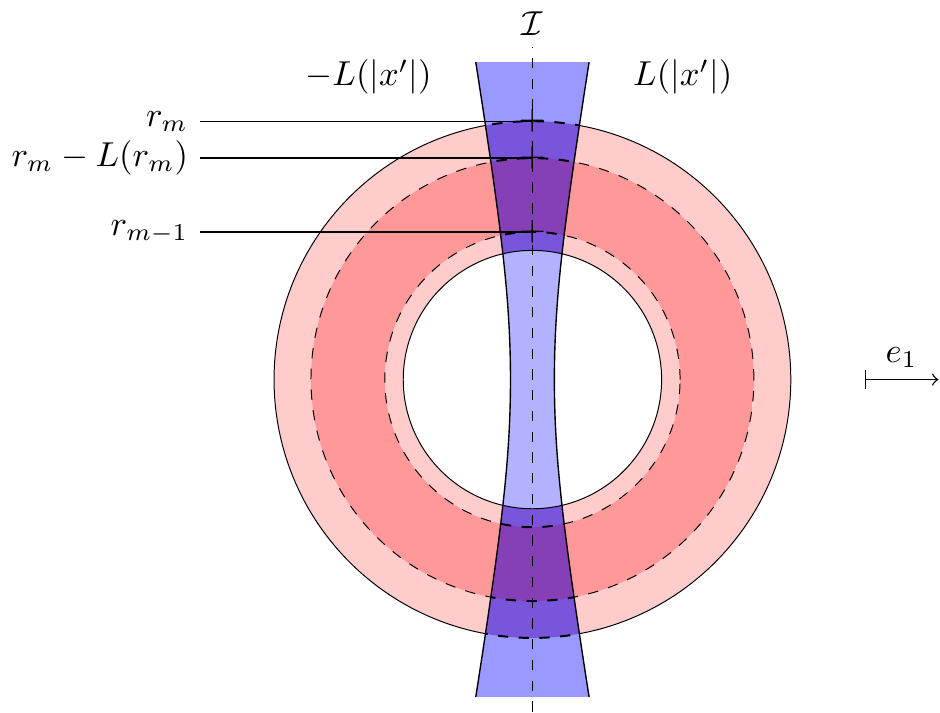}
\end{center}
\caption{In this figure, we set $x_0  =0$. Here the blue area represents the support of the cut-off function $\chi$ along the interface, \textit{i.e.} the ``trumpet'', and the pink areas denotes the support of $\eta_m- \eta_{m-1}$.  Notice that $\nabla(\eta_m- \eta_{m-1})$ is supported in the lighter pink area and $\eta_m- \eta_{m-1} =1$ in the darker pink area.}
\label{Fig_Decompose}
\end{figure}

\paragraph{Local ungauged generalized corrector $\lt(\phi^M, \sigma^{M,{\rm u}}\rt)$:}
In this subsection, for simplicity, we set $x_0 = 0$ (the cases $x_0 \neq 0$ can be dealt with similarly). For $m \geq 0$, we first derive an equation for the $m^{\textrm{th}}$ contribution to the interface layer correction $\tilde{\phi}^m$. For each $M \geq0$ we then sum these contributions up to the scale $r_02^M$ in order to obtain the local corrector $\phi^M$. In particular, we have the following result: 

\begin{lemma}\label{LemConstW}
For every $m\geq 0$ and $k \in \[ 1 , d \]$, there exists $\tilde{\phi}_k^m \in \HH^1_{\rm{loc}}(\R^d)$ that is a weak solution of 
\begin{align}\label{Defwm}
-\div\lt( a \nabla \tilde{\phi}^m_k \rt) 
=
\partial_i(g_{ik}^m) \quad\text{in}\quad \R^d,
\end{align}
for the vector $g^m_{k}$ defined by
\begin{align}
g^m_{ik}:=\lt(\chi_m \lt(a_{ij}-\overline{a}_{ij}\rt)
+\lt(a_{il} \check{\phi}_j-\check{\sigma}_{ilj}\rt)\partial_l \lt((\eta_m-\eta_{m-1})-\chi_m\rt)\rt)\partial_j P_k,
\label{Defg}
\end{align}
such that
\begin{align}
\int_{\R^d} \lt|\nabla \tilde{\phi}^m\rt|^2
\lesssim_{d,\lambda} \int_{\R^d} \lt|g^m\rt|^2.
\label{Estim1}
\end{align}

Furthermore, the function $\phi_k^M \in \HH^1_{\rm{loc}}(\R^d)$ defined by
\begin{align}\label{Ansatz2}
\phi_k^M:=\sum_{m=0}^M\tilde{\phi}^m_k+\lt(\eta_M-\sum_{m=0}^M\chi_m\rt)\check{\phi}\cdot\nabla P_k
\end{align}
is a ``local corrector'' in the sense that it satisfies
\begin{align}\label{Loc_CorrPhi}
-\div\lt( a\lt( \nabla \phi_k^M + \eta_M \nabla P_k \rt) -\eta_M \overline{a} \nabla P_k \rt)=0 \quad\text{in}\quad \R^d,
\end{align}
whence 
\begin{align}\label{PropPhiM}
-\div\lt( a \lt( \nabla \phi^M_k + \nabla P_k \rt) \rt) =0 \quad\text{in}\quad \Boule(0,r_{M-1}).
\end{align}
\end{lemma}

\begin{remark}[Discontinuities through the interface]\label{RkConv} In order to make sense of $g^m$ in \eqref{Defg}, we notice that the functions $\check{\phi}$ and $\check{\sigma}$ are (generically) discontinuous through the interface $\mathcal{I}$. This is, however, not a problem as $(1-\chi)$ and $\nabla \chi$ vanish in a neighborhood of the interface and, therefore, possible singularities of distributions multiplied by these functions in this neighborhood are not important. To illustrate this, we remark that apriori the formal product $\partial_l \check{\phi}_j \partial_j P_k$ has no mathematical significance for $x^\perp=0$  (even in the sense of distributions), but the expression 
\begin{align*}
\lt((1-\chi) \partial_l \check{\phi}_j \partial_j P_k\rt)(x):=
\lt\{
\begin{aligned}
&0  &&\quad\text{if}\quad x \notin \Supp\lt((1-\chi)\rt),
\\
&\lt((1-\chi) \partial_l \check{\phi}_j \partial_j P_k\rt)(x) &&\quad\text{if}\quad x \in  \Supp\lt((1-\chi)\rt)
\end{aligned}
\rt.
\end{align*}
is well-defined. From now on, we will use this and all analogous conventions without further notice.
\end{remark}

\begin{proof}[Proof of Lemma \ref{LemConstW}]
The matter of the existence of $\tilde{\phi}_k^m$ and the energy estimate \eqref{Estim1} can be settled using a standard Lax-Milgram argument. 

By the definition \eqref{DefetaM} of $\eta_M$, \eqref{PropPhiM} is an obvious consequence of \eqref{Loc_CorrPhi} and \eqref{DefPj}. Then, by  \eqref{Defwm} and \eqref{Defg}, there holds:
\begin{align}
\label{Eqphim}
&-\div\lt( a \nabla \lt( \sum_{m=0}^M \tilde{\phi}^m_k \rt)\rt)
\\
& \nonumber =
\partial_i\lt(\sum_{m=0}^M\chi_m \lt(a_{ij}-\overline{a}_{ij}\rt)\partial_j P_k + \lt(a_{il} \check{\phi}_j-\check{\sigma}_{ilj}\rt)\partial_l \lt( \eta_M-\sum_{m=0}^M\chi_m\rt)\partial_j P_k\rt).
\end{align}
Notice that since
\begin{align*}
\eta_M-\sum_{m=0}^M\chi_m=\sum_{m=0}^M \lt(\eta_m-\eta_{m-1}\rt)\lt(1-S_m\rt),
\end{align*}
the right-hand side of \eqref{Eqphim} is well-defined by Remark \ref{RkConv}. 

To finish, we show that with the ansatz \eqref{Ansatz2}, the relations \eqref{Loc_CorrPhi} and \eqref{Eqphim} are equivalent. By plugging the ansatz \eqref{Ansatz2} in \eqref{Loc_CorrPhi}, we obtain:
\begin{align*}
&-\partial_i\lt(a_{ij} \sum_{m=0}^M \partial_j \tilde{\phi}^M_k\rt)
\\
&=\partial_i \Bigg( \lt(\eta_M-\sum_{m=0}^M\chi_m\rt) a_{il}\partial_l \check{\phi}_j\partial_j P_k+ a_{il}\partial_l\lt(\eta_M-\sum_{m=0}^M\chi_m\rt)\check{\phi}_j\partial_j P_k
\\
&~~~~~+\eta_M\lt(a_{ij}-\overline{a}_{ij}\rt)\partial_j P_k \Bigg)
\\
&=\partial_i \Bigg( \lt(\eta_M-\sum_{m=0}^M\chi_m\rt) \lt(a_{il}\partial_l \check{\phi}_j +a_{ij}-\overline{a}_{ij} \rt)\partial_j P_k
\\
&~~~~~~+ a_{il}\partial_l\lt(\eta_M-\sum_{m=0}^M\chi_m\rt)\check{\phi}_j\partial_j P_k
+\sum_{m=0}^M\chi_m\lt(a_{ij}-\overline{a}_{ij}\rt)\partial_jP_k \Bigg).
\end{align*}
By the definition and skew-symmetry of $\check{\sigma}$, the first term on the right-hand side reads
\begin{align*}
\partial_i \lt( \theta^M \lt(a_{il}\partial_l \check{\phi}_j +a_{ij}-\overline{a}_{ij} \rt)\partial_j P_k\rt)
&=-\partial_i \lt( \theta^M \partial_l\check{\sigma}_{lij} \partial_j P_k\rt)
\\
&=- \partial_l\check{\sigma}_{lij}\partial_i \lt( \theta^M \partial_j P_k\rt)
\\
&=-\partial_l\lt(\check{\sigma}_{lij}\partial_i \lt( \theta^M \partial_j P_k\rt)\rt)
\\
&=-\partial_l\lt(\check{\sigma}_{lij}\partial_i \theta^M \partial_j P_k\rt),
\end{align*}
where we used $\theta^M:=\eta_M-\sum_{m=0}^M\chi_m$ for brevity.
Whence \eqref{Eqphim} is established.
\end{proof}

With the local correctors $\phi^M$ from Lemma \ref{LemConstW}, we can now build local ungauged flux correctors:

\begin{lemma}\label{LemConstS}
Let $g^m$, $\tilde{\phi}^m$ and $\phi^M$ be defined as in Lemma \ref{LemConstW}. For $M \geq 0$ and $j,k \in  \[ 1 , d \]$, assume that there exists a function $\widetilde{N}^{M}_{jk} \in \HH^1_{\rm{loc}}(\R^d)$ that satisfies 
\begin{align}\label{DefN}
\Delta \widetilde{N}^M_{jk}=-\sum_{m=0}^M \lt(g^m_{jk}+ a_{jl} \partial_l \tilde{\phi}^m_k\rt) \quad\text{in}\quad \R^d.
\end{align}
If we, furthermore, assume $\nabla \widetilde{N}^M$ to be strictly sublinear, then $\sigma^{M, {\rm u}}$ defined by
\begin{align}
\sigma^{M, {\rm u}}_{ijk}=&\lt(\partial_i \widetilde{N}^{M}_{jk} - \partial_j \widetilde{N}^{M}_{ik}\rt)+\lt(\eta_M-\sum_{m=0}^M\chi_m\rt) \check{\sigma}_{ijl} \partial_l P_k
\label{DefSM}
\end{align}
is a ``local ungauged flux corrector associated with $\phi^M$''. In other words, it satisfies \eqref{PotSym} in $\R^d$ and is a weak solution of 
\begin{align}\label{EqFondaW2}
\partial_i \sigma^{M, {\rm u}}_{ijk} 
=\eta_M  \overline{a}_{jl} \partial_l P_k-a_{jl}\lt( \partial_l \phi_k^M + \eta_M  \partial_l P_k \rt)\quad\text{in}\quad \R^d,
\end{align}
which implies that
\begin{align}\label{EqFondaW3}
\partial_i \sigma^{M, {\rm u}}_{ijk}=
\overline{a}_{jl}  \partial_l P_k -a_{jl} \lt(\partial_l P_k + \partial_l \phi^M_k \rt) \quad\text{in}\quad \Boule\lt(0,r_{M-1}\rt).
\end{align}
\end{lemma}

\begin{proof}[Proof of Lemma \ref{LemConstS}] 
Notice that that $\sigma^{M,{\rm u}}$ defined by \eqref{DefSM} naturally satisfies \eqref{PotSym}. The only part that remains to be checked is \eqref{EqFondaW2}. For this, we recall from Lemma \ref{LemConstW} that $g^m+a \nabla \tilde{\phi}^m$ is divergence-free, which, when combined with \eqref{DefN}, implies that 
$\Delta \partial_j \widetilde{N}^M_{jk}=0$.
As a consequence, by the first order Liouville principle for harmonic functions, $ \partial_j \widetilde{N}^M_{jk}$ is constant.
We then have that
\begin{align}
\partial_{ii} \widetilde{N}^M_{jk} - \partial_j\partial_i\widetilde{N}^M_{ik}
=-\sum_{m=0}^M \lt(g^m_{jk}+a_{jl}\partial_l \tilde{\phi}^m_k\rt).
\label{3}
\end{align}
Whence, by the definitions of $\sigma^{M, {\rm u}}$ and $\check{\sigma}$, and as a consequence of \eqref{3}, we obtain that 
\begin{align*}
\partial_i \sigma^{M, {\rm u}}_{ijk} 
=&
\partial_{ii} \widetilde{N}^M_{jk} - \partial_j\partial_i\widetilde{N}^M_{ik}+\partial_i\lt(\lt(\eta_M-\sum_{m=0}^M\chi_m\rt) \check{\sigma}_{ijl} \partial_l P_k\rt)
\\
=&\sum_{m=0}^M \lt(-g^m_{jk}-a_{jl}\partial_l \tilde{\phi}^m_k\rt)
+\partial_i\lt(\eta_M-\sum_{m=0}^M\chi_m\rt) \check{\sigma}_{ijl} \partial_l P_k 
\\
&+\lt(\eta_M-\sum_{m=0}^M\chi_m\rt) \lt(\overline{a}_{jl} - a_{jl} - a_{jh}\partial_h\check{\phi}_l\rt) \partial_l P_k.
\end{align*}
By definition \eqref{Defg} of $g$, \eqref{Ansatz2}, and the antisymmetry of $\check{\sigma}$:
\begin{align*}
\partial_i \sigma^{M, {\rm u}}_{ijk} 
=&\lt(-\sum_{m=0}^M\chi_m \lt(a_{jl}-\overline{a}_{jl}\rt)-\lt(a_{ji} \check{\phi}_l-\check{\sigma}_{jil}\rt)\partial_i \lt(\eta_M-\sum_{m=0}^M\chi_m\rt)\rt)\partial_l P_k
\\
&-a_{ji}\partial_i \lt(\phi^M_k -\lt(\eta_M-\sum_{m=0}^M\chi_m\rt)\check{\phi}_l\partial_lP_k\rt)
\\
&-\check{\sigma}_{jil}\partial_i\lt(\eta_M-\sum_{m=0}^M\chi_m\rt)  \partial_l P_k
\\
&+\lt(\eta_M-\sum_{m=0}^M\chi_m\rt) \lt(\overline{a}_{jl} - a_{jl} - a_{jh}\partial_h\check{\phi}_l\rt) \partial_l P_k
\\
=&\eta_M \lt( \lt(\overline{a}_{jl} - a_{jl}\rt)-a_{jl}\partial_l \phi^M_k \rt)  \partial_l P_k,
\end{align*}
which establishes \eqref{EqFondaW2}.
\end{proof}

\subsection{Inductive use of large-scale Lipschitz regularity} \label{SecProofConst}

\begin{lemma}\label{LemRecurrent}
Let $\alpha=1/2$, $M \geq 0$, $\nu \in (0,1]$, and $x_0 \in \R^d$. Assume that there exists a local ungauged generalized corrector $\lt(\phi^M,\sigma^{M,{\rm u}}\rt)$ on $\Boule(x_0,r_{M-1})$ that simultaneously satisfies the growth conditions \eqref{Condition_delta}, and \eqref{CorrSub} for $r^\star=r_0 \geq 1$ and $r_{\max}=r_{M-1}$ and for $\delta:=\delta(d,\lambda,1/2)$ as in Theorem \ref{ThAL}.

Under these assumptions, there exists a local corrector $\phi^{M+1}$, namely a solution of  \eqref{Loc_CorrPhi}, and a local ungauged flux corrector $\sigma^{M+1,\rm u}$, namely a solution of \eqref{EqFondaW2} satisfying  \eqref{PotSym}, such that for any $r\geq r_0$ the following estimates hold:
\begin{align}
\frac{1}{r}\lt(\fint_{\Boule(x_0,r)} \lt|\phi^{M+1}-\fint_{\Boule(x_0,r)}\phi^{M+1} \rt|^2\rt)^{\frac{1}{2}} \leq C(d,\lambda) r^{-\nu/3},
\label{Estimphim1bis}
\end{align}
and
\begin{align}
&\frac{1}{r}\lt(\fint_{\Boule(x_0,r)} \lt|\sigma^{M+1,\rm u} - \fint_{\Boule(x_0,r)} \sigma^{M+1,\rm u} \rt|^2 \rt)^{\frac{1}{2}}
\leq C(d,\lambda) r^{-\nu/3},
\label{Estimsigmam1}
\end{align}
for a constant $C(d,\lambda)$ only depending on $d$ and $\lambda$.
\end{lemma}

\begin{proof}
Using the objects that we have defined in Lemmas \ref{LemConstW} and \ref{LemConstS}, we proceed in three steps: We first establish an estimate on $\nabla \tilde{\phi}^m$ by appealing to Theorem \ref{ThAL}, from which we deduce \eqref{Estimphim1bis} in a second step. Then, to finish, we show \eqref{Estimsigmam1} by using the result of the first step and basic facts about harmonic functions.

\paragraph{Step 1: Energy estimate} We first show that, for any $r\geq r_0$ and $m\leq M+1$, the functions $\tilde{\phi}^m$ from Lemma \ref{LemConstW} satisfy the following estimate
\begin{align}
\label{Estimphim0}
\fint_{\Boule(x_0,r)} \lt|\nabla \tilde{\phi}^m\rt|^2  \lesssim_{d,\lambda} \min\lt(1, \lt( \frac{r_m}{r} \rt)^d \rt) r_m^{-2\nu/3} .
\end{align}

To obtain \eqref{Estimphim0}, we use \eqref{Defg} to write:
\begin{align}
\int_{\R^d} \lt|g^m\rt|^2
\lesssim &
\int_{\R^d}\chi_m^2 + \int_{\R^d}\lt|(\check{\phi},\check{\sigma})\rt|^2 \lt|\nabla\lt((\eta_m-\eta_{m-1})-\chi_m\rt) \rt|^2.
\label{Estim2}
\end{align}
Since the function $\chi_m$ has a localized support, \textit{i.e.} we have that 
\begin{align*}
\Supp(\chi_m) \subset \Boule(x_0,r_{m}) \backslash \Boule(x_0,r_{m-2}) \cap \lt\{ x \in \R^d \, : \,  \lt|x^\perp\rt| \leq L(r_{m}) \rt\},
\end{align*}
the first integral on the right-hand side of \eqref{Estim2} is bounded as
\begin{align*}
\int_{\R^d}\chi_m^2  \lesssim r_m^{d-1} L(r_m).
\end{align*}
For the second integral we use the decomposition
\begin{align*}
\nabla \lt(\eta_m-\eta_{m-1} -\chi_m\rt) = (1-S_m)\lt(\nabla \eta_m - \nabla \eta_{m-1}\rt) -  \lt(\eta_{m}-\eta_{m-1}\rt)\nabla S_m
\end{align*}
and properties of the cut-off functions defined in Section \ref{SecEquatConst} to estimate
\begin{align*}
&\int_{\R^d}\lt|(\check{\phi},\check{\sigma})\rt|^2 \lt|\nabla\lt((\eta_m-\eta_{m-1})-\chi_m\rt) \rt|^2
\\
&\lesssim
\int_{\Boule(x_0,r_m)} \lt|(\check{\phi},\check{\sigma})\rt|^2 L(r_m)^{-2}
+\int_{\Boule(x_0,r_{m})} \lt|(\check{\phi},\check{\sigma})\rt|^2 L(r_{m-1})^{-2}
\\
&\lesssim  r_m^{d+2(1-\nu)} L(r_m)^{-2} + r_{m}^{d+2(1-\nu)} L(r_{m-1})^{-2}.
\end{align*}
Notice that we have also used the growth condition \eqref{Condition_delta}.

As a consequence,
\begin{align}\label{Num:1010}
\int_{\R^d} \lt|g^m\rt|^2 
\lesssim
r_m^{d}\lt( \frac{L(r_m)}{r_m} + \frac{r_m^{2(1-\nu)}}{L(r_m)^2}\rt) + r_{m-1}^{d}\frac{r_{m-1}^{2(1-\nu)}}{L(r_{m-1})^2}.
\end{align}
Recalling that $r_m = 2^mr_0$, \eqref{DefLpm} appears to be the optimal choice and plugging it in  \eqref{Num:1010} yields:
\begin{align}
\label{Optichoice}
\begin{split}
\int_{\R^d} \lt|g^m\rt|^2 
&
\lesssim r_m^d r_m^{-2\nu/3}.
\end{split}
\end{align}
Therefore, from \eqref{Estim1} and \eqref{Optichoice}, we obtain \eqref{Estimphim0} for $r \geq r_{m-2}$.

For the case $r \leq r_{m-2}$, we can use Theorem \ref{ThAL} for which, up to the scale $r_02^{M-1}$, it is sufficient to have a local ungauged generalized corrector in $\Boule(x_0,r_{M-1})$. Since $\tilde{\phi}^m$ is $a$-harmonic in $\Boule(x_0,r_{m-2}) \subset \Boule(x_0,r_{M-1})$, this entails
\begin{align*}
\fint_{\Boule(x_0,r)} \lt|\nabla \tilde{\phi}^m\rt|^2 \lesssim  \fint_{\Boule(x_0,r_{m-2})} \lt|\nabla \tilde{\phi}^m\rt|^2 \lesssim r_m^{-2\nu/3} ,
\end{align*}
where we have used \eqref{Estimphim0} for $r = r_{m-2}$. This finishes the argument for \eqref{Estimphim0} for all $r \geq r_0$. 

\paragraph{Step 2: Argument for \eqref{Estimphim1bis}}
Once \eqref{Estimphim0} is established, applying the Poincaré-Wirtinger inequality yields
\begin{align*}
\frac{1}{r} \lt(\fint_{\Boule(x_0,r)} \lt| \tilde{\phi}^m-\fint_{\Boule(x_0,r)}\tilde{\phi}^m\rt|^2\rt)^{\frac{1}{2}} & \lesssim \min\lt(1, \lt( \frac{r_m}{r} \rt)^{d/2} \rt) r_m^{-\nu/3}.
\end{align*}
Whence, by definition \eqref{Ansatz2} combined with the triangle inequality, \eqref{Condition_delta}, and recalling  $r_m=2^mr_0$, we obtain:
\begin{align}
\label{Estimphim1}
\begin{split}
&\frac{1}{r}\lt(\fint_{\Boule(x_0,r)} \lt|\phi^{M+1}-\fint_{\Boule(x_0,r)}\phi^{M+1} \rt|^2\rt)^{\frac{1}{2}} \\
&\lesssim  r^{-\nu}+ \sum_{m=0}^{M+1} \min\lt(1, \lt( \frac{r_m}{r} \rt)^{d/2} \rt)r_m^{-\nu/3} \lesssim r^{-\nu/3 },      
\end{split}
\end{align}
which produces \eqref{Estimphim1bis}.

\paragraph{Step 3: Argument for \eqref{Estimsigmam1}}      
Since the right-hand term of \eqref{DefN} is divergence-free, we can rewrite it as
\begin{align*}
\Delta \widetilde{N}^{M+1}_{jk}
=&-\partial_i\lt(\lt(x - x_0\rt)\cdot e_j \sum_{m=0}^{M+1} \lt(g^m_{ik}+a_{il}\partial_l \tilde{\phi}^m_k\rt)\rt).
\end{align*}

We would now like to invoke Lemma \ref{LemDelta}. To do this we first notice that, for $r \geq 1$, we have:
\begin{align*}
& \fint_{\Boule(x_0,r)} \lt|\lt(x - x_0\rt)\cdot e_j \sum_{m=0}^{M+1} \lt(g^m_{ik}+a_{il}\partial_l \tilde{\phi}^m_k\rt)\rt|^2 \lesssim r \sum_{m=0}^{M+1}  \left( \fint_{\Boule(x_0,r)} \lt|g^m\rt|^2+\lt|\nabla \tilde{\phi}^m\rt|^2 \right).
\end{align*}
By applying \eqref{Estimphim0} this becomes:
\begin{align*}
& r \sum_{m=0}^{M+1} \lt( \fint_{\Boule(x_0,r)} \lt|\nabla \tilde{\phi}^m\rt|^2 \rt)^{\frac{1}{2}} \lesssim  r \sum_{m=0}^{M+1}  \min\lt(1, \lt( \frac{r_m}{r} \rt)^{d/2} \rt) r_m^{-\nu/3}  \lesssim  r^{1-\nu/3}.
\end{align*}
On the other hand, since $g_m=0$ outside of $\Boule(x_0,r_{m})\backslash \Boule(x_0,r_{m-2})$, \eqref{Optichoice} implies that
\begin{align*}
r \sum_{m=0}^{M+1} \lt( \fint_{\Boule(x_0,r)} \lt|g^m\rt|^2 \rt)^{\frac{1}{2}}
&\lesssim 
r \sum_{m=0}^{\min({M+1}, \lceil r \rceil+2)} r^{-d/2} \lt(\int_{\R^d} \lt|g^m\rt|^2\rt)^{\frac{1}{2}}
\\
& \lesssim
r \sum_{m=0}^{\min({M+1},\lceil  r \rceil+2)} \lt(\frac{r_m}{r}\rt)^{d/2} r_m^{-\nu/3} \lesssim r^{1-\nu/3}.
\end{align*}
Therefore, we obtain:
\begin{align*}
&\lt( \fint_{\Boule(x_0,r)} \lt|\lt(x - x_0\rt)\cdot e_j \sum_{m=0}^{M+1} \lt(g^m_{ik}+a_{il}\partial_l \tilde{\phi}^m_k\rt)\rt|^2 \rt)^{\frac{1}{2}}
\lesssim r^{1-\nu/3}.
\end{align*}
As a consequence, Lemma \ref{LemDelta} produces a solution $\widetilde{N}^{M+1}$ to \eqref{DefN}.  Moreover, \eqref{ee2} implies that, for any $r \geq 1$, there holds:
\begin{align}
&r^{-1}\lt(\fint_{\Boule(x_0,r)} \lt|\nabla \widetilde{N}^{M+1} - \fint_{\Boule(x_0,r)} \nabla \widetilde{N}^{M+1}\rt|^2  \rt)^{\frac{1}{2}}
\lesssim  r^{-\nu/3}.
\label{ee12}
\end{align}

Finally, we define $\sigma^{M+1,\rm u}$ by \eqref{DefSM}. Then, since \eqref{DefN} is satisfied and \eqref{ee12} implies that $\nabla \widetilde{N}^{M+1}$ is strictly sublinear, by Lemma \ref{LemConstS} we know that $\sigma^{M+1,\rm u}$ is a local ungauged flux corrector and solves \eqref{EqFondaW2}.
Last, summing up \eqref{ee12} and the estimate satisfied by $\check{\sigma}$ yields \eqref{Estimsigmam1}.
\end{proof}

\subsection{Proof of Proposition  \ref{PropExistCorr}}

We  are now in a position to proceed with the:

\begin{proof}[Proof of Proposition \ref{PropExistCorr}]
As in the previous works \cite{Fischer_Raithel_2017,Fischer_Otto_2015}, the proof is done by induction.

\paragraph{Step 1: Induction}Note first that the local ungauged generalized corrector $\lt(\phi^0,\sigma^0\rt)$ satisfies \eqref{Estimphim1bis} and \eqref{Estimsigmam1}, as a straightforward corollary of the energy estimate --possibly at the price of taking a larger uniform constant. We then set $r_0 \geq 1$ such that
\begin{align}
\label{defnr0}
C(d,\lambda)r_0^{-\nu/3} \leq \delta(d,\lambda,1/2),
\end{align}
where $C(d,\lambda)$ refers to the common constant of \eqref{Estimphim1bis} and \eqref{Estimsigmam1} and $\delta(d,\lambda,1/2)$ is fixed in Theorem \ref{ThAL}. 

Next, assume that the  local ungauged generalized corrector $\lt(\phi^M,\sigma^M\rt)$ satisfies \eqref{Estimphim1bis} and \eqref{Estimsigmam1} for a given $M \in \NN$. Therefore, by \eqref{defnr0}, this local generalized corrector also satisfies the growth condition \eqref{CorrSub} for $r^\star=r_0$, $r_{\max}=r_{M-1}$, and $\delta:=\delta(d,\lambda,1/2)$. Whence, applying Lemma \ref{LemRecurrent}, we obtain that $\lt(\phi^{M+1},\sigma^{M+1}\rt)$ also satisfy \eqref{Estimphim1bis} and \eqref{Estimsigmam1}.

As a conclusion of the inductive proof, the local correctors $\lt(\phi^M,\sigma^M\rt)$ satisfy \eqref{Estimphim1bis} and \eqref{Estimsigmam1} for any $M \in \NN$.

\paragraph{Step 2: Limit $M \rightarrow +\infty$}
By a compactness argument in $\LL^2_\loc(\R^d)$, the following convergences hold up to a subsequence:
\begin{align*}
\lt(\phi^M,\sigma^M\rt) \underset{M\rightarrow +\infty}{\rightharpoonup} \lt(\phi,\sigma^{\rm u} \rt) 
\quad\text{and}\quad 
\nabla \phi^M \underset{M\rightarrow +\infty}{\rightharpoonup} \nabla \phi \quad\text{in}\quad \LL^2_\loc(\R^d).
\end{align*}
By taking the limit of the weak formulations of \eqref{PropPhiM} and \eqref{EqFondaW3}, we deduce that $\phi$ and $\sigma^{\rm u} $ respectively satisfy \eqref{DefCorr} and \eqref{Def_Pot}. 
Also, the generalized ungauged correctors $\lt(\phi,\sigma^{\rm u} \rt)$ inherit \eqref{Estimphim1bis} and \eqref{Estimsigmam1}. 
As a consequence, we have established
\begin{align}
\label{SublinSSop}
\frac{1}{r} \lt(\fint_{\Boule(x_0,r)} \lt| \lt(\phi,\sigma^{\rm u} \rt) - \fint_{\Boule(x_0,r)} \lt(\phi,\sigma^{\rm u} \rt) \rt|^2  \rt)^{\frac{1}{2}} \leq \tilde{\kappa} r^{-\tilde{\nu}} \quad \text{for any} \quad r  \geq 1.
\end{align}

\paragraph{Step 3: Post-processing \eqref{SublinSSop}}
We use a standard argument to get \eqref{SublinSSop_2} from \eqref{SublinSSop}.
Applying the Cauchy-Schwarz inequality and \eqref{SublinSSop}, there holds:
\begin{align*}
\lt|\fint_{\Boule(x_0,r)} \lt(\phi,\sigma^{\rm u} \rt)-\fint_{\Boule(x_0,2r)} \lt(\phi,\sigma^{\rm u} \rt)\rt|
&\leq
\lt(\fint_{\Boule(x_0,r)} \lt|\lt(\phi,\sigma^{\rm u} \rt)-\fint_{\Boule(x_0,2r)} \lt(\phi,\sigma^{\rm u} \rt) \rt|^2\rt)^{\frac{1}{2}}
\\
&\leq \tilde{\kappa} 2^{d/2}r^{1-\tilde{\nu}},
\end{align*}
for any $r\geq 1$. Thus, by a dyadic argument, for $2^{n-1}<r\leq 2^n$, we get
\begin{align*}
&\lt|\fint_{\Boule(x_0,r)} \lt(\phi,\sigma^{\rm u} \rt)-\fint_{\Boule(x_0,1)} \lt(\phi,\sigma^{\rm u} \rt)\rt|
\\
&\leq \sum_{j=1}^n \lt|\fint_{\Boule(x_0,2^j)} \lt(\phi,\sigma^{\rm u} \rt)-\fint_{\Boule(x_0,2^{j-1})} \lt(\phi,\sigma^{\rm u} \rt)\rt|
\\
&\quad+ \lt|\fint_{\Boule(x_0,r)} \lt(\phi,\sigma^{\rm u} \rt)-\fint_{\Boule(x_0,2^n)} \lt(\phi,\sigma^{\rm u} \rt)\rt|
\\
&\leq \tilde{\kappa} 2^{d/2} \sum_{j=1}^{n+1} 2^{j(1-\tilde{\nu})}
\leq \tilde{\kappa} \frac{2^{d/2+1}}{2^{1-\tilde{\nu}}-1} r^{1-\tilde{\nu}}.
\end{align*}
As a conclusion, combining the above inequality and \eqref{SublinSSop}, we obtain \eqref{SublinSSop_2}.
\end{proof}

\section{Argument for Theorem \ref{PropRefinedDeter}}
\label{section_6}

As already discussed in Section \ref{SecResult}, we first collect some peripheral results and then combine these in Section \ref{subsec:Thm1.3_proof}, which contains the proof of Theorem \ref{PropRefinedDeter}. 

\subsection{Enforcing the gauge on the flux corrector}
\label{SecUnique0}

Here, we assume that we are given an ungauged generalized corrector satisfying a quantified sublinearity estimate (\textit{e.g.} the output of Proposition \ref{PropExistCorr}) and use Lemma \ref{LemDelta} to obtain a generalized corrector with the same sublinearity properties. We also show that strictly sublinear generalized correctors are unique up to the addition of a random constant. Here is the result of this subsection:

\begin{lemma}
\label{enforce_gauge_1}
Let $\overline{a}_+$ and $\overline{a}_- \in \R^{d\times d}$ be fixed, and let $\langle \cdot \rangle$ be an ensemble on $\Omega$ such that, $\langl\cdot\rangl$-almost surely, $a$ admits $\overline{a}$ defined by \eqref{Defabar} as its homogenized matrix (this ensemble does not necessarily satisfy the assumptions of Section \ref{ensemble}).

Assume that $(\phi, \sigma^{\rm u} )$ is an ungauged generalized corrector such that for fixed $x_0 \in \mathbb{R}^d$ and $p\in [2,\infty)$, and for any $r\geq1$ we have that
\begin{align}
\frac{1}{r} \left\langle  \left( \fint_{B(x_0, r)} \left| \phi(x) - \fint_{B(x_0, 1)} \phi  \right|^2  \, \textrm{d} x\right)^{\frac{p}{2}} \right\rangle^{\frac{1}{p}} \leq \kappa \ln^{\beta}(2+r) r^{-\nu} \label{corrector_sublinear_1}\\
\textrm{and} \quad \frac{1}{r} \left\langle\left( \fint_{B(x_0, r)} \left| \sigma^{\rm u} (x) - \fint_{B(x_0, 1)} \sigma^{\rm u}   \right|^2 \, \textrm{d} x \right)^{\frac{p}{2}}\right\rangle^{\frac{1}{p}} \leq \kappa \ln^{\tilde{\beta}}(2+r) r^{-\nu} \label{flux_corrector_sublinear_1}
\end{align}
for $\nu \in (0,1]$, $\beta\geq0$, $\tilde{\beta}\geq0$, and $\kappa>0$. Under these conditions, $\phi$ is $\langle \cdot \rangle$-almost surely unique up to the addition of a constant in the class of strictly sublinear correctors. Furthermore, $\langl\cdot\rangl$-almost surely, there exists a strictly sublinear flux corrector $\sigma$ that is unique up to the addition of a constant and that satisfies
\begin{align}
\label{flux_corrector_sublinear_2}
\begin{split}
\frac{1}{r} \left\langle  \left( \fint_{B(x_0, r)} \left| \sigma(x) - \fint_{B(x_0, 1)} \sigma  \right|^2  \, \textrm{d} x\right)^{\frac{p}{2}}\right\rangle^{\frac{1}{p}} \\
\lesssim_{d, \nu, \tilde{\beta}} \begin{cases}
\kappa \ln^{\tilde{\beta}}(2+r) r^{-\nu} & \textrm{if} \quad \nu <1,\\
\kappa \ln^{\tilde{\beta}+1} (2+r)r^{-\nu} & \textrm{if} \quad \nu =1,
\end{cases}
\end{split}
\end{align}
for any $r\geq 1$.
\end{lemma}

\begin{remark}
\label{4} An easy consequence of Proposition \ref{PropExistCorr} and Lemma \ref{enforce_gauge_1} is that for an ensemble $\langl\cdot\rangl$ on $\Omega$ satisfying the conditions of Section \ref{ensemble}, there $\langle \cdot \rangle$-almost surely exist generalized correctors $(\phi, \sigma)$ and $(\phi^{\star}, \sigma^{\star})$ respectively associated to $a$ and $a^{\star}$ that are strictly sublinear and, therefore, unique (up to the addition of random constants).
\end{remark}

\begin{proof}[Proof of Lemma \ref{enforce_gauge_1}]
We first focus on $\phi$.
By the Markov inequality and Borel-Cantelli lemma (using the same arguments as in Step 2 Lemma \ref{LemDelta}), $\phi$ is $\langle \cdot \rangle$-almost surely strictly sublinear.
Now, assume that, given a realization $a$, there are two strictly sublinear solutions to \eqref{DefCorr}.
By definition, their difference $u$ satisfies
\begin{align}
-\div\lt(a \nabla u\rt) = 0 \quad\text{in}\quad \R^d.
\end{align}
Yet, by Theorem \ref{ThAL}, there exists $r^\star>1$ which is $\langl\cdot\rangl$-surely finite, such that, for any $R\geq r\geq r^\star$, there holds
\begin{align*}
\fint_{\Boule(x_0,r)} \lt|\nabla u\rt|^2 \lesssim  \fint_{\Boule(x_0,R)} \lt|\nabla u\rt|^2.
\end{align*}
Therefore, using the Caccioppoli estimate and the strict sublinearity of $u$, we get
\begin{align*}
& \lt(\fint_{\Boule(x_0,r)} \lt|\nabla u\rt|^2  \rt)^{\frac{1}{2}}
\lesssim R^{-1} 
\lt(\fint_{\Boule(x_0,2R)} \lt| u-\fint_{\Boule(x_0,2R)} u  \rt|^2  \rt)^{\frac{1}{2}} \underset{R \uparrow \infty}{\longrightarrow} 0.
\end{align*}
Since $r \geq r^\star$ is arbitrary, we deduce that $\nabla u=0$.
Therefore, $u$ is constant in the space, $\langl\cdot\rangl$-almost surely.
As a consequence, $\phi$ is $\langle \cdot \rangle$-almost surely unique up to the addition of a random constant.

Next, we apply Lemma \ref{LemDelta} with $f = \sigma_{ijk}^{\rm u} - \fint_{B(x_0, 1)}  \sigma_{ijk}^{\rm u}$, which we may do since \eqref{ee0} is satisfied.  We obtain a solution $N_{jk}$ to \eqref{Def_N1} such that $\nabla N_{jk}$ satisfies \eqref{SpecialCase} and $\sigma$ as defined in \eqref{Def2_Pot}, \textit{i.e.} $\sigma_{ijk}:= \partial_i N_{jk} - \partial_j N_{ik}$, satisfies \eqref{flux_corrector_sublinear_2}. The $\langle \cdot \rangle$-almost sure uniqueness (up to the addition of constants) of $\sigma$ follows from the uniqueness of the $N_{jk}$ up to the addition of affine functions (shown in Lemma \ref{LemDelta}).
\end{proof}

\subsection{Control of the stochastic moments of $r^\star$}

We now bound the stochastic moments of the minimal radius $r^\star(x)$ for $x \in \mathbb{R}^d$.

\begin{lemma}
\label{moments}
Let the assumptions of Section \ref{ensemble} hold and $\Phi(a) :=(\phi, \phi^*, \sigma, \sigma^*)$ be composed of strictly sublinear generalized correctors associated to $a$ and $a^*$ respectively. Given $\delta_0>0$ and $x \in \R^d$, we define $r^\star(x)\geq 1$ as the minimal radius such that
\begin{align}
\label{DefRstar}
\sup_{r \geq r^\star(x)} \frac{1}{r} \lt( \fint_{\Boule(x,r)} \lt|\Phi - \fint_{\Boule(x,r)}\Phi  \rt|^2 \rt)^{\frac{1}{2}} \leq \delta_0.
\end{align}
Then, for any $p \in [1,\infty)$, the $p$-moment of $r^\star(x)$ is bounded as
\begin{align}\label{EspRstar}
\lt\langle \lt|r^\star(x)\rt|^p \rt\rangle^{\frac{1}{p}} \lesssim_{ d,\lambda, \delta_0, \nu, \nu_0, p}  c_{p/\nu_0}^{1/\nu_0}
\end{align}
for $\nu_0 < \nu$.
\end{lemma}

\begin{proof} Notice first that the generalized corrector $\Phi=(\phi,\phi^*,\sigma,\sigma^*)$ and thus the minimal radius $r^{\star}(x)$ are $\langle \cdot \rangle$-almost surely well-defined thanks to Remark \ref{4}. For brevity, we use the convention
\begin{align}
\label{sublinear_glued}
\check{\delta}(x_0,r) := \frac{1}{r} \lt(\int_{\Boule(x_0,r)} \lt|\Phi_\pm\rt|^2 \rt)^{\frac{1}{2}}
\end{align}
for $x_0 \in \R^d$ and $r>0$.
Without loss of generality, we assume that $p$ is large (the general case being deduced from the H\"older inequality).

\paragraph{Step 1: Post-processing \eqref{CorrSubDef}} 
If $p \geq 2$, by the Cauchy-Schwarz inequality, the Bochner inequality and \eqref{CorrSubDef}, we have, for any $x, y \in \R^d$,
\begin{equation}
\label{Num:1011}
\begin{aligned}
&\langl \lt(\int_{\Q(0,1)} \lt| \Phi_\pm(y+z) - \fint_{\Q(x,1)} \Phi_\pm\rt|^2 \dd z\rt)^{\frac{p}{2}} \rangl^{\frac{1}{p}}
\\
&\leq \langl \lt(\int_{\Q(0,1)} \int_{\Q(0,1)} \lt| \Phi_\pm(y+z) - \Phi_\pm(x+z')\rt|^2 \dd z' \dd z\rt)^{\frac{p}{2}} \rangl^{\frac{1}{p}}
\\
&\leq \langl \lt( \int_{\Q(0,1)} \int_{\Q(0,2)} \lt| \Phi_\pm(y+z) - \Phi_\pm(x+z+z'')\rt|^2\dd z'' \dd z \rt)^{\frac{p}{2}} \rangl^{\frac{1}{p}}
\\
&\leq \lt(\int_{\Q(0,2)} \langl \lt(\fint_{\Q(0,1)} \lt| \Phi_\pm(y+z) - \Phi_\pm(x+z+z'')\rt|^2\dd z\rt)^{\frac{p}{2}} \rangl^{\frac{2}{p}} \dd z''\rt)^{\frac{1}{2}}
\\
&\lesssim c_p \lt( 1+|y-x|\rt)^{1-\nu}.
\end{aligned}
\end{equation}
In the sequel, we anchor $\Phi_\pm$ at $x$, in the sense of $\int_{\Q(x,1)} \Phi_\pm=0$.
Whence, by the the above inequality combined with the Bochner estimate, we deduce that
\begin{equation}\label{A_montrer}
\langl \check{\delta}(x,r)^p \rangl^{\frac{1}{p}} \lesssim_{d} c_p r^{-\nu}.
\end{equation}

\paragraph{Step 2: Estimate on the stochastic moments of $\Phi$}
We set $0<\nu_0<\nu_1<\nu$, $q \geq 3$ and
\begin{align}\label{defepsilon}
\epsilon^{-1}:= 1+ \sup_{r \geq 1} \lt(r^{\nu_1} \check{\delta}(x,r) \rt)^{1/\nu_1},
\end{align}
which, by definition, ensures that there holds:
\begin{align}
\label{star1}
\check{\delta}(x,r) \leq (\epsilon r)^{-\nu_1}
\end{align}
for any $r \geq 1$.
We then notice that, by a simple scaling argument, the generalized corrector associated to $a_{\pm}^{\epsilon}:= a_{\pm}\lt(\frac{\cdot}{\epsilon}\rt)$ is given by $\Phi^\epsilon_\pm:=\epsilon \Phi_\pm\lt(\frac{\cdot}{\epsilon}\rt)$. This means that 
\begin{align*}
\frac{1}{r} \lt( \fint_{\Boule(x,r)} \lt| \Phi^\epsilon_\pm \rt|^2 \rt)^{\frac{1}{2}}
=&\frac{\epsilon}{r} \lt( \fint_{\Boule\lt(x,\epsilon^{-1}r\rt)} \lt| \Phi_\pm\rt|^2 \rt)^{\frac{1}{2}}
\overset{\eqref{star1}}{\leq} r^{-\nu_1},
\end{align*}
which allows us to apply Proposition \ref{PropExistCorr} to obtain an \textit{ungauged} generalized corrector $\Phi^{\epsilon,{\rm u}}=\lt(\phi,\sigma^{\rm u}\rt)$ associated with the coefficient field\footnote{By definition, the interface layer of $a_\epsilon$ has a width $2\epsilon<2$.} $a^{\epsilon}:= a\lt(\frac{\cdot}{\epsilon}\rt)$. Rescaling in the opposite direction as before, we then notice that $\Phi^{\rm u}:=\epsilon^{-1} \Phi^{\epsilon, \rm u } \lt(\epsilon \, \cdot\rt)$ is a generalized ungauged corrector associated to the original coefficient field $a$.
By \eqref{SublinSSop_2} applied to $\Phi^{\epsilon,{\rm u}}$ we obtain:
\begin{align}\label{star2}
\frac{1}{r} \lt( \fint_{\Boule(x,r)} \lt| \Phi^{\rm u} - \fint_{\Boule(x,1)} \Phi^{\rm u} \rt|^2 \rt)^{\frac{1}{2}}
\leq \kappa \epsilon^{-\nu_1/3} r^{-\nu_1/3},
\end{align}
for any $r \geq \epsilon^{-1}$. 

As a consequence, there holds
\begin{align}\label{star2_1}
\langl\lt( \fint_{\Boule(x,r)} \lt| \Phi^{\rm u} - \fint_{\Boule(x,1)} \Phi^{\rm u} \rt|^2 \rt)^{\frac{q}{2}}\rangl^{\frac{1}{q}}
\leq \kappa \langl \epsilon^{-q\nu_1/3} \rangl^{\frac{1}{q}} r^{1-\nu_1/3}.
\end{align}
Furthermore, by definition \eqref{defepsilon} of $\epsilon$, by a dyadic argument, by the Bochner estimate and by \eqref{A_montrer}, we obtain
\begin{align}\label{star2_2}
\begin{aligned}
\langl \epsilon^{-q\nu_1/3} \rangl^{\frac{1}{q}} &\lesssim 1 +  \langl \lt( \sup_{r \geq 1} r^{\nu_1} \check{\delta}(x,r) \rt)^{\frac{q}{3}}\rangl^{\frac{1}{q}} 
\\
&\lesssim 1 + \langl \lt(\sum_{j=0}^{\infty} 2^{j \nu_1} \check{\delta}(x,2^j) \rt)^{\frac{q}{3}}\rangl^{\frac{1}{q}}
\\
&\lesssim 1 + \lt(\sum_{j=0}^{\infty} 2^{j \nu_1} 2^{-j \nu} c_{q/3}\rt)^{\frac{1}{3}}
\lesssim c_{q/3}^{\frac{1}{3}}.
\end{aligned}
\end{align}
Then, combining \eqref{star2_1}, \eqref{star2_2} and Lemma \ref{enforce_gauge_1} implies that the generalized corrector $\Phi=(\phi,\phi^*,\sigma,\sigma^*)$ satisfies:
\begin{equation}
\langl\delta(x,r)^{q}\rangl^{\frac{1}{q}}
\lesssim c_{q/3}^{\frac{1}{3}} r^{-\nu_1/3},
\label{EstimLpPhi}
\end{equation}
where we denote
\begin{equation*}
\delta(x,r):=\lt( \fint_{\Boule(x,r)} \lt| \Phi - \fint_{\Boule(x,r)} \Phi \rt|^2 \rt)^{\frac{1}{2}}.
\end{equation*}

\paragraph{Step 3: Conclusion}
By a similar argument as in Step 2, we know that
\begin{align*}
r^\star(x) \leq 1 + \sup_{r \geq 1} \lt( r^{\frac{\nu_0}{3}} \frac{\delta(x,r)}{\delta_0} \rt)^{\frac{3}{\nu_0}}.
\end{align*}
As a consequence of \eqref{EstimLpPhi}, using a dyadic decomposition and the Bochner estimate as in \eqref{star2_2}, there holds:
\begin{align*}
\langl |r^\star(x)|^p \rangl^{\frac{1}{p}}
&\lesssim 1 + \delta_0^{-\frac{3}{\nu_0}} \langl\sup_{r \geq 1} \lt(  r^{\frac{\nu_0}{3}} \delta(x,r) \rt)^{\frac{3p}{\nu_0}} \rangl^{\frac{1}{p}}
\\
&\lesssim 1 + \delta_0^{-\frac{3}{\nu_0}} \lt(\sum_{j=0}^{\infty} 2^{\frac{j \nu_0}{3}} \langl\lt( \delta(x,2^j) \rt)^{\frac{3p}{\nu_0}} \rangl^{\frac{\nu_0}{3p}} \rt)^{\frac{3}{\nu_0}}
\\
&\lesssim 1 + \delta_0^{-\frac{3}{\nu_0}} c_{p/\nu_0}^{1/\nu_0} \lt(\sum_{j=0}^{\infty} 2^{\frac{j \nu_0}{3}}  2^{-\frac{j \nu_1}{3}} \rt)^{\frac{3}{\nu_0}}
\lesssim c_{p/\nu_0}^{1/\nu_0}.
\end{align*}
This concludes the proof of Lemma \ref{moments}.
\end{proof}

\subsection{Estimate for the Green's function}

We now use the uniform control of the $p^{\rm th}$ stochastic moments of $r^{\star}(x)$ for $x \in \mathbb{R}^d$ from the previous subsection in order to bound the $p^{\rm th}$ stochastic moments of local  $L^2$-averages of the second mixed derivatives of the heterogeneous Green's function.  In particular, we find:

\begin{lemma}
\label{LemGreenDet}
Assume that $d \geq 2$. Let the assumptions of Section \ref{ensemble} hold. Then, the mixed  second derivatives of the Green's function $G$ associated with the operator $-\div\lt(a\cdot\nabla\rt)$ in $\R^d$ satisfy:
\begin{align}
\label{annealedE}
\begin{split}
& \langl\lt(\fint_{\Boule(x,1)} \fint_{\Boule(y,1)} \lt|\nabla_x  \nabla_y G(x',y') \rt|^2 \dd y' \dd x'\rt)^{\frac{p}{2}} \rangl^{\frac{1}{p}}\\ &\quad \quad \quad \lesssim_{d,\lambda, \nu, \nu_0, p} c^{d/\nu_0}_{dp/\nu_0}  |x-y|^{-d} \quad \quad \quad   \textrm{for} \quad |x-y| \geq 3,
\end{split}
\end{align}
\end{lemma}
\noindent for any $\nu_0 < \nu$.
\begin{remark}
If the coefficient field $a$ is uniformly Hölder continuous in $\R^d$, estimate \eqref{annealedE} can be upgraded to a pointwise estimate for any $x, y \notin \mathcal{I}$ by invoking \cite[Th. 1.1]{LiNirenberg_2003}.
\end{remark}

\begin{proof}[Proof of Lemma \ref{LemGreenDet}]
Invoking the result \cite[Th.\ 1]{BellaGiunti_2018} (and \cite[Cor.\ 1]{BellaGiunti_2018} for $d=2$) we obtain: For a well-chosen constant $\delta>0$, depending on $d$ and $\lambda$, and $r^\star(x) \geq 1$, the minimal radius associated with the condition
\begin{align*}
\sup_{r>r^\star(x)} \frac{1}{r} \lt(\fint_{\Boule(x,r)} \lt|\Phi -  \fint_{\Boule(x,r)} \Phi \rt|^2\rt)^{\frac{1}{2}} \leq \delta,
\end{align*}
the mixed derivatives of the Green's function $G$ satisfy
\begin{align}
\label{Green3}
\lt(\fint_{\Boule(x,1)} \fint_{\Boule(y,1)} \lt|\nabla_x  \nabla_y G(x',y') \rt|^2 \dd y' \dd x'\rt)^{\frac{1}{2}} &\lesssim_{d,\lambda} \lt(\frac{\lt(r^\star(x)r^\star(y)\rt)^{\frac{1}{2}}}{|x-y|}\rt)^{d}.
\end{align}
Hence, \eqref{annealedE} is a direct consequence of \eqref{Green3} and  Lemma \ref{moments}, which yields:
\begin{align*}
\langl  \lt(r^\star(x)r^\star(y)\rt)^{dp/2}  \rangl^{\frac{1}{p}} \lesssim  \left( \langle r^\star (x)^{dp} \rangle^{1/dp} \langle r^\star (y)^{dp} \rangle^{1/dp}   \right)^{d/2} \lesssim_{d,\lambda, \nu, \nu_0, p, \delta}  c^{d/\nu_0}_{dp/\nu_0}.
\end{align*}
\end{proof}

\subsection{Proof of Theorem \ref{PropRefinedDeter}}
\label{subsec:Thm1.3_proof}
Equipped with Lemmas \ref{enforce_gauge_1}, \ref{moments} and \ref{LemGreenDet}, we are in a position to proceed with the:

\begin{proof}[Proof of Theorem \ref{PropRefinedDeter}] Throughout our argument we fix $p \in [2,\infty)$ and use the notation $``\lesssim"$ to denote $``\lesssim_{d, \lambda, \nu, \nu_0, p}"$.

\paragraph{Strategy of proof} 
By Remark \ref{4}, we already have the $\langl\cdot\rangl$-almost sure existence and the uniqueness up to a random constant of a strictly sublinear generalized corrector $\lt(\phi,\sigma\rt)$. Therefore, the following proof is devoted to establishing the improved sublinearity estimates \eqref{CorrSubNu} and \eqref{CorrSubNuSig}.
We use a similar construction as that in Section \ref{sec:Prop_1}, but replacing the use of energy estimates with that of the Green's function estimates provided by Lemma \ref{LemGreenDet}.

More precisely, we set a smooth function  $\chi(x)$ only depending on $x^\perp$ such that
\begin{align*}
[-2,2] \times \R^{d-1} \subseteq \lt\{x : \chi(x)=1\rt\} \subseteq \Supp(\chi) \subseteq [-3,3] \times \R^{d-1}.
\end{align*}
Then, imposing the anchoring relation $\fint_{\Q(x_0,1)} \Phi_{\pm} = 0$ for $x_0 \in \mathbb{R}^d$, we propose a decomposition of $\phi$ in the spirit of Proposition \ref{PropExistCorr}:
\begin{align}\label{Defw2}
\phi_k=(1-\chi) \check{\phi}_j[x_0] \partial_j P_k + \tilde{\phi}_k[x_0],
\end{align}
where the dependence of $\check{\phi}[x_0]$, $\tilde{\phi}[x_0]$, and $\check{\sigma}[x_0]$ in the anchoring is made explicit.
Moreover, we define an ungauged corrector $\sigma^{\rm u} [x_0]$ by:
\begin{equation}
\label{Defsigmau}
\lt\{
\begin{aligned}
\sigma^{\rm u} _{ijk}[x_0]&=(1-\chi) \check{\sigma}_{ijl}[x_0] \partial_l P_k +\partial_i \widetilde{N}_{jk}[x_0]-\partial_j\widetilde{N}_{ik}[x_0],
\\
\Delta \widetilde{N}_{jk}[x_0]&=-g_{jk}[x_0]-a_{jl}\partial_l\tilde{\phi}_k[x_0],
\\
g_{ik}[x_0]&=\lt(\chi \lt(a_{ij}-\overline{a}_{ij}\rt)
-\lt(a_{il} \check{\phi}_j[x_0]-\check{\sigma}_{ilj}[x_0]\rt)\partial_l \chi\rt)\partial_j P_k.
\end{aligned}
\rt.
\end{equation}
Then, we prove estimates on $\tilde{\phi}[x_0]$ and $\sigma^{\rm u}[x_0]$ that may be transfered to $\phi$ and $\sigma$, either directly using \eqref{Defw2} or indirectly through Lemma \ref{enforce_gauge_1} (note that the uniqueness of $\lt(\phi,\sigma\rt)$ plays a fundamental role).

Our argument has five steps: In Step 1, we use Lemma \ref{LemGreenDet} to show that
\begin{align}\label{Claim0}
\langl \lt(\int_{\Boule(x_0,1)} \lt|\nabla \tilde{\phi}[x_0] \rt|^2  \rt)^{\frac{p}{2}} \rangl^{\frac{1}{p}} \lesssim  c^{1+d/\nu_0}_{2dp/\nu_0} \lt(1+\lt|x_0^\perp\rt|\rt)^{-\nu}
\end{align}
for any $p < \infty$.
In Step 2,  using the uniqueness $\nabla \phi$ and changing the anchoring point $x_0$, we extend the above estimate as follows:
\begin{align}
\label{Satisw}
\begin{split}
&\langl \lt(\int_{\Boule(x,1)} \lt|\nabla \tilde{\phi}[x_0] \rt|^2 \rt)^{\frac{p}{2}} \rangl^{\frac{1}{p}}
\\
&\qquad\lesssim
c^{1+d/\nu_0}_{2dp/\nu_0} \lt\{
\begin{aligned}
&\lt(1+\lt|x^\perp\rt|\rt)^{-\nu}  && \quad\text{if}\quad \lt|x^\perp\rt| \geq 4,
\\
&\lt(1+|x-x_0|\rt)^{1-\nu}  && \quad\text{if}\quad \lt|x^\perp\rt| \leq 4
\end{aligned}
\rt. 
\end{split}
\end{align}
for any $x \in \R^d$.  From this, we deduce in Step 3 that the corrector $\phi$ satisfies \eqref{CorrSubNu}.
In Step 4, defining $x_0'=(0,x_0^\parallel)$, we show that the ungauged flux corrector $\sigma^{\rm u} [x_0']$ satisfies
\begin{align}\label{PrevStep}
\begin{aligned}
&\langl \lt( \fint_{\Boule(x_0,r)} \lt| \sigma^{\rm u} \lt[x_0'\rt] - \fint_{B(x_0,1)} \sigma^{\rm u} \lt[x_0'\rt]  \rt|^2\rt)^{\frac{p}{2}} \rangl^{\frac{1}{p}}
\\
&\qquad \lesssim c^{1+d/\nu_0}_{2dp/\nu_0}
r^{1-\nu}\ln^{1+ \lfloor \nu \rfloor}(r)
\end{aligned}
\end{align}
for any $r \geq 2$. In Step 5, by appealing to Lemma \ref{enforce_gauge_1}, we obtain that the (gauged) flux corrector $\sigma$ satisfies 
\begin{align}\label{PrevStep_new}
\langl \lt( \fint_{\Boule(x,r)} \lt| \sigma - \fint_{B(x,1)} \sigma \rt|^2\rt)^{\frac{p}{2}} \rangl^{\frac{1}{p}}
\lesssim c^{1+d/\nu_0}_{2dp/\nu_0}
r^{1-\nu}\ln^{1+ 2\lfloor \nu \rfloor}(r)
\end{align}
for any $x \in \mathbb{R}^d$ and $r \geq 2$. To finish, we convert this into \eqref{CorrSubNuSig}.

\paragraph{Step 1:} Set $x_0 \in \R^d$. We reinterpret \eqref{Defwm} and \eqref{Defg} in the easier context of \eqref{Defw2} to write:
\begin{align}\label{Defwm2}
-\div\lt( a \nabla \tilde{\phi}_k[x_0] \rt)
=\partial_i\lt(g_{ik}[x_0]\rt) \quad\text{in}\quad \R^d
\end{align}
with $g$ defined by \eqref{Defsigmau}.
By \eqref{Num:1011}, we have:
\begin{align}
\langl \left( \int_{\Q(x,1)}  \lt| \lt(\check{\phi}[x_0],\check{\sigma}[x_0]\rt) \rt|^2  \right)^{\frac{p}{2}} \rangl^{\frac{1}{p}}\lesssim c_p (1 + |x-x_0|)^{1- \nu}
\end{align}
for any $p<\infty$ and $x \in \R^d$.
The definition \eqref{Defsigmau} then yields that
\begin{align}
\label{Borneg}
\langl \lt(\int_{\Q(x,1)} \lt|g[x_0]\rt|^2\rt)^{\frac{p}{2}}\rangl^{\frac{1}{p}} 
\lesssim 
\lt\{
\begin{aligned}
&c_p \lt(1+ \lt|x-x_0\rt|\rt)^{1-\nu} && \quad\text{if}\quad x^\perp \in [-4,4],
\\
&0 && \quad\text{if}\quad x^\perp \notin [-4,4].
\end{aligned}
\rt.
\end{align}

To obtain \eqref{Claim0}, we then decompose $\tilde{\phi}[x_0]=\tilde{\phi}^1[x_0] +\tilde{\phi}^2[x_0]$, where for each $k \in  [\![ 1, d]\!]$ the function $\tilde{\phi}^1_k[x_0] \in \HH^1_{\rm{loc}}(\R^d)$ is the Lax-Milgram solution of
\begin{align}
\label{LM_new_1}
-\div\lt( a \nabla \tilde{\phi}^1_k[x_0] \rt)=\partial_i\lt(\mathds{1}_{\Q(x_0,3)} g_{ik}[x_0]\rt) \quad \quad\text{in}\quad \quad \R^d
\end{align}
and $\tilde{\phi}^2_k[x_0] \in \HH^1_{\rm{loc}}(\R^d)$ is a strictly sublinear solution to
\begin{align}
\label{LM_new_2}
-\div\lt( a \nabla \tilde{\phi}^2_k[x_0] \rt)=\partial_i\lt(\lt(1-\mathds{1}_{\Q(x_0,3)}(y)\rt) g_{ik}[x_0]\rt)   \quad \quad\text{in}\quad \quad \R^d.
\end{align}
Combining the energy estimate with \eqref{Borneg}, we get
\begin{align}
\langl \left( \int_{\R^d} \lt|\nabla \tilde{\phi}^1[x_0]\rt|^2\right)^{\frac{p}{2}} \rangl^{\frac{1}{p}} 
\lesssim  c_p \langl \left( \int_{\Q(x_0,3)} \lt|g[x_0]\rt|^2 \right)^{\frac{p}{2}} \rangl^{\frac{1}{p}} 
\lesssim c_p.
\label{Estim9}
\end{align}

To obtain estimates for $\nabla \tilde{\phi}^2[x_0]$, we differentiate the Green's function representation and decompose it on cubes. In particular, for $x \in \Q(x_0,1)$, we write:
\begin{align*}
\nabla \tilde{\phi}^2[x_0](x)
=&\int_{\R^d} \nabla_x \nabla_y G(x,y) \cdot g [x_0](y)  \lt(1-\mathds{1}_{\Q(x_0,3)}(y)\rt) \dd y.
\end{align*}
Combining this representation with the triangle inequality and H\"{o}lder's inequality, we obtain:
\begin{align*}
&\langl \lt(\int_{\Q(x_0,1)} \lt| \nabla \tilde{\phi}^2[x_0]\rt|^2 \rt)^{\frac{p}{2}} \rangl^{\frac{1}{p}}
\\
&\lesssim \sum_{k \in \Z^d \backslash \Q(0,3)} \langl \lt( \int_{\Q(x_0,1)} \lt| \int_{\Q(k+x_0,1)} \nabla_x \nabla_y G(x,y) \cdot g[x_0](y) \dd y\rt|^2 \dd x \rt)^{\frac{p}{2}} \rangl^{\frac{1}{p}}
\\
&\lesssim \sum_{k \in \Z^d \backslash\Q(0,3)}  \langl \left( \int_{\Q(k+x_0,1)} \lt|g[x_0]\rt|^2 \right)^{p}  \rangl^{1/2p}  \\
& \hspace{2.5cm} \times \langl \left(\int_{\Q(x_0,1)} \int_{\Q(k+x_0,1)} \lt| \nabla_x \nabla_y G(x,y)\rt|^2 \dd y \, \dd x \right)^{p} \rangl^{1/ 2p}.
\end{align*}

We treat the second term on the right-hand side of the above estimate with Lemma \ref{LemGreenDet} to the effect of
\begin{align*}
\langl \left(\int_{\Q(x_0,1)} \int_{\Q(k+x_0,1)} \lt| \nabla_x \nabla_y G(x,y)\rt|^2 \dd y \, \dd x \right)^{p} \rangl^{1/2p} \lesssim c^{d/\nu_0}_{2 d p / \nu_0}  |k|^{-d}.
\end{align*}
Since $g[x_0]$ is supported inside $[-3,3] \times \R^{d-1}$ and satisfies \eqref{Borneg}, we then have:
\begin{align}
\begin{aligned}
& \langl \lt(\int_{\Q(x_0,1)} \lt| \nabla \tilde{\phi}^2[x_0](x)\rt|^2 \dd x\rt)^{\frac{p}{2}} \rangl^{\frac{1}{p}}
\\
& \lesssim  
c_{2p} c^{d/\nu_0}_{2dp/\nu_0}
\sum_{ \textrm{\scriptsize $\begin{array}{ll}
k \in \Z^d \backslash \Q(0,3),\\
x^\perp_0 + k^\perp \in [-4,4]
\end{array}$}} |k|^{-d} |k|^{1-\nu}
\\
&  \lesssim c^{1+d/\nu_0}_{2dp/\nu_0} \sum_{ k^\parallel \in \Z^{d-1}} \lt(1+|k^\parallel| + \lt|x_0^\perp\rt|\rt)^{-d+1-\nu}
\\
&\lesssim  c^{1+d/\nu_0}_{2dp/\nu_0} \lt( 1 + \lt|x_0^\perp\rt|\rt)^{-\nu}.
\end{aligned}
\label{Estim8}	    
\end{align}
(By monotonicity, we have $c_{2dp/\nu_0}>c_{2p}$.) Thus, we have established \eqref{Claim0}.

\paragraph{Step 2:}
Here comes the argument for \eqref{Satisw}.
Recall that $\nabla \phi$ is uniquely defined (almost-surely).
Therefore, by \eqref{Defw2}, changing the anchoring of $\phi_{\pm}$ in the sense of $x_0 \rightsquigarrow x \in \R^d$, implies replacing $\nabla \tilde{\phi}[x_0] \rightsquigarrow \nabla \tilde{\phi}[x]$ as follows:
\begin{align}
\begin{aligned}
\nabla \tilde{\phi}_k[x]=&\nabla \tilde{\phi}_k[x_0] 
+   \mathds{1}_{\R_+\times \R^{d-1}}\lt(\fint_{\Q(x_0,1)}\phi_+ -  \fint_{\Q(x,1)}\phi_+\rt) \cdot \nabla P_k \nabla \chi
\\
&+  \mathds{1}_{\R_-\times \R^{d-1}}\lt(\fint_{\Q(x_0,1)}\phi_-  - \fint_{\Q(x,1)}\phi_-\rt) \cdot \nabla P_k \nabla \chi.
\end{aligned}
\label{Remplacer}
\end{align}
In particular, changing the anchoring point $x_0$ does not change the value of $\nabla \tilde{\phi}[x_0]$ outside of $[-3,3] \times \R^{d-1}$.
Hence, for any $x \in \R^d \backslash ([-4,4] \times \R^{d-1})$, \eqref{Claim0} yields that
\begin{align}
\begin{aligned}
\langl \lt(\fint_{\Q(x,1)} \lt|\nabla \tilde{\phi}[x_0]\rt|^2 \rt)^{\frac{p}{2}} \rangl^{\frac{1}{p}} 
& =\langl \lt(\fint_{\Q(x,1)} \lt|\nabla \tilde{\phi}[x]\rt|^2 \rt)^{\frac{p}{2}} \rangl^{\frac{1}{p}} 
\\
& \lesssim c^{1+d/\nu_0}_{2dp/\nu_0}  \lt(1 + \lt|x^\perp\rt|\rt)^{-\nu}.  
\end{aligned}
\label{Claim3}
\end{align}
Moreover, when $x \in [-4,4] \times \R^{d-1}$, using \eqref{Remplacer}, by assumption \eqref{CorrSubDef} in the form \eqref{Num:1011} and by \eqref{Claim0}, we obtain:
\begin{align*}
& \langl \lt(\int_{\Q(x,1)} \lt|\nabla \tilde{\phi}[x_0] \rt|^2  \rt)^{\frac{p}{2}} \rangl^{\frac{1}{p}}
\\
&\lesssim \langl \lt(\int_{\Q(x,1)} \lt|\nabla \tilde{\phi}[x] \rt|^2 \rt)^{\frac{p}{2}} \rangl^{\frac{1}{p}}
+ \langl \lt|  \int_{\Q(x,1)}\phi_{\pm} -  \int_{\Q(x_0,1)}\phi_{\pm} \rt|^p \rangl^{\frac{1}{p}} 
\\
&\lesssim c^{1+d/\nu_0}_{2dp/\nu_0}  \lt(1 + |x_0-x|\rt)^{1-\nu}.
\end{align*}
As a consequence, \eqref{Satisw} holds.

\paragraph{Step 3:}
We now establish \eqref{CorrSubNu}. To estimate $\phi$, we set a suitable anchoring point $x_0 \in \R^d$ and apply the triangle inequality on \eqref{Defw2}: The part $(1-\chi)\check{\phi}[x_0]$ is treated with \eqref{CorrSubDef} whereas $\tilde{\phi}[x_0]$ is handled by integrating $\nabla \tilde{\phi}[x_0]$, which is controlled by \eqref{Satisw}, along a suitable path $\Gamma$.

Fix $x, y \in \R^d$; we assume that $x^\perp \leq -4 < 4 \leq y^\perp$. (It is easy to check that our method extends to the general case.) We then introduce the points
\begin{align}
\label{points}
\begin{aligned}
& x_0 :=\lt(0,x^\parallel\rt), && x_1:=\lt(2 |x-y|,x^\parallel\rt), & \textrm{ and }& x_2 :=\lt(2 |x-y|,y^\parallel\rt).
\end{aligned}
\end{align}
From these points we draw the path $\Gamma$ parametrized by
\begin{align}
\label{path}
\gamma(t)= 
\begin{cases}
x (1 - 3t) + x_1 3 t  &\quad \text{for} \quad 0 \leq t \leq 1/3,\\
x_1 (2 - 3t) +  x_2 (3t - 1) &\quad \text{for} \quad 1/3 \leq t \leq 2/3, \\
x_2 (3 - 3t)+ y  (3t - 2) & \quad \text{for} \quad 2/3 \leq t \leq 1.\\
\end{cases}
\end{align}

\begin{figure}[h]
\begin{center}
\includegraphics{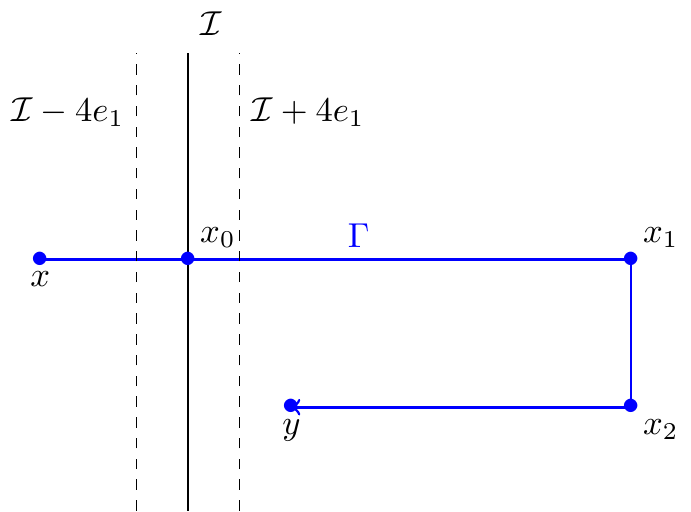}
\end{center}
\caption{Path of integration $\Gamma$.}\label{FigPath}
\end{figure}

By \eqref{Num:1011}, using $x_0$ as the reference point, there holds:
\begin{align*}
& \langl \left(  \int_{\Q(0,1)}\lt|\check{\phi}[x_0](x+z)-\check{\phi}[x_0](y+z)\rt|^2 \dd z \right)^{\frac{p}{2}} \rangl^{\frac{1}{p}}
\lesssim_d c_p( 1+ \lt|x-y\rt|^{1-\nu}).
\end{align*}
Combining  \eqref{Num:1011} with the splitting \eqref{Defw2} and an application of the triangle inequality, we obtain:
\begin{align*}
&  \langl \left(  \int_{\Q(0,1)}\lt|\phi(x+z)-\phi(y+z)\rt|^2 \dd z \right)^{\frac{p}{2}} \rangl^{\frac{1}{p}}\\
& \lesssim c_p ( 1+ \lt|x-y\rt|^{1-\nu}) + \langl \left(  \int_{\Q(0,1)}\lt|\tilde{\phi}[x_0](x+z)-\tilde{\phi}[x_0](y+z)\rt|^2 \dd z \right)^{\frac{p}{2}} \rangl^{\frac{1}{p}},
\end{align*}
which implies that proving \eqref{CorrSubNu} amounts to establishing
\begin{align}
\label{ToEstab}
\begin{split}
& \langl \left(  \int_{\Q(0,1)}\lt|\tilde{\phi}[x_0](x+z)-\tilde{\phi}[x_0](y+z)\rt|^2 \dd z \right)^{\frac{p}{2}} \rangl^{\frac{1}{p}}   
\\
&\lesssim c^{1+d/\nu_0}_{2dp/\nu_0} \begin{cases}
(1 + |x-y|)^{1-\nu} & \textrm{if} \quad \nu < 1,\\
\ln(2 + |x-y|)  & \textrm{if} \quad \nu = 1.\\
\end{cases} 
\end{split}
\end{align}

For \eqref{ToEstab}, we first use the Bochner inequality:
\begin{align*}
& \lt(\int_{\Q(0,1)}  \lt| \tilde{\phi}[x_0](x+z)  -\tilde{\phi}[x_0](y+z) \rt|^2 \dd z \rt)^{\frac{1}{2}}
\\
&= \lt(\int_{\Q(0,1)}  \lt| \int_{0}^1 \nabla \tilde{\phi}[x_0](z + \gamma(t)) \gamma^{\prime}(t) \, \dd t  \rt|^2 \dd z \rt)^{\frac{1}{2}}
\\
&\leq \int_{0}^1 \lt( \int_{\Q(\gamma(t),1)} \lt|\nabla \tilde{\phi}[x_0](z) \rt|^2 \dd z\rt)^{\frac{1}{2}} \lt| \gamma^{\prime}(t) \rt| \dd t,
\end{align*}
where $\gamma$ is defined in \eqref{path}.
As a consequence, using once more the Bochner inequality, we get from \eqref{Satisw} that
\begin{align*}
&\langl \lt(\int_{\Q(0,1)}  \lt| \tilde{\phi}[x_0](x+z)  -\tilde{\phi}[x_0](y+z) \rt|^2 \dd z \rt)^{\frac{p}{2}} \rangl^{\frac{1}{p}}
\\
&\lesssim \int_{0}^1 \langl \lt( \int_{\Q(\gamma(t),1)} \lt|\nabla \tilde{\phi}[x_0](z) \rt|^2 \dd z\rt)^{\frac{p}{2}}\rangl^{\frac{1}{p}} \lt| \gamma^{\prime}(t) \rt| \dd t
\\
&\lesssim
c^{1+d/\nu_0}_{2dp/\nu_0} \int_{0}^{1}  \lt(1+\lt|\gamma(t)^{\perp} \rt|\rt)^{-\nu} |\gamma^{\prime}(t)| \, \dd t  
\\
&\lesssim  c^{1+d/\nu_0}_{dp/\nu_0}
\begin{cases}
(1 + |x-y|)^{1-\nu} & \textrm{if} \quad \nu < 1,\\
\ln(2 + |x-y|)  & \textrm{if} \quad \nu = 1.\\
\end{cases}
\end{align*}
This establishes \eqref{ToEstab}, so that \eqref{CorrSubNu} is proved.

\paragraph{Step 4:} 
We now consider the ungauged corrector $\sigma^{\rm u} [x_0']$ defined by \eqref{Defsigmau}, for $x_0'=(0,x_0^\parallel)$.

By splitting $\widetilde{N}[x_0']$ into a far-field and near-field contribution (depending on the support of the right-hand side), we can use the Green's function associated to the Laplacian for the far-field piece and the standard energy estimate for the near-field piece to obtain:
\begin{align*}
\int_{\Boule(x,1)} \lt| \nabla^2 \widetilde{N}[x_0']\rt|^2 
\lesssim & \int_{\Q(x,   2)} |g[x_0']|^2 + |\nabla \tilde{\phi}[x_0']|^2
\\
&+\lt( \int_{\R^d \backslash \Q(x,2)} \frac{ |g[x_0'](z)| +|\nabla \tilde{\phi}[x_0'](z)| }{|x-z|^d} \dd z\rt)^2.
\end{align*}
Using \eqref{Satisw} and \eqref{Borneg}, we can then decompose the above estimate to find that 
\begin{align}
\label{Estim12}
\begin{aligned}
& \langl \lt(\int_{\Boule(x,1)} \lt| \nabla^2 \widetilde{N}[x_0']\rt|^2 \rt)^{\frac{p}{2}}\rangl^{\frac{1}{p}}
\\
& \lesssim  \sum_{k \in \Z^d} \frac{1}{\lt(|k-x| + 1 \rt)^{d}}\langl \lt(\int_{\Q(k,1)} \lt|g[x_0']\rt|^2 + \lt|\nabla \tilde{\phi}[x_0']\rt|^2  \rt)^{\frac{p}{2}} \rangl^{\frac{1}{p}}
\\
&\lesssim
c^{1+d/\nu_0}_{2dp/\nu_0} \lt( \sum_{\scriptsize{ \begin{array}{c} k \in \Z^d, \\ |k^\perp| \leq 4 \end{array}}} \frac{\lt(1 + |k-x_0'|\rt)^{1-\nu}}{\lt(|k-x| +1\rt)^{d}} 
+\sum_{\scriptsize{ \begin{array}{c} k \in \Z^d, \\ |k^\perp| \geq 4 \end{array}}} \frac{|k^\perp|^{-\nu}}{\lt(|k-x|+1\rt)^{d}} \rt).
\end{aligned}
\end{align}
Now, up to a multiplicative constant, we bound the first right-hand term of \eqref{Estim12} by
\begin{align*}
\sum_{k^\parallel \in \Z^{d-1}} \frac{\lt(1+|k^\parallel-x_0^\parallel|\rt)^{1-\nu}}{\lt(1+|k^\parallel-x^\parallel|+|x^\perp|\rt)^{d}}=&\sum_{k^\parallel \in \Z^{d-1}} \frac{\lt(1+|k^\parallel-x_0^\parallel+x^\parallel|\rt)^{1-\nu}}{\lt(1+|k^\parallel|+|x^\perp|\rt)^{d}}
\\
\lesssim& \lt(1 + |x^\parallel-x_0^\parallel |\rt)^{1-\nu} \lt(1 + |x^\perp|\rt)^{-1} 
\\
&\quad + \lt(1 + |x^\perp|\rt)^{-\nu}
\end{align*}
and the second right-hand term of \eqref{Estim12} by
\begin{align}
\label{new_sum}
\sum_{k^\perp \in \Z \setminus \left\{ 0 \right\} } \frac{ | k^\perp|^{-\nu}}{|k^\perp - x^\perp| + 1}
\lesssim \lt(1 + |x^\perp|\rt)^{-\nu} \ln\lt(2+ |x^\perp|\rt). 
\end{align}

By summation, we deduce from the three above inequalities that
\begin{align*}
\langl \lt(\int_{\Boule(x,1)} | \nabla^2 \widetilde{N}[x_0'] |^2 \rt)^{\frac{p}{2}} \rangl^{\frac{1}{p}}
\lesssim & c^{1+d/\nu_0}_{2dp/\nu_0}\lt(1 + |x^\parallel-x_0^\parallel|\rt)^{1-\nu} \lt(1 + \lt|x^\perp\rt|\rt)^{-1} 
\\
&+ c^{1+d/\nu_0}_{2dp/\nu_0}\lt(1 + \lt|x^\perp\rt|\rt)^{-\nu}  \ln\lt(2+ \lt|x^\perp\rt|\rt) .
\end{align*}
Integrating the previous estimate along a path $\Gamma$ similar to Step 3, we obtain that
\begin{align*}
& \langl \lt(\int_{\Boule(0,1)} \lt|\nabla \widetilde{N}[x_0'](x_0+ \cdot) - \nabla \widetilde{N}[x_0'](x + \cdot)\rt|^2 \rt)^{\frac{p}{2}} \rangl^{\frac{1}{p}}\\
& \quad \lesssim  c^{1+d/\nu_0}_{2dp/\nu_0}
\lt|x - x_0\rt|^{1-\nu} \ln^{1+ \lfloor \nu \rfloor}\lt(2 + \lt|x-x_0\rt| \rt).
\end{align*}
As a consequence, for any $R \geq 1$ and $\nu \in \left(0,1\right]$, we find:
\begin{align*}
&  \left\langle \left( \fint_{\Boule(x_0,R)} \lt| \nabla \widetilde{N}[x_0] - \fint_{\Boule(x_0,1)} \nabla \widetilde{N}[x_0] \rt|^2 \right)^{\frac{p}{2}} \right\rangle^{2/p} \\
& \lesssim   \lt(c^{1+d/\nu_0}_{2dp/\nu_0}\rt)^2 R^{-d} \int_{\Boule(0,R)} |z|^{2(1-\nu)} \ln^{2 (1+ \lfloor \nu \rfloor) }\lt(2+ \lt|z\rt| \rt) \dd z
\\
& \lesssim \lt(c^{1+d/\nu_0}_{2dp/\nu_0}\rt)^2 R^{2(1-\nu)}\ln^{2 (1+ \lfloor \nu \rfloor) }(2+R).
\end{align*} 
Then, recalling \eqref{CorrSubDef} in the form \eqref{Num:1011}, we obtain that $\sigma^{\rm u} $ defined by \eqref{Defsigmau} satisfies \eqref{PrevStep}.

\paragraph{Step 5:}
From the previous step, we have a family of ungauged flux correctors $\sigma^{\rm u}[x_0^{\prime}]$ depending on the anchoring point $x_0'=(0,x_0^\parallel) \in \mathbb{R}^d$ and satisfying \eqref{PrevStep}. Thus, we may apply Lemma \ref{enforce_gauge_1} to each $\sigma^{\rm u}[x^{\prime}_0]$ and obtain a new estimate on the $\langl\cdot\rangl$-almost surely unique (up to the addition of a random constant) flux corrector $\sigma$.
Namely, the latter satisfies \eqref{PrevStep_new}, for any $x \in \R^d$.

Finally, we notice that, for fixed $x, y \in \mathbb{R}^d$ such that $|x-y| =R \geq 1$, we may apply \eqref{PrevStep_new} for $x$ and $y$ with $r =1$ and then again for $y$ with $r=2R$, in order to obtain:
\begin{align*}
\int_{\Boule(0,1)} |\sigma(x + z) - \sigma(y+z)| \dd z \lesssim  1 + R^{1 - \nu} \ln^{1 + 2 \lfloor \nu \rfloor}(2 + R)
\end{align*}
as desired.
\end{proof}

\section{Proof of Corollary \ref{ThCvgDet}}\label{Sec_proof_ThCvgDet}

Since a very similar version of Corollary \ref{ThCvgDet} has already been proved in \cite{Josien_InterfPer_2018}, we only emphasize the main steps.
Indeed, the only technical difference between the setting in \cite{Josien_InterfPer_2018} and here is that the generalized correctors considered here are not bounded, but only satisfy \eqref{CorrSubNu}. We refer the interested reader to \cite{Josien_InterfPer_2018} for the technicalities due to the interface, and to \cite{Article_Homog} for the technicalities due to the growth rate \eqref{CorrSubNu} of the generalized correctors. We also refer to \cite{KLSGreenNeumann}, from which the main ideas of the aforementioned articles are borrowed.

\begin{proof}[Proof of Corollary \ref{ThCvgDet}]
We first justify \eqref{Eq:ThLinfty3-1}, which corresponds to \cite[Prop.\ 4.3]{Josien_InterfPer_2018}.
It is a consequence of \eqref{Algebre}, that can be reformulated thanks to the Green's function $G^\epsilon$ associated with the operator $-\div\lt(a(\cdot/\epsilon)\nabla \cdot \rt)$ in $\R^d$ as
\begin{align*}
&u^\epsilon(x) - \overline{u}(x) - \epsilon \phi\lt(\frac{x}{\epsilon}\rt) \cdot \overline{\nabla} \overline{u}(x)
\\
&=
-\epsilon\int_{\R^d} \partial_{y_i}G^\epsilon(x,y) \lt(\lt(a_{ij}\phi_k-\sigma_{ijk}\rt)\lt(\frac{y}{\epsilon}  \rt)\partial_j \overline{\partial}_k \overline{u}(y)  \rt) \dd y.
\end{align*}
By applying a Hölder inequality on the above right-hand term, in which we inject the regularity of $\overline{\nabla} \overline{u}$ (see \cite[Th.]{Lorenzi_1972}), the estimates on the generalized correctors provided by \eqref{CorrSubNu} and the following estimate  on the Green's function
\begin{align}
\label{Green2prim}
\lt|\nabla_y G^{\epsilon}(x,y) \rt| \lesssim |x-y|^{-d+1},
\end{align}
(which follows from \cite[Th.\ 1]{BellaGiunti_2018}) we obtain \eqref{Eq:ThLinfty3-1}.

By a duality argument detailed in \cite[Th.\ 4.6]{Article_Homog} and in \cite[Prop.\ 4.4]{Josien_InterfPer_2018}, we deduce from \eqref{Eq:ThLinfty3-1} that, for any $x \neq y$, the following estimate hold:
\begin{align}
\left|G^\epsilon(x,y) - \bar{G}(x,y)\rt| 
&\lesssim
\epsilon^\nu  \frac{\ln^{1+ 2\lfloor \nu \rfloor}\lt(2 + \epsilon^{-1} |x-y|\rt)}{|x-y|^{d-2+\nu}},
\label{DefGreen1}
\end{align}
where $\bar{G}$ is the Green's function of the homogenized operator $-\div\lt(\overline{a} \nabla \cdot \rt)$ in $\R^d$.

As in \cite[Th.\ 4.5]{Josien_InterfPer_2018}, it is deduced from Lipschitz regularity (\textit{\textit{i.e.}} Theorem \ref{ThAL}) that the previous estimate \eqref{DefGreen1} can be upgraded to the level of the gradients:
\begin{align}
\label{DefGreen2}
\begin{split}
& \left| \nabla_x G^\epsilon(x,y)-\nabla_x\bar{G}(x,y)-\nabla \phi\lt(\frac{x}{\epsilon}\rt) \cdot \overline{\nabla}_x \bar{G}(x,y) \right|\\
&\leq C \epsilon^\nu \frac{\ln^{2 + 2\lfloor \nu \rfloor }\lt(2+\epsilon^{-1}|x-y|\rt) }{|x-y|^{d-1+\nu}}.
\end{split}
\end{align}

Finally, as in the proof of \cite[Cor.\ 4.6]{Josien_InterfPer_2018}, we express
\begin{align*}
&\nabla u^\epsilon(x) - \nabla \overline{u}(x) - \nabla\phi\lt(\frac{x}{\epsilon} \rt) \cdot \overline{\nabla} \overline{u}(x)
\\
&=\int_{\Boule(x_0,1)} \lt( \nabla_x G^{\epsilon}(x,y) - \nabla \bar{G}(x,y) -\nabla\phi\lt(\frac{x}{\epsilon} \rt) \cdot \overline{\nabla} \bar{G}(x,y)\rt) f(y) \dd y.
\end{align*}
Then, by using a simple Hölder inequality involving the previous estimate, we obtain the desired result \eqref{Eq:ThWLinfty}. (In the case $\nu=1$, an additional decomposition is required to avoid the singularity $x=y$ for $\nu=1$; it is detailed in \cite[Cor.\ 4.6]{Josien_InterfPer_2018}.)
\end{proof}

\section{Proof of Proposition \ref{PropContrex}}\label{Sec_Proof_PropContrex}

\begin{proof}
In the sequel, the symbol ``$\lesssim$'' will be used for ``$\lesssim_{d, \eta_1, \eta_2}$''. The equation on $\phi_1$ reads:
\begin{align*}
&-\div\lt( a \nabla \phi_1\rt)=\div\lt( ae_1\rt) = \div\lt( \eta e_1 \rt) \qquad \quad\text{in}\quad \R^d.
\end{align*}
The strategy of the proof is to make use of the Green's function $G$ associated with the operator $-\div\lt(a \cdot\nabla\rt)$ to express the growth rate of $\phi_1$ between two points $x$ and $x' \in \R^d$:
\begin{align}\label{GwthRate_strip}
\phi_1(x')-\phi_1(x)
=& \int_{[x',x]} \lt(\int_{D} \eta(y) \nabla_x\nabla_y G(x'',y) \cdot e_1 \dd y\rt) \cdot \dd x''.
\end{align}

The proof then falls in two steps.
Step 1 is devoted to deriving an asymptotic estimate for $\nabla_y G(x,y)$ when $y$ is far from the interface $\mathcal{I}$.
Equipped with this, in Step 2, we choose suitable points $x$ and $x' \in \R^d$ and explicitly compute the leading order of \eqref{GwthRate_strip}, thus obtaining \eqref{BorneCorr_strip}.

\paragraph{Step 1:}
Recall that the correctors $\phi_\pm$ and $\sigma_\pm$ are bounded.
Therefore, we may apply Theorem \ref{PropRefinedDeter} to obtain strictly sublinear generalized correctors $\lt(\phi,\sigma\rt)$.
Since the correctors $\Phi_\pm$ are uniformly bounded in $\R^d$, we have the additional estimate
\begin{align}\label{ApproxBarphi}
\lt|\nabla \phi(y) \rt| \lesssim \lt(1+\lt|y^\perp\rt|\rt)^{-1},
\end{align}
for any $y \in \R^d$. (The above estimate is \eqref{Claim0} turned into a pointwise estimate thanks to the smoothness of $a$).

Also, the Green's function $G$ satisfies \eqref{DefGreen2}.
Thus, we may deduce from the regularity and symmetry of $a$ the following (suboptimal but convenient) estimate:
\begin{align}
\left| \nabla_y G (x,y) + \lt(\Id+ \nabla \phi\lt(y\rt)\rt) \cdot \nabla \bar{G}(x-y) \right|
&\lesssim |x-y|^{-d+\frac{1}{2}},
\label{GreenDef_strip}
\end{align}
for any $x, y \in \R^d \backslash \mathcal{I}$, $|x-y| \geq 1$, where $\bar{G}$ is the Green's function of the operator $-\Delta$. 
The latter function satisfies
\begin{align}\label{GDelta_strip}
\lt\{
\begin{aligned}
& \bar{G}(x)=C_d |x|^{-d+2}, \quad   \nabla\bar{G}(x)=
-C_d(d-2) \frac{x,}{|x|^{d}},
\\
&\nabla \nabla\bar{G}(x)=C_d (d-2) \frac{d (x \otimes x) - |x|^2 \Id}{|x|^{d+2}}.
\end{aligned}
\rt.
\end{align}
for a universal constant $C_d>0$ (see \cite[(4.1) Chap.\ 4 p.\ 51]{GT}).

By a triangle inequality involving \eqref{ApproxBarphi}, \eqref{GreenDef_strip}, and \eqref{GDelta_strip}, we obtain an approximation of $\nabla_y G(x,y)$ in the form of
\begin{align}
&\left| \nabla_y G (x,y) + \nabla \bar{G}(x-y) \right|
\lesssim \frac{|x-y|^{-\frac{1}{2}} +\lt(1+\lt|y^\perp\rt|\rt)^{-1}}{|x-y|^{d-1}},
\label{Green2_strip}
\end{align}
for any $x, y \in \R^d$ with $|x-y| \geq 1$.

\paragraph{Step 2:}
We now show that, given $r\geq 1$, estimate \eqref{BorneCorr_strip} holds for the following choice of $x$ and $x'$:
\begin{align*}
x:=-e_1 \quad\text{and}\quad x':=-(r+1) e_1.
\end{align*}
More precisely, we prove that there exists a positive constant $C$ such that
\begin{align}\label{ee7_strip}
\lt|\phi_1(x)-\phi_1(x')\rt| \geq C \ln(r) + O(1).
\end{align}

We split the strip $D$, defined in \eqref{DefD}, into a far domain and a near domain:
\begin{align}
\label{far_and_near}
D_1:= [r,+\infty) \times [-1,1] \times \R^{d-2} \quad\text{and}\quad D_2:= [0,r] \times [-1,1] \times \R^{d-2}
\end{align}
(see Figure \ref{FigureDefa_strip}).
The identity \eqref{GwthRate_strip} is accordingly reinterpreted as
\begin{align}
\phi_1(x')-\phi_1(x)
=&\int_{[x',x]} \lt(\int_{D_1} \eta(y) \nabla_x\partial_{y_1} G(x'',y)\dd y\rt) \cdot \dd x''
\nonumber
\\
&+ \int_{[x',x]} \lt(\int_{D_2} \eta(y) \nabla_x\partial_{y_1} G(x'',y) \dd y\rt) \cdot \dd x''
=:I_1 + I_2.
\label{eq101_strip}
\end{align}

By Lemma \ref{LemGreenDet} combined with \eqref{Green2prim} and the regularity of $a$, the Green's function $G$ satisfies the following estimates:
\begin{align}
\label{Green20_strip}
\lt|\nabla_y G(x,y) \rt| &\lesssim |x-y|^{-d+1} \quad\text{and}\quad \lt|\nabla_x \nabla_y G(x,y)\rt| \lesssim |x-y|^{-d},
\end{align}
for any $x, y\in \R^d \backslash \mathcal{I}$ with $|x-y| \geq 1$.
Whence, we may estimate $I_1$ defined in \eqref{eq101_strip} by
\begin{align}
\lt| I_1\rt|
&\lesssim
\int_{[x',x]} \lt(\int_{D_1} \frac{1}{|x''-y|^d} \dd y \rt) \lt|\dd x''\rt|
\nonumber
\\
&\lesssim
r \lt(\int_{[-1,1] \times \R_+ \times \R^{d-2}} \frac{1}{\lt(r+|y|\rt)^d} \dd y \rt)
\lesssim 1,
\label{ee5_strip}
\end{align}
where we used that $|x-x'|=r$.

We now consider $I_2$ in \eqref{eq101_strip}.
We first integrate along the $x''$ variable:
\begin{align*}
I_2
&=\int_{D_2} \eta(y) \partial_{y_1} G(x,y) \dd y -\int_{D_2} \eta(y)\partial_{y_1} G(x',y) \dd y
=:I_{21}+I_{22}.
\end{align*}
On the one hand, we treat the second integral by using \eqref{Green20_strip}
\begin{align}
\nonumber
\lt|I_{22}\rt|
\lesssim & \int_{D_2} \frac{1}{|x'-y|^{d-1}} \dd y
\lesssim \int_0^r \int_{\R^{d-2}} \frac{1}{ \lt(r+1+y_1 + \lt|\tilde{y}\rt|\rt)^{d-1}}\dd y_1 \dd \tilde{y}
\\
\label{ee6_strip}
\lesssim& \int_0^r \frac{1}{(r+1+y_1)}\dd y_1 \lesssim 1.
\end{align}
On the other hand, by \eqref{Green2_strip} and since $1\leq |x-y| \lesssim |x'-y|$ if $y \in D_2$, we approximate $I_{21}$ as follows:
\begin{align}
\nonumber
\lt|I_{21} +\int_{D_2} \eta  \partial_1 \bar{G}(x-\cdot) \rt|
& \lesssim \int_{D_2} \frac{|x-y|^{-1/2}+\lt(1+\lt|y^\perp\rt|\rt)^{-1}}{|x-y|^{d-1}} \dd y
\\
\nonumber
&\lesssim \int_0^{r} \int_{\R^{d-2}} 
\frac{\lt(1+y_1+|\tilde{y}|\rt)^{-1/2}+\lt(1+\lt|y_1\rt|\rt)^{-1}}{\lt(1+y_1+|\tilde{y}|\rt)^{d-1}}
\dd \tilde{y} \dd y_1
\\
\nonumber
&\lesssim \int_0^{+\infty} \lt[\frac{1}{|1+y_1|^{3/2}} 
+ \frac{\lt(1+\lt|y_1\rt|\rt)^{-1}}{\lt( |1+y_1|\rt)}
\rt]
\dd y_1
\lesssim 1.
\end{align}

We then further simplify $I_{21}$ by using a Taylor expansion to remove the dependence of $\partial_1 \bar{G}(z)$ on $z\cdot e_2$. In particular, this requires the decay of $\nabla^2 \bar{G}$ provided by \eqref{GDelta_strip} and the observation that $\eta$ is supported in the set $ \R_+ \times [-1,1] \times \R^{d-2}$. Recalling that $\eta$ is defined by \eqref{Defeta_strip}, that $\eta_1=1$ in $[1,+\infty)$, and substituting $x = -e_1$, we then obtain:
\begin{align}
I_{21}=-C_0 \int_{0}^{r} \int_{\R^{d-2}} \partial_1 \bar{G}\lt(\lt(-1-y_1,0,\tilde{y})\rt)\rt) \dd \tilde{y}\dd y_1
+O(1),
\label{ee2_strip}
\end{align}
for the positive constant
$C_0:=\int_{\R} \eta_2>0$.
Therefore, injecting \eqref{GDelta_strip} in \eqref{ee2_strip}, and then performing the change of variables $\tilde{y}=(1+y_1)\tilde{z}$, we may explicitly compute the above integral, to the effect of
\begin{equation}
\label{ee3_strip}
\begin{aligned}
I_{21}
=&
-C_1 \int_{0}^{r} \int_{\R^{d-2}}\frac{1+y_1}{\lt( \lt(1+y_1\rt)^2 + \lt|\tilde{y}\rt|^2 \rt)^{\frac{d}{2}}} \dd \tilde{y}\dd y_1
+O(1)
\\
=&-C_2 \int_{0}^{r} \frac{1}{\lt(1+y_1\rt)} \dd y_1 + O(1)
=-C_2 \ln(r+1) + O(1)
\end{aligned}
\end{equation} 
for positive constants $C_1, C_2>0$.

Combining \eqref{eq101_strip}, \eqref{ee5_strip}, \eqref{ee6_strip} and \eqref{ee3_strip} establishes \eqref{ee7_strip} and  concludes the proof of Proposition \ref{PropContrex}.
\end{proof}

\section{Proof of Theorem \ref{ThIndepGauss}}\label{Sec_Proof_ThIndepGauss}

\begin{remark}
For the sake of simplicity, we make use of Theorem \ref{PropRefinedDeter} to show Theorem \ref{ThIndepGauss}.
However, we only need the uniqueness of correctors and a (suboptimal) growth quantification (the one provided by Proposition \ref{PropExistCorr} is sufficient).
\end{remark}

\begin{proof}[Proof of Theorem \ref{ThIndepGauss}]

Existence and uniqueness (up to the addition of a constant) of strictly sublinear generalized correctors is already provided by Theorem \ref{PropRefinedDeter}.
Thus, we only have to check that \eqref{SublinRes} is satisfied.

We follow the classical approach described in the lecture notes \cite{JosienOtto_2019}.
First, we obtain from Theorem \ref{PropRefinedDeter} and from the regularity of the coefficient field $a$ a bound on the moments of the gradient of $\phi$ and $\sigma$:
\begin{align}
\label{LpGrad}
\sup_{x \in \R^d \backslash \lt( -e_1+\mathcal{I} \cup e_1+\mathcal{I} \rt)  }\langl  \lt|\lt(\nabla \phi,\nabla \sigma\rt)(x)\rt|^{p} \rangl^{\frac{1}{p}} \lesssim 1. 
\end{align}
Next, we establish that \eqref{Prop21a} holds.
We finally obtain \eqref{SublinRes} by applying \eqref{Prop21a} to a suitable functional
and by making use of the regularity of $a$.

Before proceeding with the proof, we make a couple of simplifications. We assume that $d_{\rm g}=1$, and that $d \geq 3$. (The only difference for the case $d=2$ is located in Step 3, where dimension plays a role in the potential theory).

\paragraph{Step 1:}
Let $x \in \R^d$.
By combining the Caccioppoli estimate, \eqref{CorrSubNu}, \eqref{CorrSubNuSig}, we easily obtain that
\begin{align*}
\langl\lt(\fint_{\Boule(x,1)} \lt|\lt(\nabla \phi,\nabla\sigma\rt)\rt|^2 \rt)^p \rangl  \lesssim 1.
\end{align*}
Then, by \eqref{wg03} and \eqref{HypoA_bis}, the coefficient $a$ is $\langle \cdot \rangle$-almost surely $\alpha/2$-Hölder continuous in $\Boule(x,1)$ except on the interfaces $-e_1+\mathcal{I}$ and $e_1+\mathcal{I}$, and, for any $p  \in [1,+\infty)$, there holds
\begin{align*}
\langl\lt\| a \rt\|^p_{\CC^{0,\frac{\alpha}{2}} \lt(\Boule(x,1) \backslash \lt( -e_1+\mathcal{I} \cup e_1+\mathcal{I} \rt) \rt)} \rangl^{\frac{1}{p}} \lesssim 1.
\end{align*}

Thus, a local version of the regularity theorem \cite[Th.\ 1.1]{LiNirenberg_2003} successively yields that $\phi$ and $\sigma$ satisfy \eqref{LpGrad}.

\paragraph{Step 2:}
For simplicity, we assume that $F$ only depends on $a_-$ and $a_+$ (the argument below is easily generalized), and that $\langl F \rangl=0$.
We denote by $\langl \cdot\rangl_\pm$ the ensembles associated with $g_\pm$ respectively.
We already know (see \cite[Sec.\ 3.2]{JosienOtto_2019}), that any functional depending only on $g_{\pm}$ satisfies \eqref{Prop21a} with $a$ and $\langl \cdot\rangl$ respectively replaced by $g_{\pm}$ and $\langl\cdot\rangl_{\pm}$.
By independence of $g_-$ and $g_+$, we also have that
\begin{align}
\label{Indep}
\langl G \rangl=\langl \langl G\rangl_- \rangl_+,
\end{align}
for any random variable $G$ depending only on $g_-$ and $g_+$

Therefore, by applying the spectral gap \eqref{Prop21a} only with respect to $\langl \cdot\rangl_-$, there holds:
\begin{align*}
\langl |F|^{2p} \rangl 
=\langl \langl |F|^{2p} \rangl_- \rangl_+
\lesssim & \langl \lt( \int_{\R^d} \lt| \frac{\partial F}{\partial g_-(z)}\rt|^2 \dd z \rt)^{p}  \rangl
+ \langl \lt| \langl F \rangl_-\rt|^{2p} \rangl_+
\end{align*}
Now invoking the spectral gap only with respect to $\langl \cdot\rangl_+$, we obtain:
\begin{align*}
\langl \lt| \langl F \rangl_-\rt|^{2p} \rangl_+
\lesssim&
\langl\lt( \int_{\R^d}\lt| \frac{\partial \langl F \rangl_-}{\partial g_+(z)}\rt|^2 \dd z \rt)^p \rangl_+ + \lt|\langl \langl F \rangl_- \rangl_+\rt|^{2p}
\\
\lesssim&
\langl\lt( \int_{\R^d} \langl\lt| \frac{\partial F}{\partial g_+(z)}\rt|^2 \rangl_- \dd z \rt)^p \rangl_+ + \lt|\langl F \rangl\rt|^{2p}
\\
\lesssim&
\langl\lt( \int_{\R^d} \lt| \frac{\partial F}{\partial g_+(z)}\rt|^2 \dd z \rt)^p \rangl,
\end{align*}
where we have used the Bochner inequality and $\langl F \rangl=0$.
The two above inequalities yield:
\begin{align}
\langl |F|^{2p} \rangl 
\lesssim 
\langl \lt( \int_{\R^d} \lt( \lt| \frac{\partial F}{\partial g_-(z)}\rt| + \lt| \frac{\partial F}{\partial g_+(z)}\rt| \rt) ^2 \dd z \rt)^{p} \rangl.
\end{align}
Finally, by the chain rule (and recalling \eqref{HypoA_bis}), we establish \eqref{Prop21a}.

\paragraph{Step 3:}
For the sake of simplicity, we only show \eqref{SublinRes} for $\phi$. Also, we make take $p$ large (by the Hölder inequality, if \eqref{SublinRes} holds for $p=p_1$, then it also holds for any $p=p_2 \leq p_1$).
Let $x, y \in \R^d$ (we recall that $d \geq 3$ here).
The aim of this step is to establish the following estimate:
\begin{align}\label{Diffmean}
\langl \lt|\fint_{\Boule(x,1)}\phi - \fint_{\Boule(y,1)} \phi \rt|^{2p}\rangl \lesssim 1.
\end{align}

By the Sobolev embedding (provided $2p >d$), there also holds that
\begin{align*}
\lt| \phi(x) - \fint_{\Boule(x,1)} \phi\rt|^{2p} \lesssim \lt( \fint_{\Boule(x,1)} \lt|\nabla \phi\rt|^{2p} \rt).
\end{align*}
Thence, taking the expectation and recalling \eqref{LpGrad}, we deduce that
\begin{align*}
\langl \lt| \phi(x) - \fint_{\Boule(x,1)} \phi\rt|^{2p} \rangl 
\lesssim \langl \fint_{\Boule(x,1)} \lt|\nabla \phi\rt|^{2p} \rangl \lesssim 1.
\end{align*}
Combining this estimate with \eqref{Diffmean} establishes \eqref{SublinRes} and concludes the proof in the case $d \geq 3$.

Here comes the argument for \eqref{Diffmean}.
We define $w$ and $v$ as the strictly sublinear solutions to
\begin{align*}
-\Delta w=\mathds{1}_{\Boule(x,1)} -\mathds{1}_{\Boule(y,1)} \quad\text{and}\quad -\div\lt(a^\star \nabla v + \nabla w\rt)=0.
\end{align*}
We set
\begin{align}\label{Def_F}
F:=\fint_{\Boule(x,1)}\phi - \fint_{\Boule(y,1)} \phi= -\int_{\R^d} \Delta w \phi = \int_{\R^d} \nabla w \cdot \nabla \phi.
\end{align}
Now, notice that
\begin{align*}
\frac{\partial F}{\partial a(z)} = \int_{\R^d} \nabla w \cdot \nabla \frac{\partial \phi}{\partial a(z)} = -\int_{\R^d} \nabla v \cdot a \nabla \frac{\partial \phi}{\partial a(z)}.
\end{align*}
By differentiating the equation \eqref{DefCorr}, there holds:
\begin{align*}
-\div \lt( a \cdot \nabla \frac{\partial \phi_j}{\partial a(z)} + \delta_z \nabla \lt( \phi_j + P_j\rt)^\star  \rt)=0.
\end{align*}
Thus, we obtain:
\begin{align*}
\frac{\partial F}{\partial a(z)} =\lt(\nabla\lt( \phi_j + P_j\rt) \otimes  \nabla v\rt)(z).
\end{align*}

Then, by the spectral gap \eqref{Prop21a} and a duality argument, we deduce that
\begin{align*}
\langl \lt|F \rt|^{2p} \rangl \lesssim& \langl \lt( \int_{\R^d} \lt|\nabla \lt(\phi_j + P_j\rt) \otimes  \nabla v \rt|^2 \rt)^p \rangl
\\
\lesssim& \sup_{\langl |G|^{2p'} \rangl \leq 1} \langl \int_{\R^d}\lt|\nabla \lt( \phi_j + P_j\rt) \otimes  \nabla v \rt|^2 |G|^2 \rangl^p,
\end{align*}
where $p'$ is the conjugated exponent of $p$ (\textit{\textit{i.e.}} $1/p+1/p'=1$).
By the Hölder inequality and thanks to \eqref{LpGrad}, we estimate the above right-hand side as follows:
\begin{align*}
\langl \int_{\R^d}\lt|\lt(\nabla \phi_j + P_j\rt) \otimes  \nabla v \rt|^2 |G|^2 \rangl
&\leq 
\int_{\R^d} 
\langl \lt|\nabla (\phi_j + P_j)\rt|^{2p} \rangl^{\frac{1}{p}} 
\langl  \lt|G\nabla v \rt|^{2 p'} \rangl^{\frac{1}{p'}}
\\
&\lesssim \int_{\R^d} \langl  \lt|G\nabla v \rt|^{2 p'} \rangl^{\frac{1}{p'}}.
\end{align*}
Recall that $G v$ satisfies
\begin{align*}
\div\lt(a^\star \nabla \lt(Gv\rt) + G\nabla w\rt)=0.
\end{align*}
Therefore, if $p'$ is sufficiently small, that is, if $p$ is sufficiently large, we may apply the annealed Meyers estimates provided by \cite[Prop.\ 4.1(i)]{JosienOtto_2019} (where it is sufficient that $a^\star$ is uniformly elliptic and bounded). Thus, we get
\begin{align*}
\int_{\R^d} \langl  \lt|G\nabla v \rt|^{2 p'} \rangl^{\frac{1}{p'}} \lesssim
\int_{\R^d} \langl  \lt|G\nabla w \rt|^{2 p'} \rangl^{\frac{1}{p'}} \lesssim
\int_{\R^d} \lt|\nabla w \rt|^{2} \lesssim 1,
\end{align*}
since $w$ is deterministic and $\langl \lt|G\rt|^{2p'}\rangl \leq 1$ (the last bound being obtained by the potential theory in dimension $d \geq 3$).
This implies \eqref{Diffmean}.
\end{proof}

\paragraph{Acknowledgement}
We would like to thank:
Sonia Fliss, with whom we had very interesting discussions during her stay at the Max Planck Institute in Leipzig;
Felix Otto, for stimulating discussions;
Nicolas Forcadel, who suggested to consider the case of a thin layer between the two heterogeneous media.

\def\cprime{$'$} \def\cprime{$'$} \def\cprime{$'$}

\appendix{}

\section{A technical lemma for harmonic functions}
\label{A}

Throughout this paper we make use of the following technical lemma:
\begin{lemma_appendix}\label{LemDelta}
Let $d \geq 2$, $p \geq 2$, $x_0 \in \R^d$, $\langl \cdot \rangl$ be a given ensemble, and $f \in \LL^p(\Omega,  \LL^2_\loc\lt(\R^d,\R^d\rt))$ be a random vector field. We assume that for any $R\geq 1$ we have:
\begin{align}\label{ee0}
\langl \lt( \fint_{\Boule(x_0,R)} \lt|f\rt|^2 \rt)^{\frac{p}{2}} \rangl^{\frac{1}{p}} \leq C_0R^{1-\nu}\ln^\beta(2+R)
\end{align}
for some exponents $\nu \in (0,1]$ and $\beta \geq 0$.
Then, $\langle \cdot \rangle$-almost surely, there exists a distributional solution $u \in \HH^1_\loc(\R^d)$ of
\begin{align}\label{ee3}
-\Delta u = \nabla \cdot f \quad\text{in}\quad \R^d
\end{align}
subject to the constraint:
\begin{align}\label{ee10}
\limsup_{R\rightarrow + \infty} R^{-1}\lt(\fint_{\Boule(x_0,R)} \lt|\nabla u-\fint_{\Boule(x_0,R)} \nabla u  \rt|^2 \rt)^{\frac{1}{2}}=0.
\end{align}
The solution $u$ is unique up to the addition of affine functions.
Moreover, for any $R \geq 1$ we have that 
\begin{align}
\begin{aligned}
\langl \lt( \fint_{\Boule(x_0,R)} \lt|\nabla u - \fint_{\Boule(x_0,R)} \nabla u\rt|^2   \rt)^{\frac{p}{2}} \rangl^{\frac{1}{p}}
\lesssim_{d,\nu,\beta} C_0 R^{1-\nu} \ln^\beta(2+R)
\end{aligned}
\label{ee2}
\end{align}
and
\begin{equation}
\label{SpecialCase}
\begin{aligned}
&\langl \lt( \fint_{\Boule(x_0,R)} \lt|\nabla u - \fint_{\Boule(x_0,1)} \nabla u \rt|^2   \rt)^{\frac{p}{2}} \rangl^{\frac{1}{p}} 
\\
& \qquad \lesssim_{d,\nu,\beta}
\lt\{
\begin{aligned}
& C_0R^{1-\nu} \ln^\beta(2+R)&& \quad\text{if}\quad \nu<1,
\\
& C_0 \ln^{\beta+1}(2+R) && \quad\text{if}\quad \nu=1.
\end{aligned}
\rt.
\end{aligned}
\end{equation}
\end{lemma_appendix}

\begin{proof}[Proof of Lemma \ref{LemDelta}]
The uniqueness of the solution $u$ is a direct application of the Liouville principle for harmonic functions. Our argument then breaks down into three steps: First, in Step 1, we prove the $\langle \cdot \rangle$-almost sure existence of a solution to \eqref{ee3}. Then, in Step 2, we show that this solution satisfies \eqref{ee2} for any $R\geq 1$. To finish, in Step 3 we deduce \eqref{SpecialCase} from \eqref{ee2}.

\paragraph{Step 1:}
We decompose
\begin{align*}
f=\sum_{m=1}^{+\infty} f_m \quad\text{for}\quad f_m:=f \mathds{1}_{\Boule\lt(x_0,(2^m-1)\rt) \backslash \Boule\lt(x_0,\lt(2^{m-1}-1\rt)\rt)}
\end{align*}
and let $u_m \in H_{loc}^1(\R^d)$ be the Lax-Milgram solutions of 
\begin{align}
\label{weak_formulation_2}
\int_{\R^d} \nabla \psi \cdot \nabla u_m  = - \int_{\R^d} \nabla \psi \cdot f_m.
\end{align}
Taking $\psi = u_m$ in \eqref{weak_formulation_2} and using H\"{o}lder's inequality and  \eqref{ee0} implies that 
\begin{align}
\nonumber
\langl \lt(\int_{\R^d} \lt|\nabla u_m\rt|^2\rt)^{\frac{p}{2}} \rangl^{\frac{1}{p}} 
\leq&
\langl \lt(\int_{\R^d} \lt|f_m\rt|^2 \rt)^{\frac{p}{2}} \rangl^{\frac{1}{p}}
\\
\leq&
\lt(2^m\rt)^{d/2}\lt(2^m\rt)^{1-\nu}\ln^\beta(2+2^m).
\label{ee19}
\end{align}

We then define:
\begin{align}
\label{defined_1}
\begin{split}
\tilde{u}_m(x):=
\lt\{
\begin{aligned}
&u_m(x)  && \quad\text{if}\quad m=1,
\\
&u_m(x)  -u_m(x_0)- x\cdot \nabla u_m(x_0) && \quad\text{if}\quad m >1,
\end{aligned}
\rt.        
\end{split}
\end{align}
and claim that the series
\begin{align}\label{ee4}
u:=\sum_{m=1}^{+\infty} \tilde{u}_m
\end{align}
almost-surely defines a distributional solution to \eqref{ee3}. In particular, we show that the series in \eqref{ee4} converges in $\HH^1_{\rm{loc}}(\R^d)$ $\langle \cdot \rangle$-almost surely.

Let $R:=2^{m_0}$ for $m_0 \in \mathbb{N}$. Notice that, to show the desired convergence, we may discard the terms of \eqref{ee4} with $m \in \[1,m_0+3\]$ as the estimate \eqref{ee19} ensures that $\langle \cdot \rangle$-almost surely $\nabla u_m \in \LL^2 (\R^d)$ for any $m \geq 1$.
For $m\geq m_0+3$, the function $u_m$ is harmonic in $\Boule(x_0,2R)$, which means that we have access to the mean-value property and the Caccioppoli inequality. We obtain:
\begin{align}
\label{Argument}
\begin{split}
\fint_{\Boule(x_0,R)} \lt|\nabla \tilde{u}_m\rt|^2
\stackrel{\eqref{defined_1}}{\lesssim}
R^2 \sup_{\Boule(x_0,R)} \lt| \nabla^2 u_m\rt|^2
&  \lesssim  R^2 \fint_{\Boule(x_0,2R)} \lt| \nabla^2 u_m\rt|^2
\\
&  \lesssim  R^2 \fint_{\Boule(x_0,2^{m-1})} \lt| \nabla^2 u_m\rt|^2
\\
&  \lesssim  R^22^{-2m} \fint_{\Boule(x_0,2^{m})} \lt| \nabla u_m\rt|^2.
\end{split}
\end{align}
Therefore, recalling \eqref{ee19} we obtain for any $m \geq m_0+3$
\begin{align*}
\langl \lt( \fint_{\Boule(x_0,R)} \lt|\nabla \tilde{u}_m\rt|^2 \rt)^{\frac{p}{2}} \rangl^{\frac{1}{p}}
\lesssim 2^{m_0-m} 2^{m(1-\nu)} \ln^\beta\lt(2+2^m\rt).
\end{align*}
Thus, by the triangle inequality, we get
\begin{align}
\langl \lt(\sum_{m=m_0+3}^{+\infty} \lt( \fint_{\Boule(x_0,R)} \lt|\nabla \tilde{u}_m\rt|^2 \rt)^{\frac{1}{2}}\rt)^p \rangl^{\frac{1}{p}}
&\lesssim 2^{m_0} \sum_{m=m_0+3}^{+\infty} 2^{-m\nu} \ln^\beta\lt(2+2^{m}\rt)
\nonumber
\\
&\lesssim 2^{m_0(1-\nu)} \ln^\beta\lt(2+2^{m_0}\rt),
\label{ee77}
\end{align}
which shows the desired convergence. Hence, $u$ is a solution to \eqref{ee3}.

\paragraph{Step 2:}
As above, we set $R:=2^{m_0}$. We then obtain:
\begin{align*}
& \langl\lt(\fint_{\Boule(x_0,R)} \lt|\nabla \tilde{u}_m -\fint_{\Boule(x_0,R)}\nabla \tilde{u}_m \rt|^2\rt)^{\frac{p}{2}}\rangl^{\frac{1}{p}}\\
&\quad \leq \langl\lt( \fint_{\Boule(x_0,R)} \lt|\nabla \tilde{u}_m\rt|^2 \rt)^{\frac{p}{2}}\rangl^{\frac{1}{p}} \overset{\eqref{ee19}}{\lesssim} 2^{(m-m_0)d/2} 2^{m(1-\nu)} \ln^\beta\lt(2+2^m\rt).
\end{align*}
Whence, by the triangle inequality, we have:
\begin{align}
\nonumber
&\langl\lt(\fint_{\Boule(0,R)} \lt|\sum_{m=1}^{m_0+2} \lt(\nabla \tilde{u}_m -\fint_{\Boule(x_0,R)}\nabla \tilde{u}_m\rt) \rt|^2\rt)^{\frac{p}{2}}\rangl^{\frac{1}{p}}
\\
&\lesssim \sum_{m=1}^{m_0+2} 2^{(m-m_0)d/2} 2^{m(1-\nu)} \ln^\beta\lt(2+2^m\rt)
\lesssim 2^{m_0(1-\nu)} \ln^\beta\lt(2+2^{m_0}\rt).
\label{ee6}
\end{align}
Therefore, by a triangle inequality involving \eqref{ee77} and \eqref{ee6} (recall that $R=2^{m_0}$), we obtain \eqref{ee2}.
Hence
\begin{align*}
\sum_{m_0=1}^{+\infty} 2^{-m_0}\langl\lt( \fint_{\Boule\lt(x_0,2^{m_0}\rt)} \lt|\nabla u - \fint_{\Boule\lt(x_0,2^{m_0}\rt)} \nabla u\rt|^2  \rt)^{\frac{p}{2}}\rangl^{\frac{1}{p}}
< + \infty.
\end{align*}
Therefore, by the Markov inequality and the Borel-Cantelli Lemma, this shows that \eqref{ee10} is $\langle \cdot \rangle$-almost surely satisfied.

\paragraph{Step 3}
We finally replace the innermost integral of the left-hand side of \eqref{ee2} by an average over $\Boule(x_0,1)$.
Indeed, by the Cauchy-Schwarz inequality, there holds:
\begin{align*}
\lt|\fint_{\Boule\lt(x_0,2^{m-1}\rt)} \nabla u - \fint_{\Boule\lt(x_0,2^{m}\rt)} \nabla u\rt|
\leq &
\lt(\fint_{\Boule\lt(x_0,2^{m-1}\rt)} \lt|\nabla u - \fint_{\Boule\lt(x_0,2^{m}\rt)} \nabla u\rt|^2\rt)^{\frac{1}{2}}
\\
\leq & 2^{d/2} \lt(\fint_{\Boule\lt(x_0,2^{m}\rt)} \lt|\nabla u - \fint_{\Boule\lt(x_0,2^{m}\rt)} \nabla u\rt|^2\rt)^{\frac{1}{2}}.
\end{align*}
Hence, invoking \eqref{ee2} yields
\begin{align*}
\langl \lt|\fint_{\Boule\lt(x_0,1\rt)} \nabla u - \fint_{\Boule\lt(x_0,2^{m_0}\rt)} \nabla u\rt|^p \rangl^{\frac{1}{p}}
\lesssim \sum_{m=0}^{m_0} 2^{m(1-\nu)} \ln^{\beta}\lt(2+2^m\rt),
\end{align*}
which, by a triangle inequality involving once more \eqref{ee2}, produces \eqref{SpecialCase} for $R:=2^{m_0}$ (the general case $R \geq 1$ being a consequence of it).
\end{proof}

\section{Proof of Lemma \ref{LemRegUbar}}\label{Sec:ProofLemmas}

The proof is standard and relies on ideas from \cite[Prop. 2.1]{Vogelius_2000}:
\begin{proof}[Proof of Lemma \ref{LemRegUbar}] 
We adapt the classical argument to our situation with the interface:

\paragraph{Step 1: Caccioppoli estimate} We first notice that, for any $x_0 \in \R^d$ and $R>0$, if $\overline{u} \in \HH^1(\Boule (x_0, R))$ is $\overline{a}$-harmonic, then we have that 
\begin{align}
\label{Cacc}
\fint_{\Boule(x_0,R/2) } \lt| \nabla \overline{u} \rt|^2 \lesssim_{d, \lambda} \left( \frac{1}{R} \right)^{2} \fint_{\Boule(x_0,R)}  \lt| \overline{u} - \fint_{\Boule(x_0,R)} \overline{u} \rt|^2.
\end{align}

We include the argument of \eqref{Cacc} for completeness. 
Without loss of generality, we assume that $\overline{u}$ is of zero mean in $\Boule(x_0,R)$.
Let $\eta$ be a smooth cut-off of $\Boule(x_0, R/2)$ in $\Boule(x_0, R)$ such that $|\nabla \eta | \lesssim \frac{1}{R}$ and test the equation satisfied by $\overline{u}$ with $\eta^2 \overline{u}$. This yields 
\begin{align}
\int_{\Boule(x_0, R)} ( \eta^2 \nabla \overline{u}  + 2 \eta \nabla \eta \overline{u}) \cdot a \nabla \overline{u} = 0,
\end{align}
which, after an application of Young's inequality and using the properties of $\eta$ and $a$,  yields \eqref{Cacc}.

\paragraph{Step 2: Iteration of the Caccioppoli estimate} 
We now show that, for $x \in \Boule(x_0,1-\rho)$ and $n \in \mathbb{N}$, the following estimate holds for any multi-index $\beta \in \mathbb{N}^d$ such that $|\beta|=n\geq 1$:
\begin{align}
\label{lemma_1_8}
\int_{\Boule(x, \rho/2) \setminus \mathcal{I}} | \partial^\beta \overline{u} |^2 \lesssim_{d, \lambda,n} \left( \frac{1}{\rho} \right)^{2(n-1)} \int_{\Boule(x,\rho)}  | \nabla \overline{u} |^2.
\end{align}
Notice that combining \eqref{lemma_1_8} with Morrey's inequality and a covering argument gives that
\begin{align}
\label{lemma_1_4.5}
\displaystyle\sup_{x \in \Boule(x_0, 1 - \rho) \setminus \mathcal{I}} \rho^2 |\nabla^2 \overline{u}(x)|^2 + \displaystyle\sup_{x \in \Boule(x_0, 1 - \rho) \setminus \mathcal{I}}  |\nabla \overline{u}(x)|^2  \lesssim  \rho^{-d} \fint_{\Boule(x_0, 1) }|\nabla \overline{u}|^2 .
\end{align}
Since the terms on the left-hand side of \eqref{lemma_1_4.5} are treated in the same way, we concentrate on the $| \nabla^2 \overline{u}|^2$-term and notice that 
\begin{align*}
\sup_{x \in \Boule(x_0, 1 - \rho) \setminus \mathcal{I}} |\nabla^2 \overline{u}(x)|^2
& \lesssim \rho^{- (d + 2)} \int_{\Boule(x_0, 1)} |\nabla \overline{u}|^2.
\end{align*}
In the above line, we used the estimate
\begin{align*}
\sup_{y \in \Boule(x, \rho/2) \setminus \mathcal{I} }  |\nabla^2 \overline{u}(y)|^2 \lesssim_{d,\lambda} \rho^{-d- 2} \int_{\Boule(x,\rho)} |\nabla \overline{u}|^2,
\end{align*}
which may be obtained  first in the ball of radius $1$ by means of \eqref{lemma_1_8} (with $\rho:=1/2$), and then in the ball of radius $\rho$ by a rescaling argument.

Our argument for \eqref{lemma_1_8} is inductive: Notice that the case $n=1$ is trivial. 
Next, since $\overline{a}$-harmonicity is equivalent to $\overline{a}_{\pm}$-harmonicity to the left and right of the interface respectively complemented with transmission condition, we find that $\partial^{\alpha} \overline{u} \in \HH^1(\Boule(x, \rho))$ is still $\overline{a}$-harmonic for any multi-index such that $\alpha_1 = 0$. Applying \eqref{Cacc} to these $\partial^{\alpha} \overline{u}$, we inductively obtain that 
\begin{equation}
\label{lemma_1_5}
\begin{aligned}
\int_{\Boule(x, \rho/4)} | \nabla \partial^{\alpha} \overline{u} |^2  
\lesssim& \frac{1}{\rho^2} \int_{\Boule(x, \rho/2)} |  \partial^{\alpha}  \overline{u} |^2
\lesssim \left( \frac{1}{\rho} \right)^{2n} \int_{\Boule(x,\rho)}  | \nabla \overline{u} |^2.
\end{aligned}
\end{equation}
This observation establishes \eqref{lemma_1_8}, for any multi-index $\beta$ such that $\beta_1\leq 1$.

We still have to increase $\beta_1$.
For this we can use the equation satisfied by $\overline{u}$ in $\Boule(x, \rho)$ off of the interface.
This implies that the following relation holds in $\Boule(x, \rho) \setminus \mathcal{I}$
\begin{align}
\label{lemma_1_6}
\partial^\beta \partial^{n+2}_1  \overline{u} = \frac{-1}{\overline{a}_{11}} \displaystyle\sum_{(i, j) \neq (1 , 1) }  \overline{a}_{i j} \partial^\beta \partial^{n}_1\partial_{i} \partial_{j} \overline{u}.
\end{align}
Using once more an inductive argument, we finally obtain \eqref{lemma_1_8} for any multi-index $\beta \in \mathbb{N}^d$.

\paragraph{Step 3: Harmonic coordinates} Going from \eqref{lemma_1_4.5} to the desired \eqref{BorneNabla2} is a simple matter of using \eqref{DefUUU}. In particular, we use that, due to \eqref{DefPj}, $(\nabla P)^{-1}$ is bounded off of the interface and, since $\overline{\nabla} \overline{u}$ satisfies the transmission condition through the interface, it is continuous.
\end{proof} 


\begin{thebibliography}{10}

\bibitem{Allaire}
G.~Allaire.
\newblock {\em Shape optimization by the homogenization method}, volume 146 of
  {\em Applied Mathematical Sciences}.
\newblock Springer-Verlag, New York, 2002.

\bibitem{Armstrong_book_2018}
S.~Armstrong, T.~Kuusi, and J.-C. Mourrat.
\newblock {\em Quantitative stochastic homogenization and large-scale
  regularity}, volume 352 of {\em Grundlehren der Mathematischen Wissenschaften
  [Fundamental Principles of Mathematical Sciences]}.
\newblock Springer, Cham, 2019.

\bibitem{Armstrong_Shen_2016}
S.~Armstrong and Z.~Shen.
\newblock Lipschitz estimates in almost-periodic homogenization.
\newblock {\em Comm. Pure Appl. Math.}, 69(10):1882--1923, 2016.

\bibitem{AvellanedaLin}
M.~Avellaneda and F.-H. Lin.
\newblock Compactness methods in the theory of homogenization.
\newblock {\em Comm. Pure Appl. Math.}, 40(6):803--847, 1987.

\bibitem{Bakhvalov_Panasenko}
N.~Bakhvalov and G.~Panasenko.
\newblock {\em Homogenisation: averaging processes in periodic media},
  volume~36 of {\em Mathematics and its Applications (Soviet Series)}.
\newblock Kluwer Academic Publishers Group, Dordrecht, 1989.

\bibitem{BellaGiunti_2018}
P.~Bella and A.~Giunti.
\newblock Green's function for elliptic systems: moment bounds.
\newblock {\em Netw. Heterog. Media}, 13(1):155--176, 2018.

\bibitem{Article_Homog}
X.~Blanc, M.~Josien, and C.~Le~Bris.
\newblock Precised approximations in elliptic homogenization beyond the
  periodic setting.
\newblock {\em Asymptotic Analysis}.
\newblock Under press.

\bibitem{CRAS_Homog}
X.~Blanc, M.~Josien, and C.~Le~Bris.
\newblock Approximation locale pr\'ecis\'ee dans des probl\`emes
  multi-\'echelles avec d\'efauts localis\'es.
\newblock {\em Comptes Rendus Math\'ematique de l'Acad\'emie des Sciences},
  357:167--174, 2019.

\bibitem{BLLcpde}
X.~Blanc, C.~Le~Bris, and P.-L. Lions.
\newblock Local profiles for elliptic problems at different scales: defects in,
  and interfaces between periodic structures.
\newblock {\em Comm. Partial Differential Equations}, 40(12):2173--2236, 2015.

\bibitem{NewBLL1}
X.~Blanc, C.~Le~Bris, and P.-L. Lions.
\newblock On correctors for linear elliptic homogenization in the presence of
  local defects.
\newblock {\em Comm. Partial Differential Equations}, 43(6):965--997, 2018.

\bibitem{Fischer_Otto_2015}
J.~Fischer and F.~Otto.
\newblock A higher-order large-scale regularity theory for random elliptic
  operators.
\newblock {\em Comm. Partial Differential Equations}, 41(7):1108--1148, 2016.

\bibitem{Fischer_Raithel_2017}
J.~Fischer and C.~Raithel.
\newblock Liouville principles and a large-scale regularity theory for random
  elliptic operators on the half-space.
\newblock {\em SIAM J. Math. Anal.}, 49(1):82--114, 2017.

\bibitem{GT}
D.~Gilbarg and N.~Trudinger.
\newblock {\em Elliptic partial differential equations of second order}.
\newblock Classics in Mathematics. Springer-Verlag, Berlin, 2001.

\bibitem{Gloria_Notes}
A.~Gloria.
\newblock A quantitative theory of stochastic homogenization: regularity,
  oscillations, and fluctuations.
\newblock Lecture notes.

\bibitem{Gloria_Neukamm_Otto_2015}
A.~Gloria, S.~Neukamm, and F.~Otto.
\newblock A regularity theory for random elliptic operators.
\newblock {\em ArXiv e-print arXiv:1409.2678v3}, 2015.

\bibitem{JKO}
V.~Jikov, S.~Kozlov, and O.~Ole{\u\i}nik.
\newblock {\em Homogenization of differential operators and integral
  functionals}.
\newblock Springer-Verlag, Berlin, 1994.

\bibitem{Josien_InterfPer_2018}
M.~Josien.
\newblock Some quantitative homogenization results in a simple case of
  interface.
\newblock {\em Communications in Partial Differential Equations},
  44(10):907--939, 2019.

\bibitem{JosienOtto_2019}
M.~Josien and F.~Otto.
\newblock Introduction to stochastic homogenization: oscillations and
  fluctuations.
\newblock In preparation.

\bibitem{KLSGreenNeumann}
C.~Kenig, F.-H. Lin, and Z.~Shen.
\newblock Periodic homogenization of {G}reen and {N}eumann functions.
\newblock {\em Comm. Pure Appl. Math.}, 67(8):1219--1262, 2014.

\bibitem{LiNirenberg_2003}
Y.~Y. Li and L.~Nirenberg.
\newblock Estimates for elliptic systems from composite material.
\newblock {\em Comm. Pure Appl. Math.}, 56(7):892--925, 2003.

\bibitem{Vogelius_2000}
Y.~Y. Li and M.~Vogelius.
\newblock Gradient estimates for solutions to divergence form elliptic
  equations with discontinuous coefficients.
\newblock {\em Arch. Ration. Mech. Anal.}, 153(2):91--151, 2000.

\bibitem{Lorenzi_1972}
A.~Lorenzi.
\newblock On elliptic equations with piecewise constant coefficients. {II}.
\newblock {\em Ann. Scuola Norm. Sup. Pisa (3)}, 26:839--870, 1972.

\bibitem{Necas_1965}
J.~Ne\v{c}as.
\newblock Sur une m\'{e}thode pour r\'{e}soudre les \'{e}quations aux
  d\'{e}riv\'{e}es partielles du type elliptique, voisine de la variationnelle.
\newblock {\em Ann. Scuola Norm. Sup. Pisa (3)}, 16:305--326, 1962.

\bibitem{Raithel_2017}
C.~Raithel.
\newblock A large-scale regularity theory for random elliptic operators on the
  half-space with homogeneous neumann boundary data.
\newblock {\em arXiv preprint arXiv:1703.04328}, 2017.

\bibitem{Shen}
Z.~Shen.
\newblock {\em Periodic homogenization of elliptic systems}, volume 269 of {\em
  Operator Theory: Advances and Applications}.
\newblock Birkh\"{a}user/Springer, Cham, 2018.
\newblock Advances in Partial Differential Equations (Basel).

\bibitem{Vinoles}
Valentin Vinoles.
\newblock {\em Probl\`{e}mes d'interface en pr\'{e}sence de
  m\'{e}tamat\'{e}riaux: mod\'{e}lisation, analyse, et simulations}.
\newblock PhD thesis, Universit{\'e} Paris-Saclay, 2017.

\end{thebibliography}
\end{document}